\RequirePackage{fix-cm}
\documentclass[11pt]{imsart}
\usepackage[T1]{fontenc}
\usepackage[utf8]{inputenc}

\usepackage[margin=1in]{geometry}
\usepackage{microtype}

\RequirePackage{amsthm,amsmath,amsfonts,amssymb}
\RequirePackage[numbers]{natbib}
\RequirePackage[colorlinks,citecolor=blue,urlcolor=blue]{hyperref}

\usepackage{comment} 
\usepackage{graphicx}

\usepackage{subcaption}
\usepackage{cleveref}
\crefname{figure}{Figure}{Figure}

\usepackage{booktabs}
\usepackage{calc}

\usepackage{algorithm}
\usepackage{algpseudocode}

\providecommand{\tightlist}{%
  \setlength{\itemsep}{0pt}\setlength{\parskip}{0pt}}

\allowdisplaybreaks 

\startlocaldefs
\theoremstyle{plain}

\newtheorem{theorem}{Theorem}[section]
\newtheorem{lemma}[theorem]{Lemma}
\newtheorem{proposition}[theorem]{Proposition}
\newtheorem{corollary}[theorem]{Corollary}
\theoremstyle{definition}
\newtheorem{definition}[theorem]{Definition}
\newtheorem{assumption}{Assumption}
\newtheorem*{remark}{Remark}
\newtheorem{example}{Example}
\newtheorem*{fact}{Fact}
\newcounter{fact}[section]
\newenvironment{property}[1][]{\refstepcounter{fact}\par
   \noindent \textbf{P\thefact. #1} \rmfamily}{\smallskip}
\labelformat{fact}{P#1}

\usepackage{amsmath}
\DeclareMathOperator*{\argmax}{arg\,max}

\newcommand{\DD}{{\mathcal D}}
\newcommand{\dd}{{\operatorname d}}
\newcommand{\F}{{\mathcal F}}

\newcommand{\PP}{{\mathbb P}}
\newcommand{\QQ}{{\mathbb Q}}
\newcommand{\R}{{\mathbb R}}
\newcommand{\bSS}{{\mathbb S}}

\newcommand{\vol}{{\operatorname{vol}}}
\newcommand{\Sec}{{\operatorname{Sec}}}

\newcommand{\grad}{{\operatorname{grad}}}

\newcommand{\hess}{{\mathcal{H}}}
\newcommand{\supp}{{\operatorname{Supp}}}
\newcommand{\tang}{{\operatorname{tang}}}

\newcommand{\tv}{{\operatorname{TV}}}

\endlocaldefs

\begin{document}

\begin{frontmatter}
\title{Horospherical Depth and Busemann Median on Hadamard Manifolds}
\runtitle{Horospherical Depth on Hadamard Manifolds}

\begin{aug}
\author{\fnms{Yangdi}~\snm{Jiang}\ead[label=e1]{yangdi.jiang@ntu.edu.sg}},
\author{\fnms{Xiaotian}~\snm{Chang}\ead[label=e2]{xiaotian.chang@ntu.edu.sg}}
\and
\author{\fnms{Cyrus}~\snm{Mostajeran}\ead[label=e3]{cyrussam.mostajeran@ntu.edu.sg}}
\address{Division of Mathematical Sciences,
Nanyang Technological University
\printead[presep={,\ }]{e1,e2,e3}}

\end{aug}

\begin{abstract}
We introduce the \emph{horospherical depth}, an intrinsic notion of statistical depth on Hadamard manifolds, and define the \emph{Busemann median} as the set of its maximizers. The construction exploits the fact that the linear functionals appearing in Tukey’s half-space depth are themselves limits of renormalized distance functions; on a Hadamard manifold the same limiting procedure produces Busemann functions, whose sublevel sets are horoballs, the intrinsic replacements for halfspaces. The resulting depth is parametrized by the visual boundary, is isometry-equivariant, and requires neither tangent-space linearization nor a chosen base point. For arbitrary Hadamard manifolds, we prove that the depth regions are nested and geodesically convex, that a centerpoint of depth at least $1/(d+1)$ exists, and hence that the Busemann median exists for every Borel probability measure. Under strictly negative sectional curvature and mild regularity assumptions, the depth is strictly quasi-concave and the median is unique. We also establish robustness: the depth is stable under total-variation perturbations, and under contamination escaping to infinity the limiting median depends on the escape direction but not on how far the contaminating mass has moved along the geodesic ray, in contrast with the Fr\'echet mean. Finally, we establish uniform consistency of the sample depth and convergence of sample depth regions and sample Busemann medians; on symmetric spaces of noncompact type, the argument proceeds through a VC analysis of upper horospherical halfspaces, while on general Hadamard manifolds it follows from a compactness argument under a mild non-atomicity assumption.
\end{abstract} 



\end{frontmatter}
\tableofcontents


\section{Introduction}
\label{sec:intro}

\subsection{A Euclidean benchmark}
Tukey's halfspace depth, going back to Tukey's center-outward view of multivariate data, is one of the benchmark notions in multivariate nonparametric statistics \cite{tukey1975mathematics, donoho1992breakdown, rousseeuw1999depth, zuo2000general}. It simultaneously supplies a multivariate rank, a family of depth contours, and a robust location estimator through the Tukey median. Because it packages ordering, geometry, and robustness in a single construction, it remains a natural standard against which newer depth notions are compared \cite{zuo2000general, zuo2000structural}. For a Borel probability measure $\PP$ on $\R^d$ and a point $z \in \R^d$, it is defined by
\begin{equation}\label{eqn:tukey_depth}
D^T(z;\PP)
= \inf_{u \in S^{d-1}} \PP\bigl(\{x \in \R^d : \langle u,x\rangle \le \langle u,z\rangle\}\bigr).
\end{equation}
Points maximizing $D^T(\cdot;\PP)$ are Tukey medians \cite{donoho1992breakdown}. The force of \eqref{eqn:tukey_depth} lies not only in providing a center-outward ordering of $\R^d$, but in the geometric structure built into the definition. The depth regions are convex and nested \cite{donoho1992breakdown}; a point of depth at least $1/(d+1)$ always exists \cite{rado1946theorem, rousseeuw1999depth}; and the deepest points enjoy strong robustness properties \cite{donoho1992breakdown, chen2002influence}. The axiomatic framework of \cite{zuo2000general} makes the same point from another angle: affine invariance, maximality at a center, monotonicity away from a deepest point, and vanishing at infinity all hold because the relevant sets in \eqref{eqn:tukey_depth} are halfspaces, hence convex. Any intrinsic extension to curved spaces must therefore identify the right substitute for halfspaces, not only the right substitute for the median.

\subsection{The manifold problem and related work}

Many modern datasets are not vectors in $\R^d$ but points on manifolds or other non-Euclidean spaces. In such settings one still wants a notion of location, a center-outward ordering, and robust summaries that respect the intrinsic geometry of the sample space. This perspective underlies a substantial literature on manifold-valued statistics, where Fr\'echet means, intrinsic and extrinsic means, and related $L^p$-type centers serve as basic location summaries \cite{frechet1948elements, bhattacharya2003large, pennec2006intrinsic, afsari2011riemannian, bacak2014computing}. Representative applications include diffusion tensor imaging  \cite{fletcher2007riemannian}, hyperbolic representation learning for hierarchical data \cite{nickel2017poincare}, and structured object spaces such as phylogenetic tree space \cite{billera2001geometry}. Recent work on robust Fr\'echet-type procedures underscores the same need from a different angle \cite{lee2025huber}.

Hadamard manifolds are a natural setting in which to seek depth-based analogues of Euclidean location. Their nonpositive curvature guarantees unique geodesics between points and preserves a useful notion of convexity, so optimization-based estimators are globally well behaved \cite{bridson2013metric, afsari2011riemannian, bacak2014computing}. At the same time, the dominant statistical tools on such spaces remain distance-based. Fr\'echet-type estimators summarize data through minimization of $d(x,\cdot)$ or $d^2(x,\cdot)$, whereas a depth function would provide an order-based notion of centrality together with contours, centerpoint guarantees, and outlyingness diagnostics. This distinction matters especially when robustness is a primary concern, because contamination can interact with distance in ways that are qualitatively different from order-based depth.

Existing notions of depth recover only part of the Tukey theory. Lens depth \cite{liu2011lens}, and its extensions to general metric spaces \cite{cholaquidis2023weighted, cholaquidis2023level}, is built from metric balls and preserves a useful notion of centrality, but it does not retain the halfspace geometry that drives convex depth regions and centerpoint results. Metric halfspace depth \cite{dai2023tukey} is closer in spirit: it replaces a Euclidean halfspace by a metric comparison set of the form $\{x : d(x,z) \le d(x,y)\}$. This gives a broadly applicable and isometry-invariant construction, but its very generality comes at a price: because it depends only on the metric, the resulting regions need not reflect the geodesic convexity of the ambient manifold, and the classical $1/(d+1)$ centerpoint bound is not available in that framework. Tangent-space depth \cite{rusciano2018riemannian} restores a centerpoint theorem by projecting to tangent spaces, but it depends on a chosen base point and is therefore not fully intrinsic. What remains missing is a notion of halfspace depth on Hadamard manifolds that is simultaneously intrinsic, isometry-equivariant, and strong enough to recover the structural features that make Tukey depth effective.

\subsection{Horospheres as intrinsic halfspaces}
Our starting point is that the linear functionals in \eqref{eqn:tukey_depth} already admit a metric interpretation. Fix $u \in S^{d-1}$, where $S^{d-1}$ denotes the $(d-1)$-dimensional unit sphere in $\R^d$. Then
\[
\|x-tu\|-t \to -\langle u,x\rangle \qquad \text{as } t\to\infty.
\]
Thus a Euclidean halfspace can be read as a super-level set of a limit of renormalized distance functions. On a Hadamard manifold, the same limiting procedure along a geodesic ray produces a Busemann function, introduced by \cite{busemann1955geometry} and standard in the geometry of nonpositive curvature \cite{bridson2013metric}. Its level sets are horospheres, its sublevel sets are horoballs, and horoballs are geodesically convex \cite{heintze1977geometry}. The relevant directions are parametrized by the visual boundary $\partial X$ \cite{eberlein1973visibility}. These objects are not ad hoc analogues of Euclidean notions; they arise from the same asymptotic distance construction that in flat space collapses to linear functionals.

This suggests the definition 
\[
D(z;\PP)=\inf_{\xi\in\partial X} \PP\big(\{x\in X:B_\xi(x)\ge B_\xi(z)\}\big),
\] 
which we call the \emph{horospherical depth}. Its maximizers form the \emph{Busemann median}. In Euclidean space this reduces to Tukey's halfspace depth; on a Hadamard manifold it replaces Euclidean directions by points at infinity and halfspaces by the sides of horospheres, while the resulting depth regions are intersections of horoballs. The construction is intrinsic: changing the base point in the definition of $B_\xi$ changes the function only by an additive constant depending on $\xi$ and the chosen base point, so the comparison sets $\{x:B_\xi(x)\ge B_\xi(z)\}$ are unchanged. The main question of the paper is whether the theory follows the definition: once halfspaces are replaced by horoballs, do the central structural and statistical properties of Tukey depth survive the passage to nonpositive curvature?

\subsection{Main contributions and scope}

The results of the paper come in three layers.

First, for arbitrary Hadamard manifolds, we define horospherical depth and show that it satisfies the core structural properties expected of a Tukey-type depth: isometric equivariance, maximality at natural symmetry centers, vanishing at infinity, and geodesically convex nested depth regions represented as intersections of horoballs (Theorem \ref{lemma:superlevel} and Section \ref{sec:properties}). We also prove a centerpoint theorem: every Borel probability measure on a $d$-dimensional Hadamard manifold has a point of depth of at least $1/(d+1)$ (Theorem \ref{thm:depth_lbound}). Hence the Busemann median exists for every Borel probability measure.

Second, under strictly negative sectional curvature and mild regularity assumptions on the one-dimensional Busemann projections, the theory becomes sharper. The depth is strictly quasi-concave, so the Busemann median is unique (Theorem \ref{thm:unique_rank_1}). We also prove robustness in total variation and analyze contamination escaping to infinity. In that regime, the contaminated depths converge pointwise to an explicit limit determined by the escape direction; when the limiting depth has a unique maximizer above level $\varepsilon$, the contaminated medians converge
accordingly, whereas distance-based estimators such as the Fr\'echet mean are pulled toward the boundary (Theorems \ref{thm:robust} and \ref{thm:boundary_robust}).

Third, on the statistical side, we prove uniform almost sure convergence of the sample horospherical depth and derive convergence of sample depth regions and, in the singleton-median regime, of sample Busemann medians (Theorems \ref{thm:sample_converg}--\ref{thm:sample_converg_depth_region}). On symmetric spaces of noncompact type, the argument proceeds through a VC analysis of the class $\{H^+_{\xi,z}:\xi\in\partial X,\ z\in X\}$; on general Hadamard manifolds, the same conclusion holds under a mild non-atomicity assumption via compactness of the visual boundary and continuity of the Busemann functions. Section \ref{sec:numerical} then gives explicit formulas and finite-direction computational illustrations in hyperbolic space $\mathbb H^d$ and the manifold of $p\times p$ symmetric positive-definite matrices $\bSS_p^+$ equipped with the affine-invariant metric; these are intended to make the construction concrete rather than to provide a final algorithmic theory.



Our construction occupies a middle ground between existing approaches. Relative to metric halfspace depth \cite{dai2023tukey}, it is less general in the underlying space but yields stronger geometric consequences. Relative to tangent-space depth \cite{rusciano2018riemannian}, it avoids any chosen origin in a tangent space. The definition and the convex-region/existence theory apply to all Hadamard manifolds. The sharper uniqueness results require strict negative curvature, so they apply directly to rank-one examples such as $\mathbb{H}^d$. Higher-rank spaces such as $\bSS_p^+$ remain within the scope of the definition and the centerpoint theory, but uniqueness there requires additional ideas. Possible extensions beyond the smooth setting, for example to singular CAT$(0)$ spaces such as phylogenetic tree space \cite{billera2001geometry}, are discussed later as open problems.

Because the geometric hypotheses vary across sections, we summarize the scope of each main result in Table \ref{table:result}.  Throughout, $X$ denotes a $d$-dimensional Hadamard manifold and $\PP$ a Borel probability measure on $X$.

\begin{table}
\caption{Scope of the main results.}
\begin{center}
\small
\resizebox{\columnwidth}{!}{%
\begin{tabular}{l l l}
\toprule
\textbf{Result} & \textbf{Geometry required} & \textbf{Measure conditions}\\
\midrule
Depth-region convexity (Thm. \ref{lemma:superlevel}) & Hadamard & None\\
Isometry equivariance (Thm. \ref{lemma:g-depth}) & Hadamard & None\\
Maximality at center (Prop. \ref{lemma:maximality}) & Hadamard & Horospherical symmetry\\
Centerpoint theorem (Thm. \ref{thm:depth_lbound}) & Hadamard & None\\
Strict quasi-concavity (Thm. \ref{thm:unique_rank_1}) & $\mathrm{Sec}_X < 0$ & Assump. \ref{assump:no_atom}-\ref{assump_connected_supp}\\
TV robustness (Thm. \ref{thm:robust}) & Hadamard & None\\
Boundary robustness (Thm. \ref{thm:boundary_robust}) & $\mathrm{Sec}_X < 0$ & Assump. \ref{assump:no_atom}\\
VC consistency (Thm. \ref{thm:sample_converg}) & Symmetric space, noncompact & None\\
General consistency (Thm. \ref{thm:sample_converg_2}) & Hadamard & Assump. \ref{assump:no_atom}\\
\bottomrule
\end{tabular}
}
\end{center}
\label{table:result}
\end{table}

\subsection{Organization}
Section \ref{sec:background} reviews Busemann functions, horospheres, and the visual boundary. Section \ref{sec:definition} introduces horospherical depth and studies the geometry of its depth regions, including the choice between one-sided and two-sided formulations. Section \ref{sec:properties} establishes isometric equivariance, symmetry-based maximality, and basic properties of the depth regions. Section \ref{sec:existence} proves vanishing at infinity, existence of the Busemann median, and the centerpoint theorem. Section \ref{sec:uniqueness} treats strict quasi-concavity and uniqueness under negative curvature. Section \ref{sec:robust} studies robustness under total variation perturbations and contamination escaping to the boundary. Section \ref{sec:computation} proves empirical convergence of the sample depth and depth regions. Section \ref{sec:numerical} presents examples and computational constructions in $\mathbb{H}^d$ and $\bSS_p^+$. Section \ref{sec:discussion} concludes with comparisons, limitations, and open problems. All the proofs are collected in the supplement.

\section{Busemann functions, horospheres, and the visual boundary}
\label{sec:background}

The introduction identified Busemann functions as the correct replacements for linear functionals on a Hadamard manifold. This section fixes notation and records the geometric facts that drive the rest of the paper. A useful heuristic is that if a point escapes along a unit-speed geodesic ray $\gamma$, then in a $\operatorname{CAT}(0)$ space one has $d(x,\gamma(t))-t\to B_\gamma(x)$ and $d^2(x,\gamma(t))=t^2+2tB_\gamma(x)+o(t)$. The same asymptotic object therefore governs both directional location and the effect of contamination escaping to infinity.

\subsection{Visual boundary and asymptotic rays}
\label{sec:boundary}

Let $X$ be a metric space. Two geodesic rays $\gamma,\gamma':[0,\infty)\to X$ are called \emph{asymptotic} if $d(\gamma(t),\gamma'(t))\le K$ for all $t\ge 0$ for some constant $K<\infty$. This is an equivalence relation on the set of rays. The equivalence class of $\gamma$ is denoted by $\gamma(\infty)$, and the set of all such classes is denoted by $\partial X$. We call $\partial X$ the \emph{visual boundary} of $X$.

When $X$ is proper, $\partial X$ is compact in the usual cone topology. In the settings considered here, a geodesic ray is uniquely determined by its initial point and its endpoint at infinity: if $\gamma(0)=\gamma'(0)$ and $\gamma(\infty)=\gamma'(\infty)$, then $\gamma=\gamma'$. If $X$ is a complete Riemannian manifold, every geodesic ray $\gamma:[0,\infty)\to X$ extends uniquely to a complete geodesic $\widetilde\gamma:\R\to X$ by Hopf--Rinow \cite[Chapter I.3, Proposition 3.7]{bridson2013metric}. We write $\gamma^{-}(t):=\widetilde\gamma(-t)$ for the opposite ray and, when convenient, write $\gamma(-\infty)$ for the class of $\gamma^-$. 

\subsection{Busemann functions and horospheres}
\label{sec:busemann_def}

Given a geodesic ray $\gamma:[0,\infty)\to X$, its \emph{Busemann function} is
\begin{equation}
\label{eqn:busemann_def}
B_\gamma(x):=\lim_{t\to\infty}\bigl(d(\gamma(t),x)-t\bigr),
\end{equation}
whenever the limit exists. On every $\operatorname{CAT}(0)$ space, and hence on every Hadamard manifold, the limit exists for all $x\in X$ \cite[Lemma 8.18]{bridson2013metric}. If $\widetilde\gamma$ denotes the bi-infinite extension of $\gamma$, then $B_\gamma(\widetilde\gamma(s))=-s$ for every $s\in\R$; in particular $B_\gamma(\gamma(0))=0$.

The next proposition records the dependence of the Busemann function on the choice of base point along a fixed boundary direction.

\begin{proposition}
\label{prop:busemann_asymptotic}
Let $X$ be a Hadamard manifold, and let $\gamma,\gamma':[0,\infty)\to X$ be geodesic rays with the same endpoint at infinity, say $\gamma(\infty)=\gamma'(\infty)=\xi$. Then for every $x\in X$, 
\[
B_\gamma(x)=B_{\gamma'}(x)-B_{\gamma'}(\gamma(0)).
\] 
Equivalently, $B_\gamma-B_{\gamma'}$ is constant, with value $-B_{\gamma'}(\gamma(0))=B_\gamma(\gamma'(0))$. In particular, if $\gamma(0)=\gamma'(0)$, then $B_\gamma\equiv B_{\gamma'}$.
\end{proposition}

Accordingly, once a boundary point $\xi\in\partial X$ and a base point $x_0\in X$ are fixed, we may write $B_{\xi,x_0}$ for the Busemann function of the unique ray starting at $x_0$ and ending at $\xi$. After fixing an origin $o\in X$, we abbreviate $B_{\xi,o}$ to $B_\xi$. Then for any other base point $p\in X$, 
\begin{equation}\label{eqn:busemann_relation}
B_{\xi,p}(x)=B_\xi(x)-B_\xi(p).
\end{equation}
The sublevel sets $H_\xi(t):=\{x\in X:B_\xi(x)\le t\}$ are called \emph{horoballs}, and their level sets $\partial H_\xi(t):=\{x\in X:B_\xi(x)=t\}$ are the corresponding \emph{horospheres}.

\begin{example}[Busemann function on $\R^d$]
\label{example:euclidean}
If $X=\R^d$ and $\gamma(t)=tu$ for a unit vector $u$, then $d(x,\gamma(t))=\|tu-x\|=t-\langle x,u\rangle+o(1)$ as $t\to\infty$. Hence $B_\gamma(x)=-\langle x,u\rangle$, so Busemann functions reduce to affine functionals and horospheres reduce to hyperplanes.
\end{example}

\begin{example}[Busemann function on $\mathbb{H}^d$]
\label{example:busemann_hyper}
Under the Poincar\'e ball model of $\mathbb{H}^d$, the visual boundary is identified with $S^{d-1}$. If $\gamma$ is the ray from the origin to $\xi\in S^{d-1}$, then
\[
B_\xi(x)=\log\bigl(\|x-\xi\|^2\bigr)-\log\bigl(1-\|x\|^2\bigr).
\]
The horosphere $\partial H_\xi(c)$ is therefore a Euclidean sphere internally tangent to the boundary sphere at $\xi$ \cite{heintze1977geometry}; see Figure \ref{fig:horosphere_busemann}.
\end{example}

\begin{figure}
    \centering
    \begin{subfigure}{0.35\columnwidth}
        \includegraphics[width=\columnwidth]{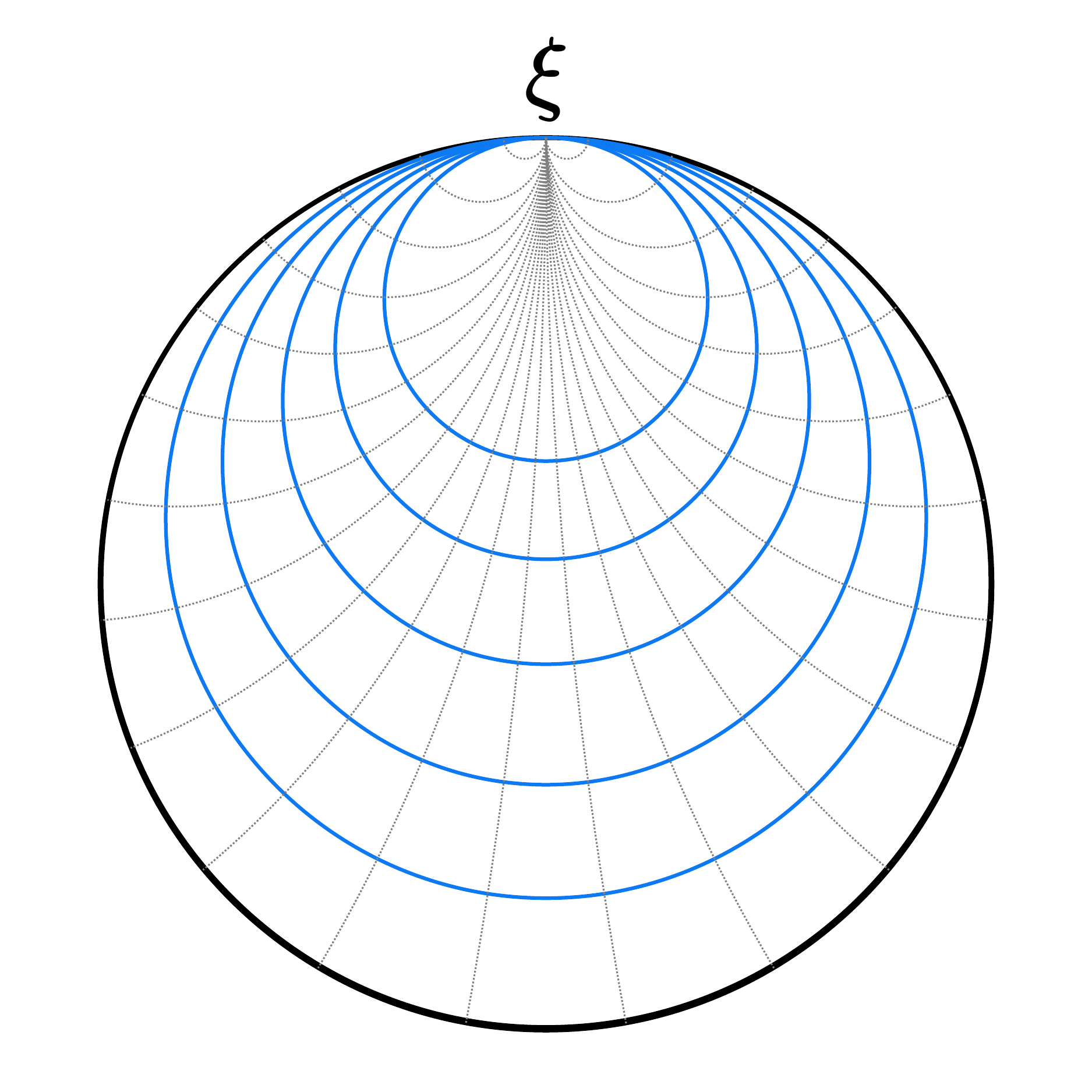}
          \caption{}
          \label{fig:horospheres}
      \end{subfigure}
    \hspace{0.15\columnwidth}
    \begin{subfigure}{0.35\columnwidth}
        \includegraphics[width=\columnwidth]{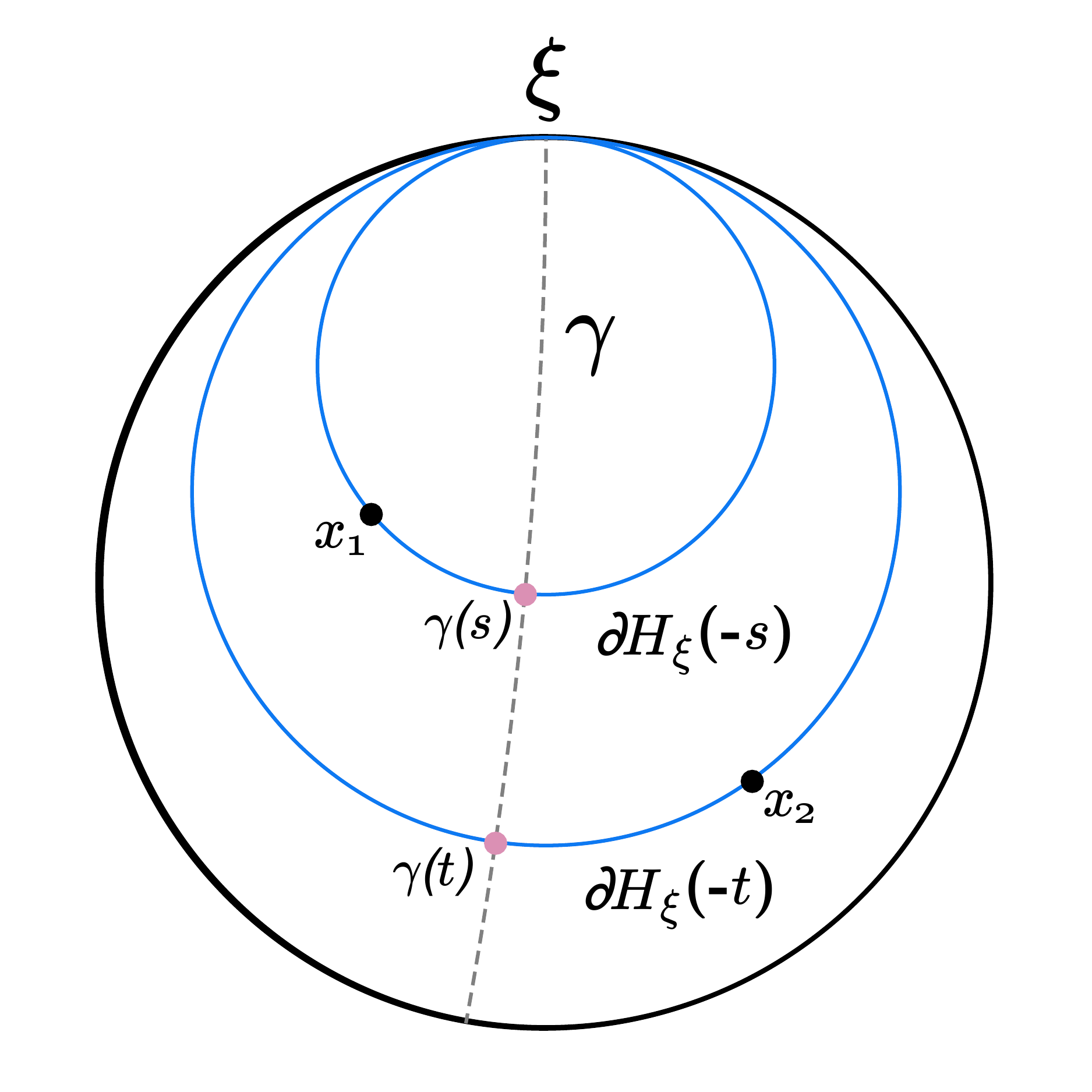}
          \caption{}
          \label{fig:busemann}
      \end{subfigure}
    \caption{Horospheres and the Busemann function in the Poincar\'e ball model of hyperbolic geometry. Panel \textup{(a)} shows horospheres $\partial H_\xi(t)$ with common center $\xi$ and several geodesics asymptotic to $\xi$. Larger values of $t$ correspond to larger horoballs. Panel \textup{(b)} illustrates the horospherical-coordinate interpretation of $B_\xi$: the horosphere through a point $x$ meets the reference ray to $\xi$ at a unique parameter value, and the Busemann function records that level.}
    \label{fig:horosphere_busemann}
\end{figure}

Figure \ref{fig:busemann} is worth reading as a coordinate picture. Fix an origin $o$ and let $\gamma$ be the geodesic ray from $o$ to $\xi$. For each $x\in X$, the horosphere centered at $\xi$ that passes through $x$ intersects $\gamma$ at a unique point $\gamma(s)$, and then $B_\xi(x)=B_\xi(\gamma(s))=-s$. Thus $B_\xi$ is not a distance to $\xi$; rather, it is a horospherical coordinate whose level sets foliate $X$ by hypersurfaces orthogonal to the radial flow toward $\xi$.

We now record the basic regularity properties that will be used repeatedly.

\begin{proposition}
\label{prop:cont}
For every $\xi\in\partial X$, the Busemann function $B_\xi$ is geodesically convex and $1$-Lipschitz \cite[Proposition 8.22]{bridson2013metric}. Consequently every horoball $H_\xi(t)$ is geodesically convex. If $X$ is a Hadamard manifold and $\gamma:[0,\infty)\to X$ is the geodesic ray with $\gamma(0)=p$ and $\gamma(\infty)=\xi$, then $\grad_p B_\xi=-\dot\gamma(0)$; in particular $\|\grad B_\xi\|\equiv 1$. Furthermore, $B_\xi$ is $C^2$. \cite[Proposition 3.1]{heintze1977geometry}.
\end{proposition}

The gradient identity has a simple geometric meaning: horospheres centered at $\xi$ are everywhere orthogonal to the geodesic rays asymptotic to $\xi$, exactly as suggested by Figure \ref{fig:horospheres}.

The facts recorded in this section enter the statistical theory in three distinct ways. Geodesic convexity of horoballs is the structural input behind the representation and convexity of depth regions in Section \ref{sec:definition}. Under strictly negative curvature, the strict form of this convexity becomes the key geometric ingredient in the uniqueness theory of Section \ref{sec:uniqueness}. Finally, the asymptotic interpretation of Busemann functions along escaping rays explains why they are the correct objects for studying contamination that moves toward the visual boundary, which is the focus of Section \ref{sec:robust}.

\section{Horospherical depth and depth regions}
\label{sec:definition}

We now use the geometric objects of Section \ref{sec:background} to build a depth function. In one dimension, the median can be characterized as the maximizer of the depth $D^1(z;\PP)=\min\{\PP([z,\infty)),\PP(( -\infty,z])\}$. On a Hadamard manifold, the role of a half-line through $z$ is played by the side of a horosphere through $z$.

\subsection{Definition and superlevel representation}

Throughout this section, $X$ is a Hadamard manifold. For $\xi\in\partial X$ and $z\in X$, define the two horospherical halfspaces through $z$ by
\[
H^+_{\xi,z}:=\{x\in X:B_\xi(x)\ge B_\xi(z)\},
\qquad
H^-_{\xi,z}:=\{x\in X:B_\xi(x)\le B_\xi(z)\}.
\]
These sets are intrinsic. Indeed, if $B_{\xi,p}$ is defined using a base point $p$ instead of the fixed origin $o$, then Proposition \ref{prop:busemann_asymptotic} gives $B_{\xi,p}(x)=B_\xi(x)-B_\xi(p)$, so the inequality $B_{\xi,p}(x)\ge B_{\xi,p}(z)$ is equivalent to $B_\xi(x)\ge B_\xi(z)$.

\begin{definition}
\label{def:depth}
Given a probability measure $\PP$ on $X$ and a point $z\in X$, the \emph{horospherical depth} of $z$ with respect to $\PP$ is
\begin{equation}
\label{eq:depth_one_sided}
D(z;\PP):=\inf_{\xi\in\partial X}\PP(H^+_{\xi,z}).
\end{equation}
For $\alpha\ge 0$, the \emph{$\alpha$-depth region} is $\DD^\alpha(\PP):=\{z\in X:D(z;\PP)\ge \alpha\}$, and the corresponding \emph{$\alpha$-depth contour} is $\partial \DD^\alpha(\PP):=\{z\in X:D(z;\PP)=\alpha\}$.
\end{definition}


The main structural fact is that these superlevel sets are intersections of horoballs. For $\xi\in\partial X$, write $S_\xi(t):=\PP(\{x\in X:B_\xi(x)\ge t\})$ and define the upper survival quantile by
\begin{equation}
\label{eq:t_xi}
t_\xi(\alpha):=\sup\{t\in\R:S_\xi(t)\ge \alpha\}, \qquad \alpha\in(0,1).
\end{equation}

\begin{theorem}
\label{lemma:superlevel}
For every $\alpha\in(0,1)$,
\begin{equation}
\label{eq:superlevel}
\DD^\alpha(\PP)=\bigcap_{\xi\in\partial X} H_\xi\bigl(t_\xi(\alpha)\bigr)=\bigcap_{\xi\in\partial X}\{z\in X:B_\xi(z)\le t_\xi(\alpha)\}.
\end{equation}
Consequently, every depth region $\DD^\alpha(\PP)$ is geodesically convex.
\end{theorem}

Theorem \ref{lemma:superlevel} is the reason the one-sided definition in \eqref{eq:depth_one_sided} is so effective. For each boundary direction $\xi$, the projected measure $(B_\xi)_\#\PP$ determines a threshold $t_\xi(\alpha)$, and the depth region is obtained by intersecting the corresponding horoballs. Thus, convexity is not a secondary theorem but is built directly into the definition. In Euclidean space, the same formula reduces to the familiar representation of Tukey regions as intersections of closed halfspaces.

\subsection{Why we use the one-sided definition}
\label{sec:busemann_median_def}

In $\R^d$, the usual one-sided formula for Tukey depth already encodes both sides of a hyperplane because replacing $u$ by $-u$ swaps the two complementary halfspaces. It is therefore natural to ask whether the more visibly symmetric definition
\begin{equation}
\label{eq:depth_two_sided}
\widetilde D(z;\PP):=\inf_{\xi\in\partial X}\min\{\PP(H^+_{\xi,z}),\PP(H^-_{\xi,z})\}
\end{equation}
could be used instead of \eqref{eq:depth_one_sided}.

The key point is that, on a curved Hadamard manifold, the one-sided and two-sided formulas are not equivalent. In order to mimic the Euclidean argument, one would need that for every $\xi\in\partial X$ there exists an opposite direction $\eta\in\partial X$ such that $H^+_{\xi,z}=H^-_{\eta,z}$ for all $z$. Equivalently, one would need $B_\xi=-B_\eta+C$ for some constant $C$. This is far too restrictive.

\begin{proposition}
\label{prop:depth_non_equivalence}
If for every $\xi\in\partial X$ there exists $\eta\in\partial X$ such that $B_\xi=-B_\eta+C$ for some constant $C$, then $X$ has vanishing sectional curvature.
\end{proposition}

Thus the one-sided definition is not merely a simplification of the two-sided one; it is the formulation that remains intrinsic in curved geometry while retaining the halfspace-type structure needed later.

\begin{figure}
    \centering
    \begin{subfigure}{0.35\columnwidth}
        \includegraphics[width=\columnwidth]{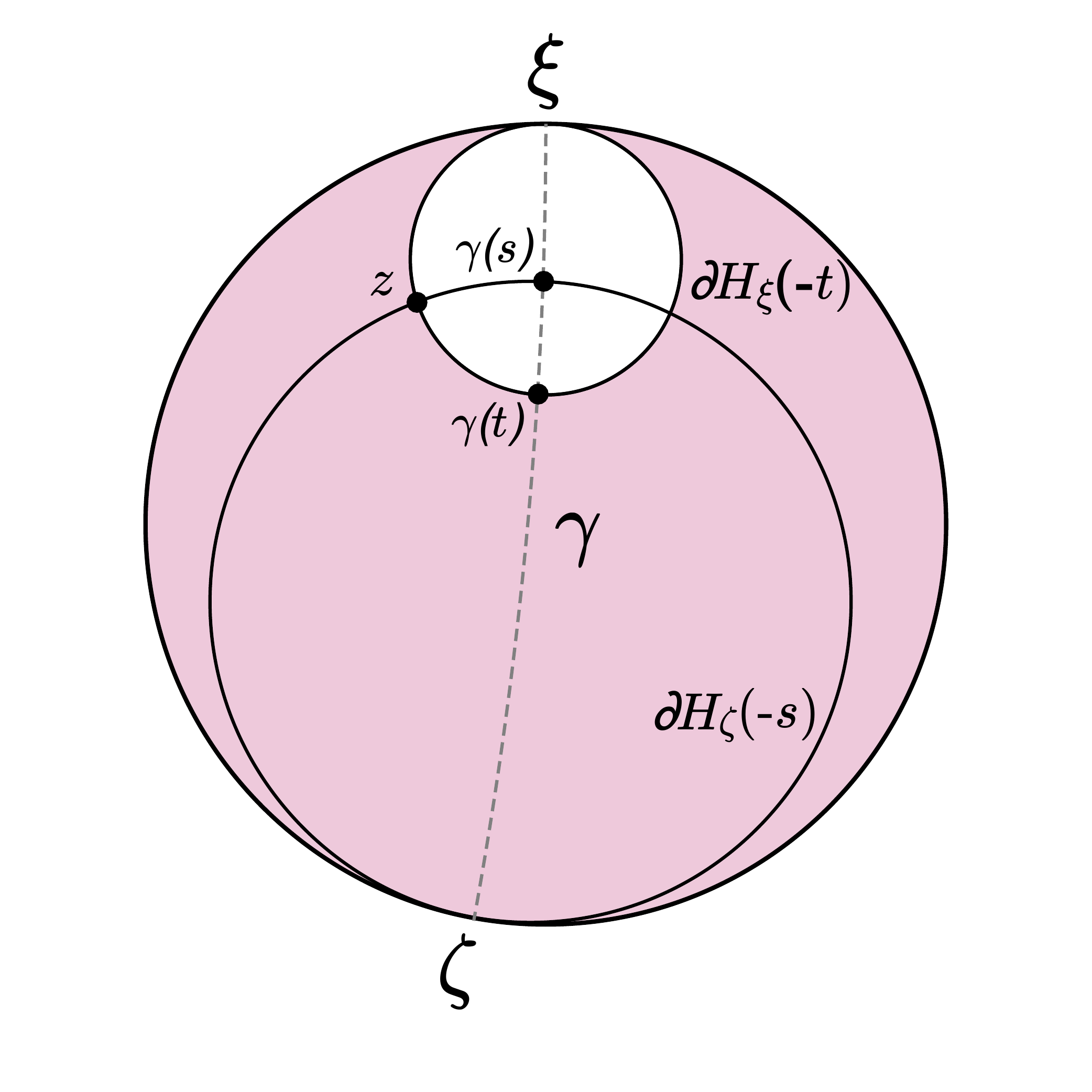}
          \caption{}
          \label{fig:set_2}
      \end{subfigure}
    \hspace{0.15\columnwidth}
    \begin{subfigure}{0.35\columnwidth}
        \includegraphics[width=\columnwidth]{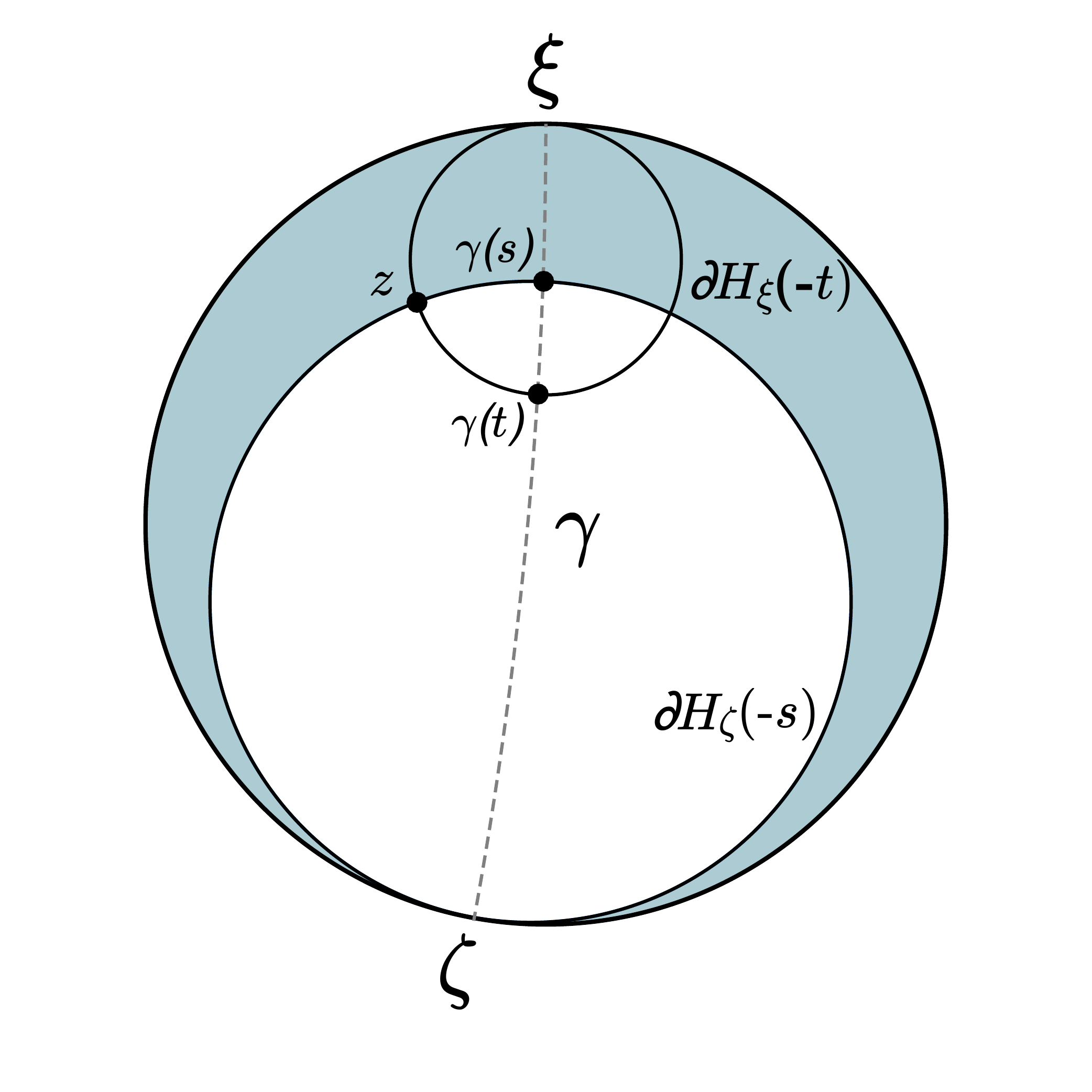}
          \caption{}
          \label{fig:set_1}
      \end{subfigure}
    \caption{Opposite boundary directions through the same point $z$ in the Poincar\'e disc model. The two panels show the sets $H^+_{\xi,z}$ and $H^+_{\zeta,z}$ for opposite-looking directions $\xi$ and $\zeta$. They illustrate why the one-sided definition does not privilege a direction: a large mass on one side is typically offset by a small mass for a competing boundary direction.}
    \label{fig:horospherical_depth}
\end{figure}

Figure \ref{fig:horospherical_depth} helps explain the point geometrically. Although a single set $H^+_{\xi,z}$ only looks at one side of the horosphere through $z$, the infimum over all $\xi\in\partial X$ automatically forces competition between opposite directions. What the two-sided definition changes is not the centrality intuition but the geometry of the resulting contours.

Indeed, the superlevel sets of $\widetilde D$ are intersections of horospherical strips rather than horoballs. If we define the lower quantile 
\begin{equation}\label{eq:q_xi}
q_\xi(\alpha):=\inf\{q\in\R:\PP(\{x\in X:B_\xi(x)\le q\})\ge \alpha\}, \quad \alpha \in (0,1), 
\end{equation}
then one obtains the representation below.

\begin{proposition}\label{prop:horo_strip}
For every $\alpha\in(0,1)$,
\begin{equation}\label{eq:strip}
\widetilde \DD^\alpha(\PP)=\bigcap_{\xi\in\partial X}\{z\in X:q_\xi(\alpha)\le B_\xi(z)\le t_\xi(\alpha)\}.
\end{equation}
In particular, the two-sided depth regions are intersections of horospherical strips and need not be geodesically convex.
\end{proposition}

For the purposes of this paper, that loss of convexity is decisive. The one-sided depth keeps the horoball intersection structure, and with it the entire convex-region theory developed in the next sections.

\subsection{Depth axioms and roadmap}

The Euclidean axiomatic framework of \cite{zuo2000general} adapts naturally to the present setting. We record the four properties that will organize the rest of the paper.

\smallskip
\begin{property}[Isometry invariance.]
\label{p:invariance}
For every isometry $\varphi:X \to X$, the depth function should satisfy $D(\varphi(z);\varphi_\#\PP)=D(z;\PP)$.
\end{property}

\begin{property}[Maximality at a center.]
\label{p:maximality}
If $\PP$ has a geometrically distinguished center, for example a symmetry center, then the depth should attain its maximum there.
\end{property}

\begin{property}[Vanishing at infinity.]
\label{p:vanishing}
If $d(o,z)\to\infty$, then $D(z;\PP)\to 0$.
\end{property}

\begin{property}[Monotonicity from a deepest point.]
\label{p:monotonicity}
If $z_*\in\argmax D(\cdot;\PP)$ and $\gamma:[0,1]\to X$ is the geodesic segment from $z_*$ to $z$, then $D(\gamma(t);\PP)\ge D(z;\PP)$ for every $t\in[0,1]$.
\end{property}

The remainder of the paper verifies these properties in a natural order. Section \ref{sec:properties} proves isometry invariance and maximality at symmetry centers. Section \ref{sec:existence} proves vanishing at infinity and uses it, together with a centerpoint theorem, to establish existence and compactness of positive-depth regions. Section \ref{sec:uniqueness} proves monotonicity and then a strictly stronger statement, namely strict quasi-concavity under negative curvature.

\section{Symmetry and equivariance}
\label{sec:properties}

Section \ref{sec:definition} identified the horospherical depth regions as intersections of horoballs. We now study how this depth behaves under the intrinsic symmetries of the manifold and of the underlying probability measure. The construction is intrinsic: isometries transport Busemann functions to Busemann functions, and therefore transport directional halfspaces, depth regions, and deepest points in a canonical way. We begin by naming the maximizers of the depth.

\begin{definition}[Busemann median]
Given a probability measure $\PP$ on $X$, the \emph{Busemann median} is the set of maximizers $\mu_*(\PP):=\argmax_{z\in X} D(z;\PP)$. We write $D_*(\PP):=\sup_{z\in X} D(z;\PP)$ for the maximal depth. When there is no risk of confusion, we abbreviate $\mu_*(\PP)$ and $D_*(\PP)$ to $\mu_*$ and $D_*$.
\end{definition}

The first basic property is the manifold analogue of affine invariance for Tukey depth.

\begin{theorem}[\ref{p:invariance}]\label{lemma:g-depth}
Let $\varphi:X\to X$ be an isometry. Then
\[
D(\varphi(z);\varphi_\#\PP)=D(z;\PP)
\qquad\text{for every } z\in X.
\]
Consequently, for every $\alpha\ge 0$,
\[
\DD^\alpha(\varphi_\#\PP)=\varphi\bigl(\DD^\alpha(\PP)\bigr),
\qquad
\mu_*(\varphi_\#\PP)=\varphi\bigl(\mu_*(\PP)\bigr).
\]
\end{theorem}

Thus, the entire depth landscape is transported equivariantly by the full isometry group of $X$. This is the precise intrinsic counterpart of affine invariance in $\R^d$: the relevant symmetry group is no longer the affine group of a vector space, but the isometry group of the manifold itself.

\begin{corollary}\label{coro:g-action}
Suppose a group $G$ acts on $X$ by isometries. Then for every $g\in G$,
\[
D(gz;g_\#\PP)=D(z;\PP),
\qquad
\mu_*(g_\#\PP)=g\mu_*(\PP).
\]
In particular, if $g_\#\PP=\PP$, then $g\mu_*(\PP)=\mu_*(\PP)$.
\end{corollary}


We next turn to maximality at a center of symmetry. As in the Euclidean theory, the right conclusion should hold under a symmetry hypothesis tailored to the geometry that the depth actually sees. On a Hadamard manifold, the following notions are useful.

\begin{definition}
Let $X$ be a Hadamard manifold and $\PP$
be a probability measure on $X$.

\begin{itemize}
\tightlist
    \item $\PP$ is \textbf{centrally symmetric} about $\theta\in X$ if $(s_\theta)_\#\PP=\PP$, where $s_\theta(x):=\exp_\theta(-\exp_\theta^{-1}x)$ is the geodesic reflection at $\theta$.
    \item $\PP$ is \textbf{horospherically symmetric} about $\theta\in X$ if $ \PP(H^+_{\xi,\theta})\ge 1/2 $ for every $\xi\in\partial X$.
\end{itemize}
\end{definition}


Horospherical symmetry is the weaker notion from the viewpoint of the depth itself. On spaces where point reflections are isometries---e.g., symmetric spaces---central symmetry coincides with invariance under geodesic symmetry. We show that central symmetry implies horospherical symmetry.

\begin{proposition}\label{prop:symmetric_prob}
If $\PP$ is centrally symmetric about $\theta$, then it is horospherically symmetric about $\theta$.
\end{proposition}

Horospherical symmetry is already sufficient for the depth to peak at the symmetry center.

\begin{proposition}[\ref{p:maximality}]\label{lemma:maximality}
If $\PP$ is horospherically symmetric about a unique point $\theta\in X$, then $\mu_*(\PP)=\{\theta\}$.
\end{proposition}

\begin{remark}[Uniqueness of horospherical symmetry centers]
Proposition \ref{lemma:maximality} requires that $\theta$ be the \emph{unique} point of horospherical symmetry. Can a measure be horospherically symmetric about two distinct points $\theta_1\ne\theta_2$?  If so, then $D(\theta_i;\PP)\ge 1/2$ for $i=1,2$, so $\mu^*(\PP)$ contains both $\theta_1$ and $\theta_2$.  In $\mathbb{R}^d$ this can happen (e.g. a measure that is halfspace-symmetric about every point on a line segment).  On a Hadamard manifold the same phenomenon is possible when the curvature is not strictly negative: for instance, a measure on a Euclidean factor of a product Hadamard manifold can be horospherically symmetric along an entire flat.  Under $\Sec_X<0$, however, the strict convexity of Busemann functions (Theorem \ref{thm:busemann_f_convex}) eventually forces any two such centers to coincide once Assumptions \ref{assump:no_atom}-\ref{assump_connected_supp} hold; this is precisely the content of the uniqueness theorem (Theorem \ref{thm:unique_rank_1}).
\end{remark}

Theorem \ref{lemma:g-depth} has an immediate consequence for the depth regions themselves.

\begin{corollary}\label{coro:region-equivariance}
For every $\alpha\ge 0$ and every isometry $\varphi:X\to X$,
\[
\DD^\alpha(\varphi_\#\PP)=\varphi\bigl(\DD^\alpha(\PP)\bigr).
\]
In particular, if $\varphi_\#\PP=\PP$, then $\DD^\alpha(\PP)$ is $\varphi$-invariant.
\end{corollary}

\begin{figure}
    \centering
    \begin{subfigure}{0.30\columnwidth}
        \includegraphics[width=\columnwidth]{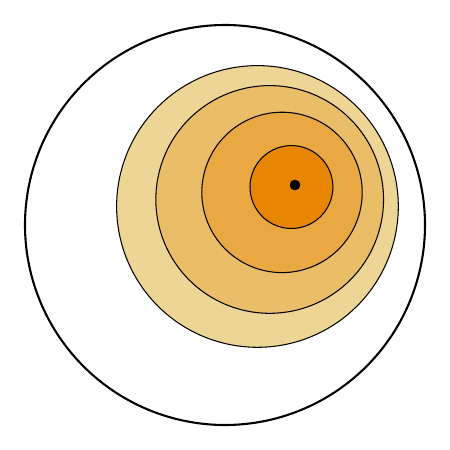}
          \caption{}
          \label{fig:rotational}
      \end{subfigure}
    \hspace{0.15\columnwidth}
    \begin{subfigure}{0.30\columnwidth}
        \includegraphics[width=\columnwidth]{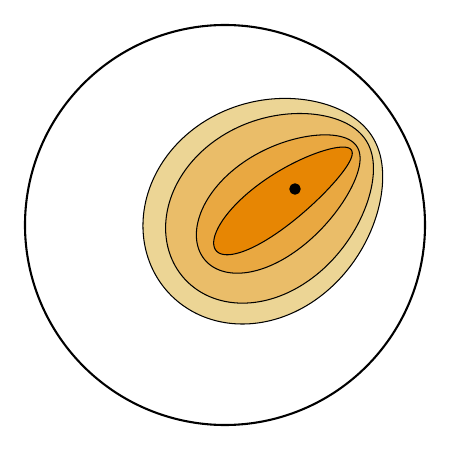}
          \caption{}
          \label{fig:central}
      \end{subfigure}
    \caption{$\alpha$-depth regions for symmetric distributions on $\mathbb H^2$ in the Poincar\'e disc model where $\alpha$ increases from outermost to innermost. Panel (a) shows a rotationally symmetric distribution and Panel (b) a centrally symmetric distribution, both centered at $(0.35,0.2)$. The common center, nested contours, and inherited symmetries visualize Theorem \ref{lemma:g-depth}, Proposition \ref{lemma:maximality}, and the depth-region representation of Theorem \ref{lemma:superlevel}.}
    \label{fig:depth_region}
\end{figure}

Combined with Theorem \ref{lemma:superlevel}, this says that every depth region is a geodesically convex set that inherits all isometric symmetries of the underlying probability measure. Figure \ref{fig:depth_region} makes the structural picture concrete. In \cref{fig:rotational}, rotational symmetry forces the contours to inherit the full stabilizer of the center, so the displayed regions appear as a concentric family about $(0.35,0.2)$. Here ``concentric'' is meant in the hyperbolic sense. In the Poincar\'e disc model, hyperbolic circles centered away from the origin are represented by Euclidean circles whose Euclidean centers are shifted toward the origin, because the model is conformal but not distance-preserving. For the same reason, the contours in \cref{fig:central} can look like distorted Euclidean ellipses even though they are centered at the same hyperbolic point and remain geodesically convex. The visual asymmetry is therefore a coordinate effect of the disc model, not a failure of symmetry of the depth itself.

At this point the qualitative geometry is clear, but one essential analytic question remains: do nontrivial depth levels actually occur for an arbitrary probability measure? This question is addressed in the next section, where vanishing at infinity and a centerpoint theorem turn the formal level sets into a genuine location theory.

\section{Existence and the centerpoint theorem}
\label{sec:existence}

Section \ref{sec:properties} shows that horospherical depth is intrinsic and symmetry-adapted. To obtain a genuine location functional, however, one still has to prove that maximal depth is attained for an arbitrary probability measure and that the positive-depth regions are not empty formal shells. The Euclidean strategy remains the correct one: first show vanishing at infinity, then prove upper semicontinuity, and finally invoke compactness.

\subsection{Vanishing at infinity and existence}

\begin{proposition}[\ref{p:vanishing}]\label{prop:no_infty}
For any Borel probability measure $\PP$ on $X$, the horospherical depth vanishes at infinity. That is,
\[
D(z;\PP)\to 0
\qquad\text{as } d(o,z)\to\infty.
\]
\end{proposition}

Proposition \ref{prop:no_infty} is the manifold analogue of the usual vanishing-at-infinity axiom for statistical depth: points escaping to spatial infinity cannot remain deep. In particular, any maximizer must lie in a bounded region.

\begin{proposition}\label{prop:usc}
The map $(z,\PP)\mapsto D(z;\PP)$ is jointly upper semicontinuous, where $z\in X$ carries the manifold topology and $\PP$ carries weak convergence of probability measures. In particular, for each fixed Borel probability measure $\PP$, the function $z\mapsto D(z;\PP)$ is upper semicontinuous.
\end{proposition}

\begin{remark}[Attainment of the infimum direction]
For every fixed $z\in X$ and every Borel probability measure $\PP$ on $X$, the map $\xi\mapsto \PP(H^+_{\xi,z})$ is upper semicontinuous on the compact space $\partial X$ (equipped with the cone topology), so the infimum in \eqref{eq:depth_one_sided} is attained. That is, there exists a \emph{least favorable direction} $\xi^*(z)\in\partial X$ with $D(z;\PP) = \PP(H^+_{\xi^*(z),z})$. 
\end{remark}

Upper semicontinuity is exactly what turns boundedness of approximate maximizers into actual existence of a deepest point.

\begin{lemma}\label{lemma:existence}
For any Borel probability measure $\PP$ on $X$, the set of maximizers
\[
\mu_*(\PP):=\argmax_{z\in X} D(z;\PP)
\]
is nonempty. If $D_*(\PP)>0$, then $\mu_*(\PP)$ is compact.
\end{lemma}

Lemma \ref{lemma:existence} leaves open one degenerate possibility, namely that $D_*(\PP)=0$ and the depth carries no locational information. The next theorem rules this out uniformly in dimension.

\subsection{A centerpoint theorem}

\begin{theorem}[Centerpoint]\label{thm:depth_lbound}
Let $X$ be a $d$-dimensional Hadamard manifold and let $\PP$ be a Borel probability measure on $X$. Then
\[
D_*(\PP)\ge \frac{1}{d+1}.
\]
\end{theorem}


\begin{proof}[Proof Sketch]
Following \cite[Lemma 6.3]{donoho1992breakdown}, regularize $\PP$ by the heat kernel to obtain $\PP_h\ll\vol$ with $\PP_h\Rightarrow \PP$ as $h\downarrow 0$. For $z\in X$ and $v\in T_zX$, write
\[
H^{\tang}_{v,z}:=\{\exp_z(w):\langle w,v\rangle_z\ge 0\}.
\]
Rusciano's centerpoint theorem \cite[Theorem 2]{rusciano2018riemannian} yields, for each $h>0$, a point $z_h$ such that every tangent halfspace $H^{\tang}_{v,z_h}$ has $\PP_h$-mass at least $1/(d+1)$. Write $v_h(\xi):= \grad_{z_h}B_\xi$, then geodesic convexity of $B_\xi$ implies $H^{\tang}_{v_h(\xi),z_h}\subset H^+_{\xi,z_h}$, and the map $\xi\mapsto v_h(\xi)=\grad_{z_h}B_\xi$ exhausts the unit sphere in $T_{z_h}X$. Hence $D(z_h;\PP_h)\ge 1/(d+1)$ for every $h$. Finally, the family $(\PP_h)$ is tight as $h\downarrow0$, so the points $z_h$ cannot escape to infinity; after passing to a convergent subsequence, joint upper semicontinuity of $D$ yields a limit point $z_*$ with
\[
D(z_*;\PP)\ge \limsup_{h\downarrow0} D(z_h;\PP_h)\ge \frac{1}{d+1}.
\]
Thus $D_*(\PP)\ge 1/(d+1)$. The full argument is given in the supplement.
\end{proof}

\begin{corollary}\label{thm:busemann_exist}
Let $X$ be a $d$-dimensional Hadamard manifold and let $\PP$ be a Borel probability measure on $X$. Then the Busemann median $\mu_*(\PP)$ is nonempty and compact.
\end{corollary}

The centerpoint theorem also tells us where the structural theorem for depth regions truly belongs. The geodesic convexity and nesting of $\DD^\alpha(\PP)$ come directly from their representation as intersections of horoballs in Section \ref{sec:definition}; the genuinely new information supplied here is boundedness, nonemptiness, and topological regularity at positive depth levels.

\begin{theorem}\label{thm:region_region_prop}
Let $X$ be a $d$-dimensional Hadamard manifold and let $\PP$ be a Borel probability measure on $X$. Then the depth regions satisfy the following properties.
\begin{enumerate}
    \item $\DD^\alpha(\PP)$ is geodesically convex and nested in $\alpha$.
    \item $\DD^\alpha(\PP)$ is compact for every $\alpha>0$.
    \item $\DD^\alpha(\PP)$ is nonempty for every $0<\alpha\le 1/(d+1)$.
    \item If $D(\cdot;\PP)$ is continuous, then $\DD^\alpha(\PP)$ is a topological $d$-ball for every $0<\alpha<D_*(\PP)$.
\end{enumerate}
\end{theorem}


\begin{proof}[Proof Sketch for part (4)]
Fix $0<\alpha<D_*(\PP)$ and choose $y\in X$ with $D(y;\PP)>\alpha$. Continuity of $D(\cdot;\PP)$ gives a small ball $B(y,r_0)\subset \DD^\alpha(\PP)$, so the region has nonempty interior. For each unit vector $u\in S_yX$, let $\rho(u)$ be the exit radius of the ray $t\mapsto \exp_y(tu)$ from $\DD^\alpha(\PP)$. Geodesic convexity and compactness imply that each ray meets $\DD^\alpha(\PP)$ in exactly the interval
\[
\{\exp_y(tu):0\le t\le \rho(u)\},
\]
with $0<\rho(u)<\infty$. A standard compactness/openness argument shows that $u\mapsto \rho(u)$ is continuous. Defining $\Phi(0):=y$ and
\[
\Phi(ru):=\exp_y(r\,\rho(u)\,u),\qquad u\in S_yX,\ 0<r\le 1,
\]
one obtains a homeomorphism from the closed unit ball in $T_yX$ onto $\DD^\alpha(\PP)$. Hence $\DD^\alpha(\PP)$ is a topological $d$-ball. The details are given in the supplement.
\end{proof}

\begin{remark}[Conditions for continuity of the depth]
Part (4) of Theorem \ref{thm:region_region_prop} assumes that $D(\cdot;\PP)$ is continuous. Since the depth is always upper semicontinuous (Proposition \ref{prop:usc}), continuity reduces to lower semicontinuity. A sufficient condition is that $(\xi,t)\mapsto \PP(\{x:B_\xi(x)\ge t\})$ is jointly continuous on $\partial X\times\mathbb{R}$, which holds whenever $\PP$ satisfies Assumption \ref{assump:no_atom} (atomlessness of all Busemann projections) and the Busemann functions vary continuously in the boundary parameter.  On a Hadamard manifold, both conditions are automatic when $\PP\ll\vol$.
\end{remark}

Theorem \ref{thm:region_region_prop} closes the structural loop initiated in Section \ref{sec:definition}. The depth regions are now known to be symmetry-adapted, geodesically convex, nested, compact sets at every positive level, and nonempty throughout the centerpoint range. In particular, the Busemann median is always a nonempty compact geodesically convex set. What remains is the issue of uniqueness: the terminal region at level $D_*(\PP)$ could still, a priori, be a segment or a higher-dimensional convex body rather than a single point.

In $\R^d$, ruling out such flat maxima typically requires additional regularity or general-position assumptions. On negatively curved manifolds, part of the necessary rigidity comes from the geometry itself, and this is what the next section exploits.

\section{Strict quasi-concavity and uniqueness}
\label{sec:uniqueness}

By Corollary \ref{thm:busemann_exist} and Theorem \ref{thm:region_region_prop}, every positive-depth region is a compact geodesically convex body and the Busemann median exists. The remaining question is whether the deepest region can have positive diameter. On negatively curved Hadamard manifolds, the decisive input is that Busemann functions are strictly convex along geodesics that do not point directly toward their own boundary direction. This turns weak convexity of horoballs into strict convexity and ultimately forces the depth function to have a single peak.

\subsection{Busemann functions along geodesics}

\begin{theorem}[Convexity along geodesics]\label{thm:busemann_f_convex}
Let $\widetilde\gamma:\R\to X$ be a complete unit-speed geodesic, and write $\xi^+:=\widetilde\gamma(+\infty)$ and $\xi^-:=\widetilde\gamma(-\infty)$. Then the following hold.
\begin{enumerate}
    \item If $\xi\notin\{\xi^+,\xi^-\}$, then $B_\xi\circ\widetilde\gamma$ is convex on $\R$. If $\Sec_X<0$, it is strictly convex.
    \item If $\xi\in\{\xi^+,\xi^-\}$, then $B_\xi\circ\widetilde\gamma$ is affine with nonzero slope. More precisely,
    \[
    B_{\xi^+}(\widetilde\gamma(s))=-s+ B_{\xi^+}(\widetilde\gamma(0)),
    \qquad
    B_{\xi^-}(\widetilde\gamma(s))=s+B_{\xi^-}(\widetilde\gamma(0)).
    \]
\end{enumerate}
In particular, any geodesic segment inherits the same conclusion by restriction.
\end{theorem}

Theorem \ref{thm:busemann_f_convex} is the curvature mechanism behind uniqueness. In flat space, the corresponding functions are affine along every geodesic, which is precisely why strict quasi-concavity is delicate in the Euclidean case. Under negative curvature, the only directions along which a Busemann function fails to bend are the two asymptotic directions of the geodesic itself.

\begin{corollary}\label{prop:horoball_convex}
For every $\xi\in\partial X$ and $t\in\R$, the horoball
\[
H_\xi(t):=\{x\in X:B_\xi(x)\le t\}
\]
is geodesically convex. If $\Sec_X<0$, then it is strictly geodesically convex.\footnote{That is, if $x\neq y$ belong to $H_\xi(t)$ and $\gamma:[0,1]\to X$ is the geodesic segment joining them, then $\gamma(s)\in H_\xi^{\circ}(t)$ for every $s\in(0,1)$. Equivalently, no nontrivial geodesic segment is contained in the boundary $\partial H_\xi(t)$.}
\end{corollary}

A second byproduct is the monotonicity axiom for depth.

\begin{corollary}[\ref{p:monotonicity}]\label{coro:monotonicity}
Let $z_*\in\mu_*(\PP)$ and let $z\in X$. If $\gamma:[0,1]\to X$ is the geodesic segment with $\gamma(0)=z_*$ and $\gamma(1)=z$, then
\[
D(\gamma(t);\PP)\ge D(z;\PP)
\qquad\text{for every } t\in[0,1].
\]
\end{corollary}


\subsection{Nondegeneracy assumptions}

To turn strict convexity of horoballs into strict quasi-concavity of the depth itself, one needs mild regularity assumptions on the one-dimensional projected measures $(B_\xi)_\#\PP$.

\begin{assumption}\label{assump:no_atom}
For every $\xi\in\partial X$ and every $t\in\R$,
\[
\PP\bigl(\{x\in X:B_\xi(x)=t\}\bigr)=0.
\]
Equivalently, each pushforward measure $(B_\xi)_\#\PP$ is atomless.
\end{assumption}

A sufficient condition is $\PP\ll\vol$, because every horosphere has zero Riemannian volume. Assumption \ref{assump:no_atom} fails for any discrete measure, and in particular for the empirical measure $\PP_n=n^{-1}\sum_{i=1}^n\delta_{X_i}$.  This means that the sample Busemann median $\mu_*(\PP_n)$ may not be a singleton, even under strictly negative curvature.  The situation is analogous to the Euclidean Tukey median: for a finite point cloud in general position in $\R^d$, the Tukey median is generically a set of positive diameter, and uniqueness is a property of the population depth, not of the sample depth.  In practice, one therefore works with the sample median \emph{set} $\mu_*(\PP_n)$, which by Theorem \ref{thm:sample_converg_depth_region} part (3) converges to the unique population median whenever the latter exists. Alternatively, one may adopt a decision rule for selecting a representative from the set $\mu_*(\PP)$, analogous to the approach used in the Euclidean setting \cite{donoho1992breakdown}.

\begin{lemma}\label{lemma:D-upper}
Under Assumption \ref{assump:no_atom}, one has
\[
D(z;\PP)<1
\qquad\text{for every } z\in X.
\]
\end{lemma}

\begin{assumption}\label{assump_connected_supp}
For every $\xi\in\partial X$, the support of $(B_\xi)_\#\PP$ is an interval in $\R$.
\end{assumption}

Assumption \ref{assump_connected_supp} asks that every one-dimensional Busemann projection have connected support. A convenient sufficient condition is that $\supp(\PP)$ be connected, since $B_\xi$ is continuous and continuous images of connected sets are connected. In particular, Assumptions \ref{assump:no_atom} and \ref{assump_connected_supp} hold simultaneously whenever $\PP\ll\vol$ and $\supp(\PP)$ is connected, the same kind of regularity commonly imposed in the Euclidean uniqueness theory \cite{rousseeuw1999depth, mizera2002continuity}.

\subsection{Strict quasi-concavity and uniqueness}

\begin{theorem}[Strict quasi-concavity on negatively curved manifolds]\label{thm:unique_rank_1}
Assume $\Sec_X<0$ and that $\PP$ satisfies Assumptions \ref{assump:no_atom} and \ref{assump_connected_supp}. Let $x\neq y$ with $\min\{D(x; \PP), D(y; \PP)\} > 0$ and let $\gamma:[0,1]\to X$ be the geodesic segment joining them. Then
\[
D(\gamma(s);\PP)
>
\min\{D(x;\PP),D(y;\PP)\}
\qquad\text{for every } s\in(0,1).
\]
That is, $D(\cdot;\PP)$ is strictly quasi-concave on $\{z \in X: D(z; \PP) > 0\}$. In particular, the deepest depth region, Busemann median $\mu_*(\PP)$, is a singleton.
\end{theorem}

\begin{remark}
    It is important to interpret Theorem \ref{thm:unique_rank_1} as a genuinely negative-curvature result. The strict part of the argument comes from the fact that in strictly negatively curved manifolds, a nontrivial geodesic segment cannot remain in the boundary of a horoball. In higher-rank nonpositively curved spaces, flats (regions of zero curvature) create boundary directions along which this strictness fails, so one cannot expect the same argument to force uniqueness. The obstruction is therefore geometric rather than technical: the presence of flats weakens the horoball geometry in precisely the way relevant to strict quasi-concavity.
\end{remark}

It is instructive to compare this with the Euclidean situation. For Tukey depth on $\R^d$, uniqueness usually requires extra hypotheses such as general position for empirical measures or absolute continuity together with connected convex support for population measures \cite{rousseeuw1999depth, mizera2002continuity}. Here the support condition enters only through the projected one-dimensional measures, while the strict part of the argument is supplied directly by curvature via Theorem \ref{thm:busemann_f_convex}. In this sense, negative curvature is not merely compatible with the classical theory; it strengthens it.

\begin{corollary}\label{coro:stability}
Assume $\Sec_X<0$ and that $\PP$ satisfies Assumptions \ref{assump:no_atom} and \ref{assump_connected_supp}, so $\mu_*(\PP)$ is singleton. Then there exists a nondecreasing function
\[
\omega:[0,1]\to[0,\infty)
\qquad\text{with}\qquad
\omega(\eta)\to 0 \text{ as } \eta\downarrow 0,
\]
such that for every Borel probability measure $\QQ$ on $X$,
\[
\sup_{z\in\mu_*(\QQ)} d\bigl(z,\mu_*(\PP)\bigr)
\le
\omega\bigl(\|\PP-\QQ\|_{\mathrm{TV}}\bigr).
\]
\end{corollary}

\begin{remark}[Convention for total variation]
Throughout Sections \ref{sec:uniqueness} and \ref{sec:robust}, we use the probability-measure convention $\|\PP-\QQ\|_{\mathrm{TV}}:=\sup_{A\in\mathcal B(X)} |\PP(A)-\QQ(A)|$, so $\|\PP-\QQ\|_{\mathrm{TV}}\in[0,1]$. With this normalization, the modulus in Corollary \ref{coro:stability} is naturally defined on $[0,1]$, and under Huber contamination one has $\|\PP_\varepsilon-\PP\|_{\mathrm{TV}}\le \varepsilon$.
\end{remark}

Corollary \ref{coro:stability} says that once negative curvature has collapsed the terminal depth region to a single point, the Busemann median becomes a sharp peak of the depth landscape. Section \ref{sec:robust} converts that geometric sharpness into robustness under perturbations of the underlying measure.

\begin{remark}[Rate of the modulus of continuity]
Corollary \ref{coro:stability} guarantees the existence of a modulus $\omega$ but says nothing about its rate.  In the Euclidean setting, quantitative moduli are known when the depth function has a \emph{quadratic peak}, meaning $D(\theta;\PP)-D(z;\PP)\ge c\,d(z,\theta)^2$ near the median $\theta$.  On a negatively curved manifold, the strict convexity of Busemann functions along non-asymptotic geodesics (Theorem \ref{thm:busemann_f_convex}) provides a curvature-dependent lower bound on the second derivative of $B_\xi\circ\gamma$, which suggests that a quadratic peak condition could be verified under lower curvature bounds.  More precisely, if $\Sec_X\le -\kappa^2<0$ and the survival function $S_\xi$ has a uniformly bounded density near the critical level, one expects $\omega(\eta)=O(\eta^{1/2})$, the same rate as in the Euclidean case.  A quantitative statement requires controlling the curvature-dependent constants in the strict convexity estimate and is left to future work.
\end{remark}

\section{Robustness}
\label{sec:robust}

Robustness enters the present theory at three levels. First, the depth itself is stable under arbitrary perturbations of the underlying probability measure in total variation distance. Second, when contaminating mass escapes to infinity along a geodesic ray, the contaminated depths admit a deterministic directional limit that depends on the boundary direction but not on the rate of escape. Under the additional uniqueness hypothesis in Theorem \ref{thm:boundary_robust} part (3), the corresponding contaminated medians converge to the unique maximizer of that limiting depth. Third, these bounds imply a nontrivial lower bound on the breakdown point of the Busemann median. The first statement is the direct analogue of the classical robustness of Tukey depth in $\R^d$ \cite{donoho1992breakdown, chen2002influence}; the second is specific to Hadamard geometry.

\subsection{Stability under total variation perturbations}

We begin with the simplest notion of perturbation: changing the underlying probability measure by a small amount in total variation. The key point is that horospherical depth is defined through probabilities of measurable halfspace-like sets, so total variation control transfers immediately to the depth function.

\begin{theorem}\label{thm:robust}
Let $X$ be a Hadamard manifold. For any Borel probability measures $\PP$ and $\QQ$ on $X$ and any $z\in X$,
\[
\bigl|D(z;\PP)-D(z;\QQ)\bigr|\le \|\PP-\QQ\|_{\mathrm{TV}}.
\]
\end{theorem}


An immediate consequence is stability under Huber's $\varepsilon$-contamination model.

\begin{corollary}[Huber's $\varepsilon$-contamination]\label{coro:huber}
Let $\PP$ and $\QQ$ be Borel probability measures on $X$, and define $\PP_\varepsilon:=(1-\varepsilon)\PP+\varepsilon\QQ$ for $\varepsilon\in[0,1]$. Then for every $z\in X$,
\[
D(z;\PP)-\varepsilon \le D(z;\PP_\varepsilon) \le D(z;\PP)+\varepsilon.
\]
\end{corollary}

The same bound propagates from the depth function to the depth regions.

\begin{theorem}[Inclusion of depth regions]\label{thm:inclusion_region}
Let $\PP_\varepsilon=(1-\varepsilon)\PP+\varepsilon\QQ$ as above. Then for every $\alpha\in[0,1]$,
\[
\DD^{\alpha+\varepsilon}(\PP)\subset \DD^{\alpha}(\PP_\varepsilon)\subset \DD^{\alpha-\varepsilon}(\PP),
\]
where, as usual, $\DD^{\alpha}(\PP)=X$ for $\alpha\le 0$ and $\DD^{\alpha}(\PP)=\varnothing$ for $\alpha>1$.
\end{theorem}


When the median is unique, the estimator itself inherits quantitative stability.

\begin{corollary}[Stability of the Busemann median under ordinary contamination]\label{coro:median_huber}
Assume $\Sec_X<0$ and that $\PP$ satisfies Assumptions \ref{assump:no_atom}--\ref{assump_connected_supp}. Let $\omega$ be the modulus of continuity from Corollary \ref{coro:stability}. Then for every Borel probability measure $\QQ$ and every $\varepsilon\in[0,1]$,
\[
\sup_{z\in \mu_*(\PP_\varepsilon)} d\bigl(z,\mu_*(\PP)\bigr)\le \omega(\varepsilon),
\qquad
\PP_\varepsilon=(1-\varepsilon)\PP+\varepsilon\QQ.
\]
\end{corollary}


\subsection{Contamination escaping to the boundary}

The previous results treat contamination in a purely measure-theoretic sense. On a Hadamard manifold there is an additional failure mode: the contaminating mass may remain small but move arbitrarily far away, converging to a point of the visual boundary. This has no Euclidean analogue. The basic mechanism is already visible in the asymptotic expansion previewed in the heuristic at the beginning of Section \ref{sec:background}. Along a geodesic ray $\gamma$ with endpoint $\xi$, 
\[
d^2(z,\gamma(t)) = t^2 + 2tB_\xi(z) + O(1).
\]
For the Fr\'echet mean, the escape direction is weighted by the diverging factor $2t$; for horospherical depth, only the ordering induced by the Busemann functions matters. The former is sensitive to how far the outliers have travelled, whereas the latter is sensitive only to the direction in which they escape.

\begin{theorem}[Boundary robustness]\label{thm:boundary_robust}
Let $X$ be a Hadamard manifold with $\Sec_X<0$, let $\PP$ be a Borel probability measure on $X$, and let $\gamma:[0,\infty)\to X$ be a unit-speed geodesic ray with $\gamma(\infty)=\xi$. Fix $\varepsilon\in(0,1)$ and define $\PP_t:=(1-\varepsilon)\PP+\varepsilon\delta_{\gamma(t)}$.
Assume that Assumption \ref{assump:no_atom} holds. Then the following statements are true.
\begin{enumerate}
    \item For every $z\in X$, $D(z;\PP_t)\to D_\infty(z)$ as $t\to\infty$,
    where
    \[
    D_\infty(z)
    =
    \min\Bigl\{(1-\varepsilon)\PP(H^+_{\xi,z}),\ (1-\varepsilon)D(z;\PP)+\varepsilon\Bigr\}.
    \]
    \item Writing $\DD_\infty^\alpha:=\{z\in X: D_\infty(z)\ge \alpha\}$, one has for every $\alpha\in(0,1-\varepsilon]$,
    \[
    \DD_\infty^\alpha
    =
    \DD^{\alpha_2}(\PP)\cap H_\xi\bigl(t_\xi(\alpha_1)\bigr),
    \qquad
    \alpha_1:=\frac{\alpha}{1-\varepsilon},
    \quad
    \alpha_2:=\frac{\alpha-\varepsilon}{1-\varepsilon},
    \]
    with the convention $\DD^{\alpha}(\PP)=X$ for $\alpha\le 0$. For $\alpha>1-\varepsilon$, one has $\DD_\infty^\alpha=\varnothing$.
    \item If $D_\infty$ has a unique maximizer $\mu_\infty$ and $D_\infty(\mu_\infty)>\varepsilon$, then every choice $\mu_t\in \mu_*(\PP_t)$ satisfies $\mu_t\to \mu_\infty$ in $X$. 
    \item Assume in addition that $\PP$ has finite second moment, and let $\mu_t^F$ denote the Fr\'echet mean of $\PP_t$. Then $\mu_t^F\to \xi$ in $X\cup\partial X$.
\end{enumerate}
\end{theorem}

The form of $D_\infty$ is geometrically transparent. The term $(1-\varepsilon)\PP(H^+_{\xi,z})$ is the new obstruction created by the escape direction $\xi$: along that direction the contaminating point eventually lies on the wrong side of every horosphere through $z$, so the extra mass does not contribute. In every other direction, the point mass eventually lies inside the corresponding upper horospherical halfspace and contributes a full $\varepsilon$. The limiting depth is therefore the minimum of one directional constraint and one globally shifted copy of the original depth. Statement (2) simply rewrites these two inequalities in the language of depth regions and horoballs.


Theorem \ref{thm:boundary_robust} isolates a qualitative distinction between depth-based and distance-based location estimators on negatively curved spaces. The limiting depth $D_\infty$ depends on the escape direction but not on the distance traveled. When $D_\infty$ has a unique maximizer above level $\varepsilon$, part (3) shows that the contaminated Busemann medians converge to that limiting maximizer. The Fréchet mean behaves differently because squared distance amplifies remote contamination: once the outliers move along a ray, the linear term $2tB_\xi(z)$ dominates the objective and pulls the minimizer toward the corresponding boundary point. In spaces such as $\bSS_p^+$, this says that increasingly singular covariance matrices influence the Busemann-depth landscape through their limiting eigenstructure, while the Fr\'echet mean is driven toward degeneracy itself.

\subsection{Breakdown under contamination}

The preceding results show that small contamination cannot arbitrarily perturb the depth. We now translate this into a breakdown statement for the median. Because the ambient space admits a visual compactification, there are two natural ways for breakdown to occur: the estimator can drift arbitrarily far from the original median inside $X$, or it can converge to a specific boundary direction.

\begin{definition}\label{def:breakdown}
For a point $x\in X$ and a nonempty set $A\subset X$, write $d(x,A):=\inf_{a\in A} d(x,a)$. Given a Borel probability measure $\PP$ on $X$, a contamination level $\varepsilon\in[0,1]$, and a sequence of Borel probability measures $\{\QQ_n\}$, define $\PP_{\varepsilon,n}:=(1-\varepsilon)\PP+\varepsilon\QQ_n$. The \emph{breakdown point} of the Busemann median is
\[
\varepsilon_{\mathrm{bd}}(\PP)
:=
\inf\Bigl\{
\varepsilon\in[0,1]:
\exists\, \QQ_n,\ \exists\, \mu_n\in \mu_*(\PP_{\varepsilon,n})
\text{ such that }
 d\bigl(\mu_n,\mu_*(\PP)\bigr)\to\infty
\Bigr\}.
\]
For $\xi\in\partial X$, the \emph{boundary-direction breakdown point in direction $\xi$} is
\[
\varepsilon_{\mathrm{bd}}^\xi(\PP)
:=
\inf\Bigl\{
\varepsilon\in[0,1]:
\exists\, \QQ_n,\ \exists\, \mu_n\in \mu_*(\PP_{\varepsilon,n})
\text{ such that }
 \mu_n\to \xi
\text{ in }X\cup\partial X
\Bigr\}.
\]
Finally, define $\varepsilon_{\mathrm{bd}}^\partial(\PP):=\inf_{\xi\in\partial X}\varepsilon_{\mathrm{bd}}^\xi(\PP)$.
\end{definition}

\begin{theorem}[Breakdown point]\label{thm:breakdown}
For any Borel probability measure $\PP$ on a Hadamard manifold $X$ and any $\xi\in\partial X$,
\[
\varepsilon_{\mathrm{bd}}^\xi(\PP)
\ge
\varepsilon_{\mathrm{bd}}^\partial(\PP)
\ge
\varepsilon_{\mathrm{bd}}(\PP)
\ge
\frac{D_*(\PP)}{1+D_*(\PP)}.
\]
In particular, Theorem \ref{thm:depth_lbound} yields the universal bound
\[
\varepsilon_{\mathrm{bd}}^\xi(\PP)
\ge
\varepsilon_{\mathrm{bd}}^\partial(\PP)
\ge
\varepsilon_{\mathrm{bd}}(\PP)
\ge
\frac{1}{d+2}.
\]
\end{theorem}


\begin{remark}
The bound $\varepsilon_\mathrm{bd}(\PP)\ge 1/(d+2)$ is weaker than the classical Euclidean benchmark $1/(d+1)$.  It is worth noting that the loss comes entirely from the algebraic step $D^*/(1+D^*)\ge 1/(d+2)$ applied to the centerpoint bound $D^*\ge 1/(d+1)$, not from any weakening of the centerpoint theorem itself.  In other words, the centerpoint theorem already provides the sharp $1/(d+1)$ depth guarantee; the gap arises in the passage from depth to breakdown.  Closing this gap, for instance, by showing directly that $\varepsilon_\mathrm{bd}(\PP)\ge D^*(\PP)$ or by improving the algebraic translation under additional regularity or symmetry assumptions, remains open.
\end{remark}

\section{Empirical convergence}
\label{sec:computation}

We now pass from the population depth to its empirical counterpart. Given i.i.d. observations $X_1,\dots,X_n\sim \PP$, write $\PP_n:=n^{-1}\sum_{i=1}^n\delta_{X_i}$ for the empirical measure. The statistical question is whether the random depth landscape $z\mapsto D(z;\PP_n)$ converges uniformly to $z\mapsto D(z;\PP)$ and whether this in turn forces the sample depth regions and sample Busemann medians to converge. Once uniform convergence is available, the rest is largely deterministic: the key geometric inputs have already been established in Sections \ref{sec:existence} and \ref{sec:uniqueness}.

\begin{definition}[Sample horospherical depth and sample Busemann median]
Given observations $X_1,\dots,X_n\in X$, the \emph{sample horospherical depth} is
\begin{equation}\label{eqn:u_conver_sample_def}
D(z;\PP_n):=\inf_{\xi\in\partial X}\frac1n\sum_{i=1}^n \mathbf 1\{B_\xi(z)\le B_\xi(X_i)\}.
\end{equation}
Its set of maximizers is the \emph{sample Busemann median}, denoted by $\mu_*(\PP_n)$. For $\alpha\ge 0$, the corresponding sample depth region is $\DD^\alpha(\PP_n):=\{z\in X:D(z;\PP_n)\ge \alpha\}$.
\end{definition}



The empirical theory reduces to a uniform law of large numbers for the class of upper horospherical halfspaces 
\[
\mathcal H^+ := \{H^+_{\xi,z}:\xi\in\partial X,\ z\in X\}.
\] 
Indeed, $D(z;\PP)=\inf_{\xi\in\partial X}\PP(H^+_{\xi,z})$ and $D(z;\PP_n)=\inf_{\xi\in\partial X}\PP_n(H^+_{\xi,z})$, so 
\[
\sup_{z\in X}\bigl|D(z;\PP_n)-D(z;\PP)\bigr|
\le \sup_{A\in\mathcal H^+}\bigl|\PP_n(A)-\PP(A)\bigr|.
\] 
The problem is therefore to control the empirical process uniformly over an uncountable family indexed jointly by a boundary direction $\xi$ and a threshold point $z$.

We treat this problem in two regimes. On symmetric spaces of noncompact type, the family $\mathcal H^+$ admits a finite-dimensional parametrization, and VC theory gives the required uniform control. On a general Hadamard manifold, we instead combine compactness of the visual boundary with continuity of Busemann functions and Assumption \ref{assump:no_atom}, which rules out mass on horospheres. The compactness argument, deferred to the supplement, approximates $\partial X$ by finite nets, controls the
resulting threshold perturbations uniformly on compact positive-depth regions, and then uses vanishing at infinity to remove the remaining
tail of the manifold.

\begin{lemma}\label{lemma:VC}
    Let \(X\) be a symmetric space of noncompact type, and let \[
    \mathcal H^+ := \{H^+_{\xi,z}:\xi\in\partial X,\ z\in X\}.
    \] 
    Then $\mathcal H^+$ has finite VC dimension. In particular, if $X=H^d$, then $\mathrm{VC}(\mathcal H^+)\le d+2$.
\end{lemma}

\begin{theorem}\label{thm:sample_converg}
Let $X$ be a symmetric space of the noncompact type and $X_1,\dots,X_n \sim \PP$. Then
\begin{equation}\label{eqn:u_conver}
\sup_{z\in X}\bigl|D(z;\PP_n)-D(z;\PP)\bigr|\to 0
\qquad\text{almost surely.}
\end{equation}
\end{theorem}

\begin{remark}
    Since $\sup_{z\in X}|D(z;\PP_n)-D(z;\PP)| \le \sup_{A\in\mathcal H^+}|\PP_n(A)-\PP(A)|$, any standard VC deviation inequality for the class $\mathcal H^+$ yields the corresponding nonasymptotic bound for the depth sup-norm. We do not record constants here.
\end{remark}


On a general Hadamard manifold, one instead uses compactness of the visual boundary and continuity of the Busemann functions.

\begin{theorem}\label{thm:sample_converg_2}
Let $X$ be a Hadamard manifold, and assume that $\PP$ satisfies Assumption \ref{assump:no_atom}. Then, for i.i.d. observations $X_1,\dots,X_n\sim\PP$, the uniform convergence in \eqref{eqn:u_conver} holds almost surely.
\end{theorem}

\begin{proof}[Proof Sketch]
Varadarajan's theorem \cite{varadarajan1958convergence} gives $\PP_n\Rightarrow \PP$ almost surely. Fix a compact set $K\subset X$. If $(\xi_n,z_n)\to(\xi,z)$ in $\partial X\times K$, continuity of the Busemann functions in the boundary parameter implies $B_{\xi_n}\to B_\xi$ uniformly on $K$. Hence the sets $H^+_{\xi_n,z_n}\cap K$ are squeezed between $\{B_\xi>B_\xi(z)+\varepsilon\}\cap K$ and $\{B_\xi\ge B_\xi(z)-\varepsilon\}\cap K$; Assumption \ref{assump:no_atom} removes the boundary mass on $\{B_\xi=B_\xi(z)\}$, so $\PP_n(H^+_{\xi_n,z_n})\to \PP(H^+_{\xi,z})$. By compactness of $\partial X\times K$, this convergence is uniform on $\partial X\times K$, and therefore
\[
\sup_{z\in K}|D(z;\PP_n)-D(z;\PP)|\to 0.
\]
To control the tail, let $z\notin \overline B(o,R)$ and let $\eta$ be the boundary point opposite to $z$ along the geodesic through $o$ and $z$. Then $B_\eta(z)=d(o,z)$ and $H^+_{\eta,z}\subset X\setminus B(o,R)$, so $D(z;Q)\le Q(X\setminus B(o,R))$ for every probability measure $Q$. Choosing $R$ so that $\PP(X\setminus \overline B(o,R))$ is small and using $\PP_n\Rightarrow \PP$ on this continuity set gives the same tail bound for $\PP_n$. Combining the compact-uniform convergence on $\overline B(o,R)$ with the tail estimate yields
\[
\sup_{z\in X}|D(z;\PP_n)-D(z;\PP)|\to 0
\]
almost surely. The full compactness argument is given in the supplement.
\end{proof}

Once \eqref{eqn:u_conver} is known, convergence of regions and medians follows from standard contour and argmax arguments. We state these consequences together because, conceptually, they are all manifestations of the same fact: a uniformly convergent sequence of depth functions cannot create or destroy macroscopic contours.

\begin{theorem}\label{thm:sample_converg_depth_region}
Assume that \eqref{eqn:u_conver} holds almost surely, and let $\alpha_n\to \alpha$.
\begin{enumerate}
    \item For every $\varepsilon>0$ and every $\delta\in(0,\varepsilon)$,
    \[
    \DD^{\alpha+\varepsilon}(\PP)\subset \DD^{\alpha_n+\delta}(\PP_n)\subset \DD^{\alpha_n}(\PP_n)\subset \DD^{\alpha_n-\delta}(\PP_n)\subset \DD^{\alpha-\varepsilon}(\PP)
    \]
    almost surely for all sufficiently large $n$.

    \item If $\PP\bigl(\{z\in X:D(z;\PP)=\alpha\}\bigr)=0$, then $\DD^{\alpha_n}(\PP_n)\to \DD^{\alpha}(\PP)$ almost surely in the Painlev\'e--Kuratowski sense. When $\alpha>0$, the sets are eventually compact by Theorem \ref{thm:region_region_prop}, so this is equivalent to Hausdorff convergence.

    \item If $\mu_*(\PP)=\{\theta\}$ is a singleton, then every selection $\theta_n\in\mu_*(\PP_n)$ satisfies $\theta_n\to\theta$ almost surely.
\end{enumerate}
\end{theorem}

Theorem \ref{thm:sample_converg_depth_region} is the statistical closure of the earlier geometric theory. The sample depth regions converge to their population counterparts, and whenever the population median is unique the sample Busemann median is strongly consistent. Thus, the intrinsic depth constructed from horoballs is not only well posed at the population level but statistically stable under sampling.

\section{Examples and computation}
\label{sec:numerical}
The construction is intrinsic, but in the model spaces that motivate the paper the Busemann functions admit explicit coordinates. This makes the theory concrete and suggests practical approximation schemes. The aim of this section is illustrative: we record formulas in basic examples and describe a finite-direction approximation for empirical depth regions and medians, while leaving a quantitative analysis of discretization and optimization error to future work.



\subsection{Explicit formulas in model spaces} 

\begin{example}[Poincar\'e ball model $\mathbb H^d$]
By Example \ref{example:busemann_hyper}, for $\xi\in S^{d-1}$ and $x\in\mathbb H^d$ one has $B_\xi(x)=\log\|x-\xi\|^2-\log(1-\|x\|^2)$. Hence
\[
D(z;\PP)=\inf_{\xi\in S^{d-1}} \PP\!\left\{x\in\mathbb H^d:\frac{\|x-\xi\|^2}{1-\|x\|^2}\ge \frac{\|z-\xi\|^2}{1-\|z\|^2}\right\}.
\]
For $\alpha\in(0,1)$, the depth region takes the form
\[
\DD^\alpha(\PP)
=\bigcap_{\xi\in S^{d-1}}\left\{x\in\mathbb H^d:\|x-\xi\|^2+e^{t_\xi(\alpha)}\|x\|^2\le e^{t_\xi(\alpha)}\right\}.
\]
Thus, in the ball model, each depth region is an intersection of Euclidean balls internally tangent to the boundary sphere. The curvature is visible directly in the geometry of the contours.
\end{example}

\begin{example}[Affine-invariant geometry on $\bSS_p^+$]\label{example:busemann_spd}
Let $X=\bSS_p^+$ be the manifold of $p\times p$ symmetric positive-definite matrices equipped with the affine-invariant metric. A boundary direction may be represented by a unit symmetric matrix $H\in \bSS_p^1$, and if $H=Q\Lambda Q^\top$ is its spectral decomposition, then \cite[Proposition 10.69]{bridson2013metric}
\[
B_H(X)=-2\bigl\langle \Lambda,\log \operatorname{diag}\bigl(\mathcal U(Q^\top XQ)\bigr)\bigr\rangle_F.
\]
Here $\mathcal U(X)$ denotes the reversed Cholesky factor, that is $\mathcal{U}(X) = U$ for $X = UU^\top$ where $U$ is an upper triangular matrix with positive diagonal, and $\operatorname{diag}(\cdot)$ extracts the diagonal part. Writing $U(X,H):=\operatorname{diag}(\mathcal U(Q^\top XQ))$, one obtains
\[
D(Z;\PP)=\inf_{H\in \bSS_p^1} \PP\!\left\{X\in\bSS_p^+:\bigl\langle \Lambda(H),\log U(X,H)-\log U(Z,H)\bigr\rangle_F\le 0\right\}.
\]
Equivalently,
\[
\DD^\alpha(\PP)=\bigcap_{H\in \bSS_p^1}\{Z\in\bSS_p^+:B_H(Z)\le t_H(\alpha)\}.
\]
The formula is more involved than in hyperbolic space, but the computational structure is the same: each direction produces one scalar score, and depth regions are obtained by intersecting the corresponding horoballs.
\end{example}


\begin{remark}\label{remark:compute}
The Busemann function is often cheaper to evaluate than the full geodesic distance. The distance $d(x,y)$ requires solving a two-point boundary-value problem or, equivalently, computing a Riemannian logarithm. By contrast, $B_\xi(x)$ keeps one endpoint at infinity fixed and typically reduces to a closed-form scalar functional. On $\bSS_p^+$, for example, evaluating $B_H(X)$ requires one eigendecomposition of $H$ and one Cholesky factorization of $Q^\top XQ$, whereas computing $d(X,Y)$ involves matrix square roots and logarithms. This matters because depth evaluation requires many repeated directional comparisons and never uses pairwise geodesic distances directly.
\end{remark}

\subsection{Approximating empirical depth regions}

For the empirical measure $\PP_n$, the threshold $t_\xi(\alpha)$ is the $\lceil n\alpha\rceil$-th largest value among $B_\xi(X_1),\dots,B_\xi(X_n)$. This suggests a direct approximation scheme: discretize the visual boundary, compute the one-dimensional directional scores, and intersect the corresponding finitely many horoballs.

\begin{algorithm}[t]
\footnotesize
\caption{Sampled-direction approximation of an empirical depth region}
\label{alg:H2}
\begin{algorithmic}[1]
\Require Data $X_1,\dots,X_n\in X$, depth level $\alpha$, sampled directions $\xi_1,\dots,\xi_m\in\partial X$
\Ensure Approximation of $\DD^\alpha(\PP_n)$
\State $k\gets \lceil n\alpha\rceil$
\For{$j=1,\dots,m$}
    \State Compute $s_i^{(j)}\gets B_{\xi_j}(X_i)$ for $i=1,\dots,n$
    \State Let $t_j$ be the $k$-th largest value among $\{s_i^{(j)}\}_{i=1}^n$
\EndFor
\State Define $F(z)\gets \max_{1\le j\le m}(B_{\xi_j}(z)-t_j)$
\State Return the sublevel set $\DD_{n,m}^\alpha=\{z\in X:F(z)\le 0\}$
\end{algorithmic}
\end{algorithm}

\begin{definition}[Sampled-direction approximation]    
    Fix $\alpha\in(0,1)$ and choose directions $\xi_1,\ldots,\xi_m\in\partial X$.  Let $k:=\lceil n\alpha\rceil$, and for each $j$ let $t_j$ be the $k$-th largest value among $B_{\xi_j}(X_1),\ldots,B_{\xi_j}(X_n)$.  The \emph{sampled-direction approximation} of the depth region is
    \[
    \mathcal{D}^\alpha_{n,m}:=\bigcap_{j=1}^m\{z\in X: B_{\xi_j}(z)\le t_j\}.
    \]
    Equivalently, defining $F(z):=\max_{1\le j\le m}(B_{\xi_j}(z)-t_j)$, one has $\mathcal{D}^\alpha_{n,m}=\{z:F(z)\le 0\}$.
\end{definition}

Since $\mathcal D^\alpha_{n,m}$ is the intersection of finitely many of the horoballs whose full intersection defines $\mathcal D^\alpha(\PP_n)$, one always has $\mathcal D^\alpha(\PP_n)\subset \mathcal D^\alpha_{n,m}$,
so the sampled-direction region is an outer approximation. Now fix a dense sequence $(\xi_j)_{j\ge1}\subset\partial X$ and define $\mathcal D^\alpha_{n,m}$ using the first $m$ directions. Then
\[
\mathcal D^\alpha_{n,m+1}\subset \mathcal D^\alpha_{n,m}
\quad\text{and}\quad
\bigcap_{m\ge1}\mathcal D^\alpha_{n,m}=\mathcal D^\alpha(\PP_n).
\]
Indeed, if $z\notin\mathcal D^\alpha(\PP_n)$, then by the attainment property (see the remark following Proposition \ref{prop:usc}) there exists $\xi^*\in\partial X$ such that $B_{\xi^*}(z)>t(\xi^*)$, where $t(\xi)$ denotes the $k$-th largest value among $B_\xi(X_1),\dots,B_\xi(X_n)$. By continuity of $\xi\mapsto B_\xi(z)$ and $\xi\mapsto t(\xi)$, this strict inequality persists for all $\xi$ sufficiently close to $\xi^*$, so $z\notin\mathcal D^\alpha_{n,m}$ for all sufficiently large $m$.

\begin{figure}[t]
    \centering
    \begin{subfigure}{0.30\columnwidth}
        \includegraphics[width=\columnwidth, trim={15cm 10cm 5cm 8.5cm},clip]{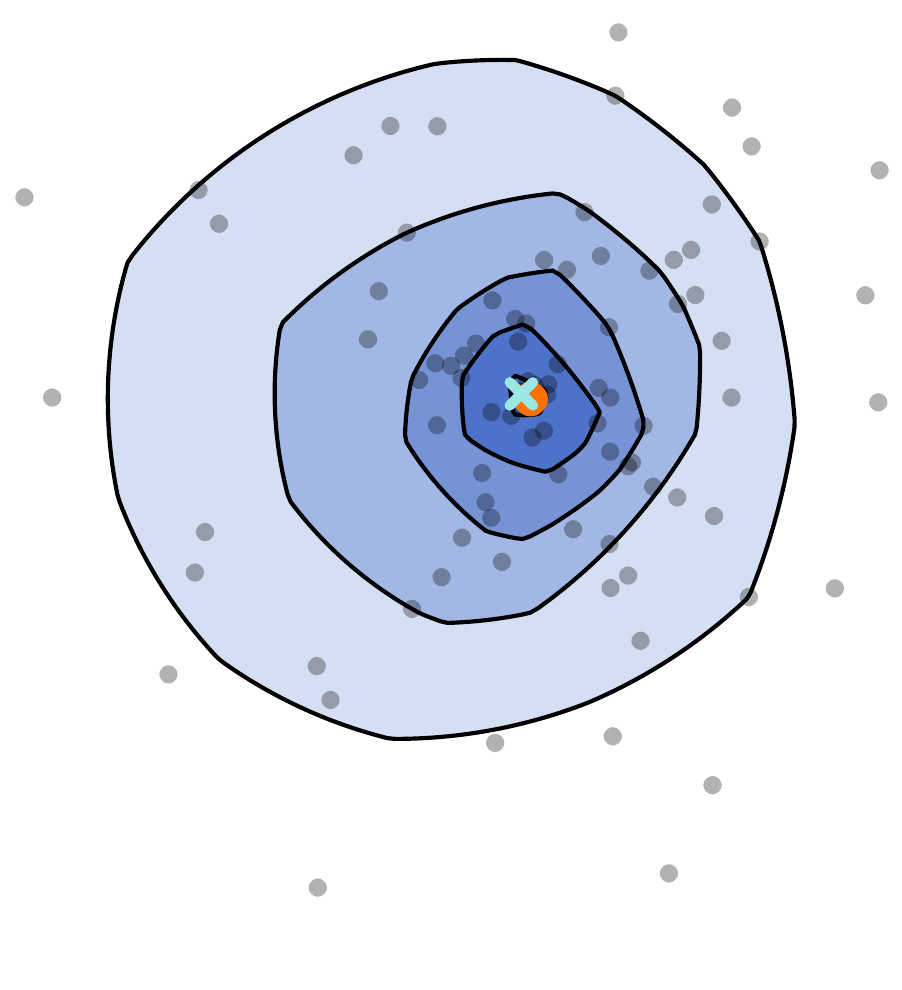}
          \caption{}
          \label{fig:h_sample_depth_region}
      \end{subfigure}
    \hspace{0.20\columnwidth}
    \begin{subfigure}{0.30\columnwidth}
        \includegraphics[width=\columnwidth, trim={15cm 10cm 5cm 8.5cm},clip]{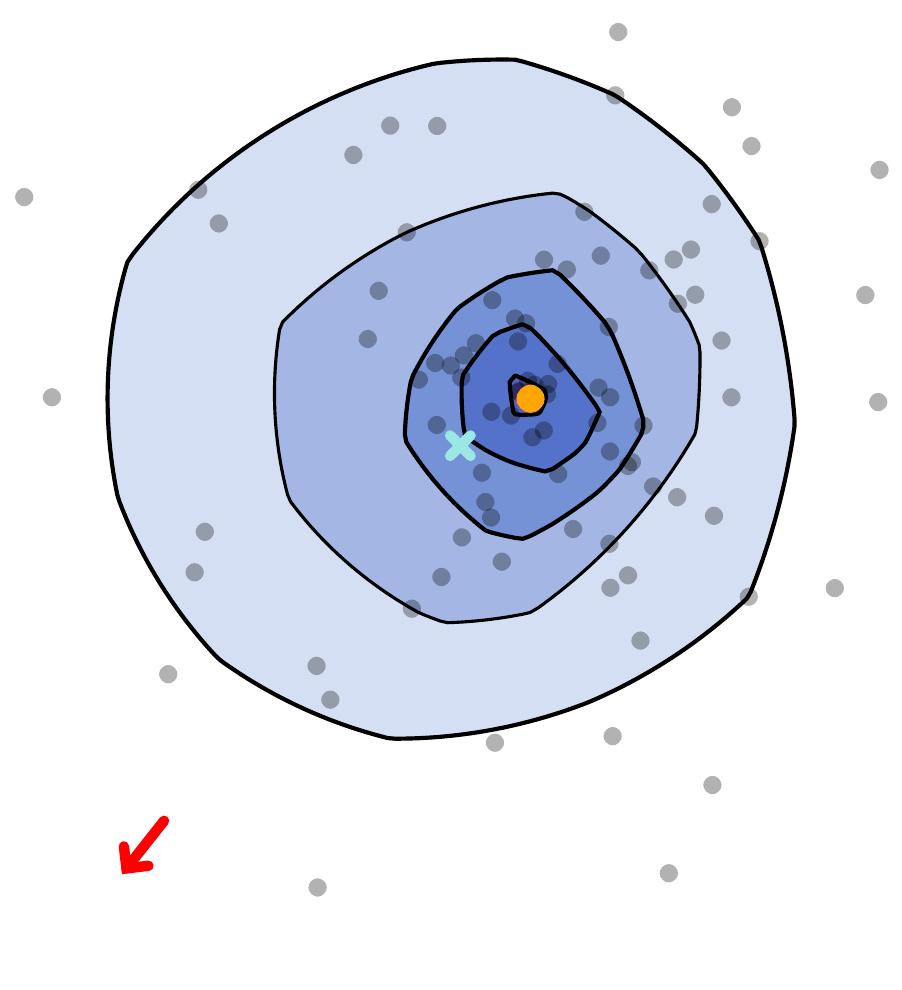}
          \caption{}
          \label{fig:h_sample_depth_region_2}
      \end{subfigure}
    \caption{Numerical illustrations of sampled-direction depth regions. Panel (a) shows zoomed-in empirical depth contours and the empirical Fr\'echet mean (marked by the $\times$ symbol) in $\mathbb H^2$ before the addition of a near-boundary outlier. Panel (b) shows the same construction after adding a near-boundary outlier; the outlier itself lies outside the displayed region, and its direction is indicated by the arrow.
    }
    \label{fig:hyper_sample_depth_region}
\end{figure}

\begin{example}[Numerical illustration on $\mathbb H^2$]\label{example_num_hyper}
\cref{fig:hyper_sample_depth_region} illustrates the sampled-direction approximation on one empirical sample of size $n=100$ in the Poincaré disc. We use $m=180$ sampled boundary directions and plot the approximate regions $D^\alpha_{n,m}$ for $\alpha\in\{0.1,0.2,0.3,0.4,0.5\}$. The center of the symmetry is plotted as the large (orange) circular dot in both panels. Panel (a) shows the uncontaminated sample. Panel (b) shows the same construction after adding a single near-boundary outlier in a fixed boundary direction; the outlier itself lies outside the displayed window, and, its direction is marked by the red arrow. For visual comparison, we also plot the empirical Fr\'echet mean (marked by the $\times$ symbol) computed numerically from the hyperbolic squared-distance objective. The point of the figure is qualitative: it illustrates the directional stability predicted by Section \ref{sec:robust}, not a benchmarking study or a definitive algorithmic claim.
\end{example}

\begin{example}[Numerical illustration on $\bSS_p^+$]
\cref{fig:spd_sample_depth_region} shows the same sampled-direction idea on $\bSS_2^+$ in the SPD cone coordinate visualization. The specific plotting coordinates are ancillary; the point is that one can evaluate $B_H(X_i)$ for finitely many boundary directions $H$, convert those directional scores into thresholds, and intersect the resulting horoballs. The anisotropy of the displayed region reflects the affine-invariant geometry rather than an ambient Euclidean notion of circularity. As in Example \ref{example_num_hyper}, we treat the figure as a schematic demonstration of the computational template, not as a
quantitative performance study.
\end{example}

\begin{figure}
    \centering
    \begin{subfigure}{0.44\columnwidth}
        \includegraphics[width = \columnwidth, trim={0cm 1.5cm 0cm 2cm},clip]{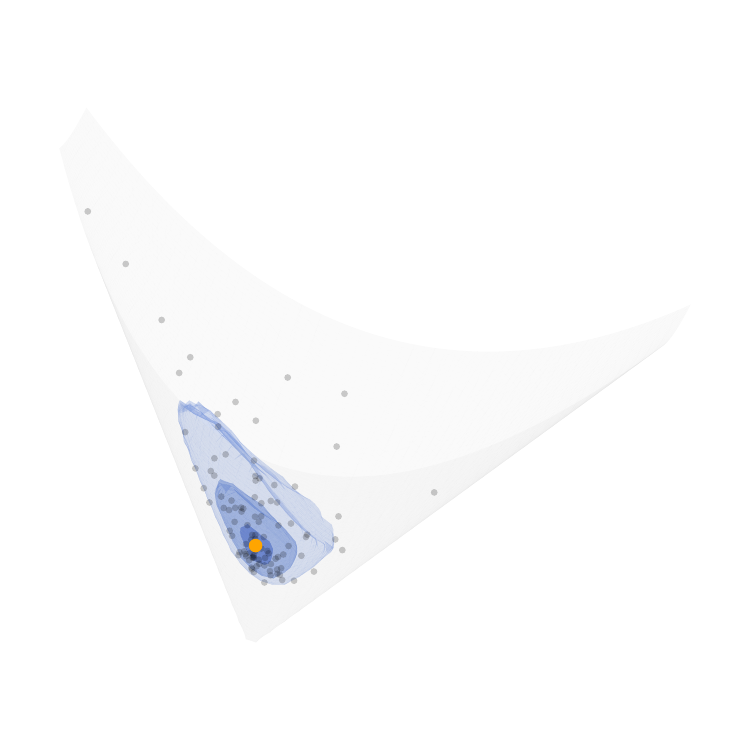}
        \caption{}
        \label{fig:spd_cone}
    \end{subfigure}
    \hspace{0.1\columnwidth}
    \begin{subfigure}{0.28\columnwidth}
        \includegraphics[width = \columnwidth, trim={3.5cm 0.5cm 4cm 6cm},clip]{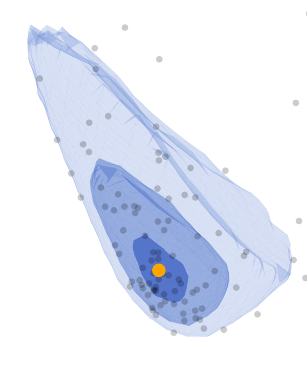}
        \caption{}
        \label{fig:spd_cone_zoomed}
    \end{subfigure}
    \caption{Illustrative sampled-direction depth regions in $\bSS_2^+$ for $\alpha\in\{0.05,0.15,0.3\}$. The blue shaded sets are approximate empirical depth regions in the SPD-cone coordinate visualization; the large orange circular dot marks the symmetric center used in the illustration. Panel (b) shows a magnified view of panel (a) around the symmetric center.}
    \label{fig:spd_sample_depth_region}
\end{figure}

\cref{fig:hyper_sample_depth_region} gives a visual illustration of the robustness theory: the depth contours remain qualitatively stable under contamination by points escaping in a fixed boundary direction, whereas the Fr\'echet mean is visibly affected. \cref{fig:spd_sample_depth_region} shows that the same finite-direction construction is not confined to hyperbolic geometry, but extends to settings in which boundary directions and the associated Busemann functions admit explicit coordinate formulas.

\subsection{Computing the sample Busemann median}

The exact sample Busemann median solves
\[
\max_{z\in X} D(z;\PP_n)=\max_{z\in X}\inf_{\xi\in\partial X}\frac1n\sum_{i=1}^n \mathbf 1\{B_\xi(z)\le B_\xi(X_i)\}.
\]
For fixed $z$, the inner problem is an optimization over the compact boundary $\partial X$. In practice, one replaces this infimum by either a finite direction set, as above, or a local search in a parametrization of the boundary. This yields the approximate depth
\[
D_m(z;\PP_n):=\min_{1\le j\le m}\frac1n\sum_{i=1}^n \mathbf 1\{B_{\xi_j}(z)\le B_{\xi_j}(X_i)\},
\]
which is a lower approximation to $D(z;\PP_n)$.

The numerical picture suggested by Sections \ref{sec:existence} and \ref{sec:uniqueness} is favorable. Vanishing at infinity restricts the search to a compact region, and in strictly negative curvature the population depth is strictly quasi-concave, so the deepest point is isolated. This suggests a coarse-to-fine strategy: first identify a compact high-depth region using sampled directions, then refine the maximizer inside that region by manifold optimization. A complete analysis of the discretization error in $m$ and of the optimization error for the sample median remains open, but the structural results of the paper indicate clearly where efficient algorithms should look.

\section{Discussion}
\label{sec:discussion}
The results establish a Tukey-type depth theory tailored to Hadamard manifolds. The restriction to this geometric setting is deliberate: by working with horoballs and the visual boundary, one recovers geodesically convex depth regions, a centerpoint theorem, curvature-driven uniqueness, and a notion of robustness to contamination escaping to infinity. The contribution is therefore not a universal depth notion for arbitrary object data, but an intrinsic depth theory for a statistically important class of nonpositively curved spaces.


\subsection{Relation to existing notions of depth}

Three points of comparison are especially relevant. The first is lens depth and its metric extensions, which are defined from balls and are available in much greater generality \citep{liu2011lens, cholaquidis2023weighted, cholaquidis2023level}. These methods are robust and flexible, but they are not built from halfspace-type objects and therefore do not naturally produce a centerpoint theorem. 
The second is the metric halfspace depth of \cite{dai2023tukey}. It is defined using metric halfspaces $H_{x_1,x_2}:=\{u\in X: d(u,x_1)\le d(u,x_2)\}$, and, for a query point $z\in X$, takes the infimum over all such halfspaces containing $z$. This construction is very general and yields nested compact depth regions under mild assumptions. However, unlike the present horospherical depth, it is not naturally parametrized by boundary directions alone. Its indexing family is given by anchor pairs $(x_1,x_2)$, or equivalently by direction-plus-level data, whereas horospherical depth is indexed by $\partial X$, the sphere of asymptotic directions. This directional parametrization is what makes horospherical depth closer in spirit to Euclidean Tukey depth and what allows centerpoint arguments.
The third is the tangent-space depth of \cite{rusciano2018riemannian}, which recovers a centerpoint theorem by linearization at a base point, at the cost of losing full intrinsic equivariance. The full comparison is summarized in Table \ref{table:compare}. Entries in the first two comparison columns record what is explicitly proved in the cited papers, rather than all properties that may hold under additional assumptions.

The horospherical depth occupies a middle ground. It is less general than the metric construction of \cite{dai2023tukey}, but more intrinsic than tangent-space linearization. The gain from specializing to Hadamard manifolds is exactly what the paper has used throughout: geodesically convex depth regions, an intrinsic centerpoint theorem, curvature-driven uniqueness, and a meaningful notion of contamination escaping to the visual boundary.

\begin{table}
\caption{Comparison with metric halfspace depth \cite{dai2023tukey}, tangent-space depth \cite{rusciano2018riemannian}, and horospherical depth.}
\begin{center}
\resizebox{\columnwidth}{!}{%
\begin{tabular}{llll}
\toprule
\textbf{Feature} & \textbf{Metric halfspace} & \textbf{Tangent-space depth} & \textbf{Horospherical depth} \\
\midrule
Underlying space & Metric spaces & Riemannian manifolds & Hadamard manifolds \\
Intrinsic/isometry-equivariant & Yes & Depends on a chosen base point & Yes \\
Convex depth regions & Not shown in \cite{dai2023tukey} & Not shown in \cite{rusciano2018riemannian} & Yes \\
Centerpoint theorem & Not proved in \cite{dai2023tukey} & Yes & Yes \\
Uniqueness from curvature & Not addressed in \cite{dai2023tukey} & Not addressed in \cite{rusciano2018riemannian} & Yes ($\Sec_X<0$) \\
Boundary robustness & Not addressed in \cite{dai2023tukey} & Not addressed in \cite{rusciano2018riemannian} & Yes \\
\bottomrule
\end{tabular}
}
\end{center}
\label{table:compare}
\end{table}

With that interpretation, the scope of the present paper is clear. The contribution is not a universal depth notion for arbitrary metric data, but an intrinsic depth notion tuned to Hadamard geometry and therefore strong enough to recover a large part of the classical Tukey picture in a curved setting.

\subsection{Rank one, higher rank, and flats}

The strict quasi-concavity theorem in Section \ref{sec:uniqueness} uses the hypothesis $\Sec_X<0$. This is satisfied by rank-one spaces such as $\mathbb H^d$, but not by higher-rank symmetric spaces such as $\bSS_p^+$, where flat totally geodesic submanifolds are unavoidable. On those spaces, the depth regions remain geodesically convex and the existence, robustness, and consistency theory remain intact, but strict uniqueness of the median can fail along flat directions.

This is not merely a technical annoyance. It points to the real geometric dividing line in the theory: not between Euclidean and non-Euclidean spaces, but between spaces with uniform negative curvature and spaces containing flats. An important next step is therefore to replace the global condition $\Sec_X<0$ by a condition that rules out concentration of the probability measure along flat geodesic submanifolds. Such a result would likely recover uniqueness for generic measures on higher-rank symmetric spaces while respecting the intrinsic geometry.

On a higher-rank symmetric space, the visual boundary \(\partial_\infty X\) is a spherical building whose points are stratified by the face types of the spherical Weyl chamber. The Furstenberg boundary \(\partial_F X\) is the space of chambers at infinity; it is closely related to, but not identical with, the open dense set of regular points in \(\partial_\infty X\). It is \(\partial_F X\) that plays the central role in harmonic analysis and the Poisson transform. This suggests a possible variant of Definition \ref{def:depth} in higher rank: instead of taking the infimum over all visual boundary points, one could restrict to a distinguished family of regular directions, for instance one representative of fixed type in each chamber at infinity. Such a restriction would amount to intersecting a smaller family of horoballs, and hence would generally produce larger depth values and larger depth regions. Whether this yields a better-behaved notion of depth is unclear. In particular, one should not expect strict quasi-concavity to follow automatically from excluding singular directions, since even for regular \(\xi\), the Busemann function \(B_\xi\) remains affine along maximal flats asymptotic to \(\xi\). Whether such a regular-direction or Furstenberg-type depth retains useful properties, such as a centerpoint theorem, appears to be an interesting open question.

\subsection{Open problems}

\smallskip
\noindent
\textbf{Higher-rank uniqueness.} Determine conditions under which uniqueness of the Busemann median holds on Hadamard manifolds with flats. The natural conjecture is that one does not need strict negative curvature everywhere; it should suffice to exclude the directions along which the measure can produce a flat maximum.

\smallskip
\noindent
\textbf{Rates and limit theory.} Theorems \ref{thm:sample_converg} and \ref{thm:sample_converg_2} establish uniform almost sure convergence, but a sharper asymptotic theory remains open. On symmetric spaces the VC bound already suggests nonasymptotic deviation inequalities, and in the unique-median regime one expects corresponding argmax rates. A functional central limit theorem for $z\mapsto D(z;\PP_n)$ would be the natural next step.

\smallskip
\noindent
\textbf{Discretization error and algorithms.} Section \ref{sec:numerical} used a finite set of boundary directions to approximate both depth regions and medians. A quantitative theory should relate the number of sampled directions $m$ to the approximation error in $D_m(z;\PP_n)$ and in the resulting contours. This is the algorithmic counterpart of the statistical convergence theory developed in Section \ref{sec:computation}.

\smallskip
\noindent
\textbf{Beyond smooth Hadamard manifolds.} Busemann functions and horoballs are available on proper CAT$(0)$ spaces, including spaces such as the Billera--Holmes--Vogtmann tree space \citep{billera2001geometry}. Extending the existence and centerpoint theory to that level of generality is plausible; extending the strict-convexity and empirical-process arguments is substantially subtler and would require new tools.

\smallskip
\noindent
\textbf{Optimal breakdown.} Section \ref{sec:robust} proves the lower bound $\varepsilon_{\mathrm{bd}}(\PP)\ge D_*(\PP)/(1+D_*(\PP))$, hence in particular $\varepsilon_{\mathrm{bd}}(\PP)\ge 1/(d+2)$. It remains open whether, under suitable regularity assumptions, the exact breakdown point matches the Euclidean benchmark $1/(d+1)$.

These questions indicate that the present paper is best viewed as a first structural theory. It identifies the correct halfspace analogue on Hadamard manifolds, proves the core geometric and statistical properties, and leaves a fairly clear map of what remains to be done.

\bibliographystyle{imsart-number} 
\bibliography{bib}       


\appendix

\section{Useful Lemmas} 


\begin{lemma}\label{lemma:compact_uniform_convergence_busemann}
    Let $X$ be a finite-dimensional Hadamard manifold. If $\xi_n\to \xi$ in $\partial X$, then $B_{\xi_n}\to B_\xi$ uniformly on compact subsets of $X$. In particular, $\xi \mapsto B_\xi$ is continuous. 
\end{lemma}
\begin{proof}
    Since, for a Hadamard manifold, the visual boundary are homeomorphic to the gromov boundary via the Busemann map \cite[Chapter II.8]{bridson2013metric}, the map $\xi \mapsto B_\xi$ is continuous when the space of Busemann functions is endowed with the compact-open topology; in particular, if $\xi_n \to \xi$ in $\partial X$, then $B_{\xi_n}\to B_\xi$ uniformly on compact subsets of $X$.
\end{proof}

\begin{lemma}\label{lemma:p_usc}
For any probability measure $\PP$ on a Hadamard manifold $X$ that satisfies 
Assumption 1
, the map
\[
\partial X\to[0,1],\qquad \xi\mapsto p_\xi(z),
\]
is lower semicontinuous.\footnote{Recall $p_\xi(z):= \PP(H_{\xi, z}^+)$ with $H^+_{\xi, z}:= \{x \in X: B_\xi(x) \ge B_\xi(z)\}$.} 
\end{lemma}

\begin{proof}

Let $\xi_n\to\xi$ in $\partial X$, and set
\[
h_n(u):=B_{\xi_n}(u)-B_{\xi_n}(z),
\qquad
h(u):=B_\xi(u)-B_\xi(z).
\]
By Lemma \ref{lemma:compact_uniform_convergence_busemann}, $h_n\to h$ uniformly on compact subsets of $X$.
Fix $\varepsilon,\delta>0$, and choose $R>0$ such that $\PP\bigl(X\setminus B(o,R)\bigr)<\delta$. For all large $n$, we have $|h_n-h|<\varepsilon$ on $B(o,R)$; hence
\[
\{h>\varepsilon\}\cap B(o,R)\subseteq \{h_n\ge 0\}.
\]
Therefore,
\[
\liminf_{n\to\infty} p_{\xi_n}(z)
\ge
\PP(\{h>\varepsilon\}\cap B(o,R))
\ge
\PP(h>\varepsilon)-\delta.
\]
Letting first $\delta\downarrow 0$ and then $\varepsilon\downarrow 0$, and using 
Assumption 1
to get
\[
\PP(h=0)=\PP\bigl(\{u:B_\xi(u)=B_\xi(z)\}\bigr)=0,
\]
we obtain $\liminf_{n\to\infty} p_{\xi_n}(z)\ge p_\xi(z)$ as required.
\end{proof}

\begin{lemma}\label{lemma:threshold-value}
    Under 
    Assumption \ref{assump:no_atom}
    , $S_\xi(t_\xi(\alpha)) = \alpha$ for every $\xi$ and $\alpha \in (0,1)$.\footnote{Recall that $S_\xi(t):= \PP(\{x \in X: B_\xi(x) \ge t\})$.}
\end{lemma}

\begin{proof}
    Under 
    Assumption \ref{assump:no_atom}
    , $S_\xi$ is continuous. For every $t < t_\xi(\alpha)$, we have $S_\xi(t) \ge \alpha$ by 
    definition (5)
    . $S_\xi(t_\xi(\alpha)) = \lim_{t \uparrow t_\xi(\alpha)} S_\xi(t) \ge \alpha$. Similarly, we must have $S_\xi(t_\xi(\alpha)) \le \alpha$. Suppose for contradiction, we have $S_\xi(t_\xi(\alpha)) > \alpha$. By continuity and monotonicity of $S_\xi$, there exists $\varepsilon > 0$ such that $S_\xi(t_\xi(\alpha) + \varepsilon) > \alpha$. It follows $t_\xi(\alpha) + \varepsilon \in \{t: S_\xi(t) \ge \alpha\}$, contradiction. Thus, we have $S_\xi(t_\xi(\alpha)) \le \alpha$.
    
\end{proof}



\begin{lemma}\label{lemma:slab}
Under 
Assumption \ref{assump:no_atom}
, for any $\xi\in\partial X$ and $z\in X$ with $B_\xi(z)<t_\xi(\alpha)$ and $\alpha\in(0,1)$:
\begin{equation}\label{eq:slab}
p_\xi(z) = \alpha + \PP\bigl(U_\xi(z,\alpha))
\end{equation}
where $U_\xi(z,\alpha):=\{x\in X : B_\xi(z)<B_\xi(x)<t_\xi(\alpha)\}$. Furthermore, $U_\xi(z,\alpha)$ is the preimage of the open interval $(B_\xi(z),t_\xi(\alpha))$ under the continuous submersion $B_\xi$, hence is a nonempty open subset of $X$.
\end{lemma}

\begin{proof}
First, we write,
\begin{align*}
p_\xi(z) &= \PP(\{x: B_\xi(x)\ge B_\xi(z)\}) \\
&= \PP(\{x: B_\xi(X)\ge t_\xi(\alpha)\}) + \PP(\{x: B_\xi(z)\le B_\xi(x) < t_\xi(\alpha)\}).\\
&= S_\xi(t_\xi(\alpha)) + \PP(\{x: B_\xi(z)\le B_\xi(x) < t_\xi(\alpha)\}).
\end{align*}
The first term equals $\alpha$ by Lemma \ref{lemma:threshold-value}.  Under 
Assumption \ref{assump:no_atom}
, $\PP(\{x: B_\xi(x)=B_\xi(z)\})=0$, so the second term equals $\PP(\{x: B_\xi(z)<B_\xi(x)<t_\xi(\alpha)\})$.

For openness: on Hadamard manifold, $B_\xi$ is $C^2$ with $\|\grad B_\xi\|\equiv 1$ 
(Proposition \ref{prop:cont})
, so it is a submersion from $X$ to $\R$. The preimage of the nonempty open interval $(B_\xi(z),t_\xi(\alpha))\subset\R$ under a submersion is open.  It is nonempty because $B_\xi$ is surjective onto $\R$ since $B_\xi(\gamma(t)) = -t \to-\infty$ along any ray $\gamma$ toward $\xi$, and $B_\xi\to+\infty$ in the opposite direction.
\end{proof}

\begin{lemma}\label{lemma:slab-positive}
Under Assumption \ref{assump:no_atom} and Assumption \ref{assump_connected_supp}
, for any $\xi \in \partial X$ and $z \in X$, if $B_\xi(z)<t_\xi(\alpha)$ for some $\alpha\in(0,1)$, then $p_\xi(z) > \alpha$.
\end{lemma}

\begin{proof}
We note that under Assumption \ref{assump_connected_supp} 
and Lemma \ref{lemma:slab}, $p_\xi(z) > \alpha$ if $B_\xi(z)$ and $t_\xi(\alpha)$ are both in $\in \overline{B_\xi(\supp(\PP))}$. Denote $\supp_\xi:=\overline{B_\xi(\supp(\PP))}$. By Lemma \ref{lemma:threshold-value}, $S_\xi(t_\xi(\alpha)) = \alpha$. For $\alpha \in (0,1)$, there is mass both above and below $t_\xi(\alpha)$. Thus, $t_\xi(\alpha) \in \supp_\xi$. 

Since $\supp_\xi$ is a closed interval, and $B_\xi(z) < t_\xi(\alpha) \in \supp_\xi$, we have either $B_\xi(z) \in \supp_\xi$ or $B_\xi(z) < \inf \supp_\xi$. If $B_\xi(z) \in \supp_\xi$, then by Lemma \ref{lemma:slab}, we have $p_\xi(z) > \alpha$. If $B_\xi(z) < \inf \supp_\xi$, then $p_\xi(z) = S_\xi(B_\xi(z)) = 1 > \alpha$.






\end{proof}


\begin{lemma}\label{lemma:grad_bijection}
    Let $\xi \in \partial X$ and denote $S_zX$ as the unit sphere on the tangent space $T_zX$. The map, $\partial X \to S_zX$, $\xi \mapsto \grad_z B_\xi$ is a bijection. 
\end{lemma}
\begin{proof}
    First, $\|\grad B_\xi\|\equiv 1$ by 
    Proposition \ref{prop:cont}.
    \textbf{Surjectivity}: For any $v \in S_zX$, the geodesic ray $\gamma_v(t) = \exp_z(-tv)$ define a boundary point $\xi_v= \gamma_{v}(+\infty) \in \partial X$, by 
    Proposition \ref{prop:cont},
    $\grad_z B_{\xi_v} = v$. 
    \textbf{Injectivity}: If $\grad_z B_\xi = \grad_z B_\eta$, then $\dot\gamma_{z, \xi} = \dot\gamma_{z, \eta}$ by 
    Proposition \ref{prop:cont},
    so $\gamma_{z, \xi} = \gamma_{z, \eta}$ since they share the same starting point, we must have $\xi = \eta$. 
\end{proof}

\begin{lemma}\label{lemma:g-Busemann}
    Let $\varphi: X \to X$ be an isometry, there is a boundary isometry $\partial \varphi: \partial X \to \partial X$ such that for each $\xi \in \partial X$, we have any all $x \in X$,  
    \begin{equation*}
        B_{\partial \varphi(\xi)}(\varphi(x)) = B_{\xi}(x) - B_\xi(\varphi^{-1}(o)).
    \end{equation*}
\end{lemma}
\begin{proof}
An isometry $\varphi$ maps geodesic rays to geodesic rays and thus induces a homeomorphism $\partial \varphi:\partial X\to\partial X$ by sending $\xi=\gamma(\infty)$ to $(\varphi\circ\gamma)(\infty)$. Fix $\xi\in\partial X$ and let $\gamma:[0,\infty)\to X$ be the geodesic ray from $o$ with $\gamma(\infty)=\xi$. Then $\varphi \circ \gamma$ is a geodesic ray with endpoint $(\varphi \circ \gamma)(\infty) = \partial \varphi(\xi)$, but based at $(\varphi\circ\gamma)(0) = \varphi(o)$ rather than $o$.

Let $\widetilde\gamma:[0,\infty)\to X$ be the geodesic ray based at $o$ with $\widetilde\gamma(\infty)=\partial \varphi(\xi)$. Since $\widetilde\gamma$ and $\varphi \circ\gamma$ have the same endpoint at infinity, they are asymptotic. 
By Proposition \ref{prop:busemann_asymptotic},
we have
\[
B_{\partial \varphi(\xi)}(y) = B_{\tilde \gamma}(y) = B_{\varphi \circ \gamma} (y) - B_{\varphi \circ \gamma}(\tilde{\gamma}(0)) = B_{\varphi \circ \gamma} (y) - B_{\varphi \circ \gamma}(o). 
\]
Now set $y=\varphi(x)$ and use that $\varphi$ is an isometry:
\begin{align*}
    B_{\varphi \circ \gamma} (y) &= \lim_{t\to\infty}\bigl(d(y,\varphi(\gamma(t))) - t \bigr) = \lim_{t\to\infty}\bigl(d(\varphi(x),\varphi(\gamma(t))) - t \bigr) \\
    &= \lim_{t\to\infty}\bigl(d(x,\gamma(t))-t\bigr)
    = B_\gamma(x) = B_\xi(x).    
\end{align*}
Similarly, we have, 
\begin{align*}
    B_{\varphi \circ \gamma}(o) &= \lim_{t\to\infty}\bigl(d(o,\varphi(\gamma(t))) - t \bigr) = \lim_{t\to\infty}\bigl(d(\varphi^{-1}(o),\gamma(t)) - t \bigr) \\
    &= B_\gamma(\varphi^{-1}(o)) = B_\xi(\varphi^{-1}(o)).
\end{align*}
Therefore, we have $B_{\partial \varphi(\xi)}(\varphi(x))=B_\xi(x) - B_\xi(\varphi^{-1}(o))$ for all $x \in X$ as required.
\end{proof}

\begin{lemma}\label{lemma:uppersemicont}
    Let $\F$ be an arbitrary family of upper semicontinuous functions on a given topological space $X$. Then the function $F(x) = \inf_{f\in\F} f(x)$ is upper semicontinuous.
\end{lemma}

\begin{proof}
    It's sufficient to show that $F^{-1}((-\infty, c))$ is open for any $c \in \R$.
    First, note that, $F(x) < c$ if and only if there exist $f \in \F$, such that $f(x) < c$. Thus, 
    \[
    F^{-1}((-\infty,c)) = \bigcup_{f \in \F}f^{-1}((-\infty, c)).
    \]
    Since $f \in \F$ is upper continuous, $f^{-1}((-\infty, c))$ is open. The result follows from that any arbitrary union of open sets is open.
\end{proof}


\begin{lemma}\label{lemma:inf_sup}
    $|\inf_x f(x) - \inf_x g(x)| \le \sup_x |f(x) - g(x)|$. 
\end{lemma}

\begin{proof}
    For any $x$, we have
    \begin{align*}
        \inf_x f(x) &\le f(x) \le g(x) + |f(x) - g(x)|\le g(x) + \sup_x |f(x) - g(x)|
    \end{align*}
    Thus, we have $\inf_x f(x)\le \inf_x g(x) + \sup_x |f(x) - g(x)|$. Swapping $f$ and $g$, and combine the two inequalities. 
\end{proof}

\section{Proofs in Section \ref{sec:background} and \ref{sec:definition}}

\subsection{Proof of Proposition \ref{prop:busemann_asymptotic}}
\begin{proof}
    We will show that $\|\grad (B_\gamma - B_{\gamma'})\| \equiv 0$, which implies $B_\gamma - B_\gamma' \equiv C$ for some constant $C$. The expression is then obtained by plugging in $x := \gamma(0)$.

    By \cite[Proposition 3.1]{heintze1977geometry}, given $\xi \in \partial X$, we have $-\grad B_\gamma = Z$ where $Z$ is the radial  field in the direction of $\xi$. Recall that given $\xi \in \partial X$, the radial field $Z$ is defined as $Z(p) = \dot\gamma_{p, \xi}$, where $\gamma_{p, \xi}$ denotes the geodesic connecting $p$ and $\xi$. Note that the radial field $Z$ is uniquely determined by $\xi$.  Therefore, for two geodesic rays $\gamma, \gamma'$ with $\xi = \gamma(\infty) = \gamma'(\infty)$, we have $\grad_p B_\gamma = -Z(p) =  \grad_p B_{\gamma'}$ for all $p \in X$.
\end{proof}

\subsection{Proof of Theorem \ref{lemma:superlevel}}
\begin{proof}
Fix $\alpha \in (0,1)$ and $z\in X$. By 
Definition \ref{def:depth}
of the horospherical depth,
\[
D(z; \PP)\ge \alpha
\quad\iff\quad
S_\xi(B_\xi(z))\ge \alpha \quad \forall \xi\in\partial X.
\]
For fixed $\xi$, by 
\eqref{eq:t_xi},
the condition $S_\xi(B_\xi(z))\ge\alpha$ is equivalent to $B_\xi(z)\le t_\xi(\alpha)$. Taking the conjunction over all $\xi$ yields 
\eqref{eq:superlevel}. 
Each set $H_\xi(t_\xi(\alpha)) = \{z : B_\xi(z) \leq t_\xi(\alpha)\}$ is a horoball, hence geodesically convex 
(Proposition \ref{prop:cont}).
The intersection of geodesically convex sets is geodesically convex, so $\DD^\alpha(\PP)$ is geodesically convex.
\end{proof}

\subsection{Proof of Proposition \ref{prop:depth_non_equivalence}}

\begin{proof}
    By \cite[Proposition 3.1]{heintze1977geometry}, $B_\xi = -B_\zeta + C$ implies that there exists two radial fields $Z_\xi, Z_\zeta$ in the direction of $\xi$ and $\zeta$, respectively such that $Z_\xi \equiv -Z_\zeta$. Equivalently, $\dot\gamma_{p, \xi} = -\dot\gamma_{p,\zeta}$ for all $p \in X$. Combine these two rays together we have $\gamma_p: \R \to X$ with $\gamma_p(+\infty) = \xi$ and $\gamma_p(-\infty) = \zeta$. Let $p, q \in X$, then the geodesic $\gamma_p$ and $\gamma_q$ are bi-asymptotic (asymptotic in both directions). By the Flat strip theorem \cite[Chapter II.2, Theorem 2.13]{bridson2013metric}, the convex hull of $\gamma_p(\R) \cup\gamma_q(\R)$ is isometric to a flat strip $\R \times [0, d(p, q)]$ in Euclidean space. Since these hold for any pair of points $p$ and $q$, it forces the sectional curvature to vanish everywhere.  
\end{proof}

\subsection{Proof of Proposition \ref{prop:horo_strip}}

\begin{proof}
    Denote $F_\xi(t): = \PP(\{x \in X: B_\xi(x) \le t\})$.
    
    Fix $\alpha \in (0,1)$ and $z\in X$. 
    By Definition \eqref{eq:depth_two_sided} 
    of the two-sided depth,
    \[
    \widetilde D(z; \PP)\ge \alpha
    \quad\iff\quad
    \min\{S_\xi(B_\xi(z)), \, F_\xi(B_\xi(z))\} \ge \alpha \quad \forall \xi\in\partial X.
    \]
    For fixed $\xi$, 
    by \eqref{eq:t_xi} and \eqref{eq:q_xi}, 
    the condition $\min\{S_\xi(B_\xi(z)), \, F_\xi(B_\xi(z))\} \ge \alpha$ is equivalent to $B_\xi(z)\le t_\xi(\alpha)$ and $B_\xi(z) \ge q_\xi(\alpha)$. Taking the conjunction over all $\xi$ yields 
    \eqref{eq:strip}.
\end{proof}


\section{Proofs in Section \ref{sec:properties}}

\subsection{Proof of Theorem \ref{lemma:g-depth}}

\begin{proof}

\medskip
\noindent
\textbf{Step 1: $\varphi$-invariance of depth.} First, we note the following, 
    \begin{align*}
        H^+_{\xi, \varphi(z)} &= \{x \in X: B_\xi(\varphi(z)) \le B_\xi(x)\} = \{\varphi(x) \in X: B_\xi(\varphi(z)) \le B_\xi(\varphi(x))\} \\
        &= \varphi(\{x \in X: B_\xi(\varphi(z)) \le B_\xi(\varphi(x))\}) \\
        &= \varphi(\{x \in X: B_\xi(\varphi(z)) - B_\xi(\varphi^{-1}(o)) \le B_\xi(\varphi(x)) - B_\xi(\varphi^{-1}(o))\}) \\
        &\stackrel{(1)}{=} \varphi(\{x \in X: B_{\partial \varphi(\xi)}(z) \le B_{\partial \varphi(\xi)}(x)\}) = \varphi(H^+_{\partial \varphi(\xi), z}). 
    \end{align*}
    where $(1)$ follows from Lemma \ref{lemma:g-Busemann}. Thus, we have,
    \begin{align*}
        D(\varphi(z); \varphi_\#\PP) :=& \inf_{\xi \in \partial X}(\varphi_\#\PP)(H^+_{\xi, \varphi(z)}) = \inf_{\xi \in \partial X} \PP(\varphi^{-1}(H^+_{\xi, \varphi(z)})) = \inf_{\xi \in \partial X} \PP(H^+_{\partial \varphi(\xi), z}) \\
        \stackrel{(2)}{=}& \inf_{\xi \in \partial X} \PP(H^+_{\xi, z}) = D(z; \PP),
    \end{align*}
    where $(2)$ follows from the fact that $\partial \varphi$ is a bijection. 

\medskip
\noindent
\textbf{Step 2: $\varphi$-equivalence of depth region and Busemann median.}
    \begin{align*}
        \DD^\alpha(\varphi_\#\PP) &= \{z \in X: D(z; \varphi_\#\PP) \ge \alpha \} = \{\varphi(z) \in X: D(\varphi(z); \varphi_\#\PP) \ge \alpha \} \\
        &= \varphi\{z \in X: D(\varphi(z); \varphi_\#\PP)\ge \alpha\}\\  &\stackrel{(1)}{=} \varphi \{z \in X: D(z; \PP) \ge \alpha\} = \varphi(\DD^\alpha(\PP)),
    \end{align*}
    where  $(1)$ follows from Step 1. In particular, we have $\mu_*(\varphi_\#\PP) = \varphi(\mu_*(\PP))$.
\end{proof}

\subsection{Proof of Proposition \ref{prop:symmetric_prob}}

\begin{proof}
Let $s_\theta:X\to X$ be the geodesic symmetry reflection on $\theta$, 
\[
s_\theta(x) = \exp_\theta(- \exp^{-1}_\theta(x)),
\] 
and by central symmetry, $(s_\theta)_\#\PP=\PP$. Fix $\xi\in\partial X$. Define $f_\xi(x):=B_\xi(x)-B_\xi(\theta)$, then $H_{\xi,\theta}^+ =\{x\in X: f_\xi(x)\ge 0\}$. Since $B_\xi$ is geodesically convex, so is $f_\xi$. Now for any $x\in X$, the point $\theta$ is the midpoint of the geodesic segment joining $x$ and $s_\theta(x)$, because $s_\theta$ is the geodesic reflection at $\theta$. Therefore, by geodesic convexity of $f_\xi$,
\[
f_\xi(\theta)\le \frac{1}{2}\bigl(f_\xi(x)+f_\xi(s_\theta(x))\bigr).
\]
But $f_\xi(\theta)=0$, so $f_\xi(x)+f_\xi(s_\theta(x))\ge 0$. Hence, if $x\not\in H_{\xi,\theta}^+$, i.e. $f_\xi(x)<0$, then necessarily $f_\xi(s_\theta(x))>0$, which implies $s_\theta(x)\in H_{\xi,\theta}^+$. Thus $X = H_{\xi,\theta}^+ \cup s_\theta^{-1}(H_{\xi,\theta}^+)$. Taking $\PP$-probabilities and using $(s_\theta)_\#\PP=\PP$, we obtain
\[
1
\le \PP(H_{\xi,\theta}^+) + \PP\bigl(s_\theta^{-1}(H_{\xi,\theta}^+)\bigr)
= \PP(H_{\xi,\theta}^+) + \PP(H_{\xi,\theta}^+)
= 2\PP(H_{\xi,\theta}^+).
\]
Therefore $\PP(H_{\xi,\theta}^+)\ge 1/2$. Since $\xi\in\partial X$ was arbitrary, $\PP$ is horospherically symmetric about $\theta$.
\end{proof}

\subsection{Proof of Proposition \ref{lemma:maximality}}

\begin{proof}
By definition, $D(\theta;\PP)=\inf_{\xi\in\partial X}\PP(H^+_{\xi,\theta})\ge 1/2$.
If $z\neq\theta$, the uniqueness of the symmetry center implies that there exists some $\xi\in\partial X$ for which $\PP(H^+_{\xi,z})<1/2$. Hence
\[
D(z;\PP)=\inf_{\xi\in\partial X}\PP(H^+_{\xi,z})<\frac12\le D(\theta;\PP),
\]
so $\theta$ is the unique maximizer.
\end{proof}

\section{Proofs in Section \ref{sec:existence}}

\subsection{Proof of Proposition \ref{prop:no_infty}}

First we need the following lemma.
\begin{lemma}\label{assump:tight1}
    Any Borel probability measure $\PP$ satisfies the following condition: there exists $\alpha(s) \to 0$ as $s \to \infty$, such that
    \begin{equation*}
        \sup_{\xi}\PP(\{x \in X: s \le B_{\xi}(x)\}) \le \alpha(s).
    \end{equation*}
\end{lemma}
\begin{proof}
    Since $B_\xi(x)$ is 1-Lipschitz and $B_\xi(o) = 0$, we have$|B_\xi(x)| \le d(o, x)$. In particular, $B_\xi(x) \le d(o, x)$ and thus we have, 
\begin{align*}
    \sup_{\xi}\PP(\{x \in X: s \le B_{\xi}(x)\}) \le \PP(\{x \in X: s \le d(o, x)\}).
\end{align*}
Therefore, a sufficient condition is the following standard tightness condition, 
    \begin{equation*}
        \PP(\{x \in X: s \le d(o, x)\}) \to 0 \quad \text{as} \quad s \to \infty,
    \end{equation*}
which holds for any Borel probability measure $\PP$ due to the continuity of measure from above. 
\end{proof}

\begin{proof}[Proof of Proposition \ref{prop:no_infty}]
    Fix $z \in X$, denote $\xi_z:=\gamma(+\infty) \in \partial X$ given by the geodesic ray $\gamma$ from origin $o$ and pass through $z$. That is, $\gamma(s) = z$ for some $s \in \R$. WLOG we assume $s\ge 0$, and it follows $d(o, z) = s$. Denote the opposite ray $\gamma^-(t) := \widetilde{\gamma}(-t)$ where $\widetilde\gamma:\R \to X$ is the complete extension of $\gamma$, denote the corresponding boundary point as $\xi^-_z$, we have
    \begin{align*}
        B_{\xi^-_z}(z) &= \lim_{t \to \infty} (d(z, \gamma^-(t))-t) =\lim_{t \to \infty} (d(\gamma(s), \gamma^-(t))-t) \\ &=\lim_{t \to \infty} (d(\widetilde \gamma(s), \widetilde \gamma(-t))-t) = s.
    \end{align*}
    Thus, we have 
    \begin{align*}
        H_{\xi^-_z, z}^+ = \{x \in X: B_{\xi^-_z}(z) \le B_{\xi^-_z}(x)\} = \{x \in X: s \le B_{\xi^-_z}(x)\}.
    \end{align*}
    It follows that,
    \begin{align*}
        D(z; \PP) &= \inf_\xi \PP(H^+_{\xi, z}) \le \PP(H^+_{\xi^-_z, z}) = \PP(\{x \in X: s \le B_{\xi^-_z}(x)\}) \\
        &\le \sup_{\xi}\PP(\{x \in X: s \le B_{\xi}(x)\}) \le \alpha(s).
    \end{align*}
    where the last inequality follows from Lemma \ref{assump:tight1}. Since $\alpha(s) \to 0$ as $s \to \infty$, we have $D(z; \PP) \to 0$ as $d(o,z) \to \infty$.
\end{proof}

\subsection{Proof of Proposition \ref{prop:usc}}

\begin{proof}
    Denote $f_{\xi}(z, \PP):=\PP(H^+_{\xi,z})$. By Lemma \ref{lemma:uppersemicont}, it's sufficient to show that $f_{\xi}$ is upper semicontinuous (u.s.c) for any $\xi \in \partial X$. 
    That is, we need to show, given $\xi \in \partial X$, for any $z_0 \in X$ and Borel probability $\PP_0$,
    \[
    \limsup_{n} \PP_n(\{x \in X: B_\xi(x) \ge B_\xi(z_n)\}) \le \PP_0(\{x \in X: B_\xi(x) \ge B_\xi(z_0)\})
    \]
    as $(z_n, \PP_n) \to (z_0, \PP_0)$ with $z_n \to z_0$ in $X$ and $\PP_n \to \PP_0$ weakly. By the continuity of $B_\xi$ 
    (Proposition \ref{prop:cont}), 
    we have,
    \[
    t_n:= B_\xi(z_n) \to t_0:= B_\xi(z_0).
    \]
    Thus, for any $\varepsilon >0$, eventually we have $t_n > t_0 - \varepsilon$, thus $\{B_\xi \ge t_n\} \subset \{B_\xi \ge t_0 - \varepsilon\}$. Since the set $\{B_\xi \ge t_0 - \varepsilon\}$ is closed, by Portmanteau theorem, we have,
    \begin{equation}\label{eqn:lessthan}
    \limsup_{n} \PP_n(B_\xi \ge t_n) \le \limsup_{n} \PP_n(B_\xi \ge t_0 - \varepsilon) \le \mathbb{P}_0(B_\xi \ge t_0 - \varepsilon).
    \end{equation}
    As $\varepsilon \downarrow 0$, we have $\{B_\xi \ge t_0 - \varepsilon\} \downarrow \{B_\xi \ge t_0\}$. By continuity of measure from above, we have
    \begin{equation}\label{eqn:fromabove}
    \lim_{\varepsilon \downarrow 0} \mathbb{P}_0(B_\xi \ge t_0 - \varepsilon) = \PP_0(B_\xi \ge t_0).
    \end{equation}
    Combine \eqref{eqn:lessthan} and \eqref{eqn:fromabove}, we have 
    $
    \limsup_{n} \PP_n(B_\xi \ge t_n) \le \PP_0(B_\xi \ge t_0)
    $
    as required. 
    
\end{proof}

\subsection{Proof of Lemma \ref{lemma:existence}}

\begin{proof}
If $D_*(\PP)=0$, then $D(z;\PP)=0$ for every $z\in X$, so $\mu_*(\PP)=X$ and there is nothing to prove. Assume now that $D_*(\PP)>0$. By 
Proposition \ref{prop:no_infty},
there exists $R>0$ such that
\[
d(o,z)\ge R \quad\Longrightarrow\quad D(z;\PP)<\frac{D_*(\PP)}{2}.
\]
Hence every maximizer lies in the closed ball $\overline{B}(o,R)$. Since closed balls are compact on Hadamard manifolds and $D(\cdot;\PP)$ is upper semicontinuous by 
Proposition \ref{prop:usc},
the supremum of $D(\cdot;\PP)$ is attained on $\overline{B}(o,R)$. Thus $\mu_*(\PP)$ is nonempty, and being the argmax of an upper semicontinuous function on a compact set, it is compact.
\end{proof}

\subsection{Proof of Theorem \ref{thm:depth_lbound}}
\label{proof:depth_lbound}

First, we introduce another notion of half-space \cite{rusciano2018riemannian}, which is defined via the tangent space as follows,
\[
    H^{\tang}_{v, z} := \{\exp_z(w): w \in T_zX, \langle w, v\rangle_z \ge 0\}.
\]
We can then define the corresponding depth as 
\[
D^{\tang}(z; \PP) = \inf_{v \in S^{d-1} \subset T_zX} \PP(H^\tang_{v, z}).
\]
It's shown in \cite[Theorem 2]{rusciano2018riemannian} that for $\PP \ll \vol$, there exists $z \in X$ such that for any $v \in T_zX$, 
\[
\PP((H_{v, z}^\tang)^c) \le 1 - 1/(d+1).
\]
Equivalently, there exists $z \in X$ such that for any $v \in T_zX$, $\PP(H_{v, z}^\tang) \ge 1/(d+1)$. It follows that for $\PP \ll \vol$, there exists $z \in X$ such that $D^\tang(z; \PP) \ge 1/(d+1)$. 

We can show that the horospherical depth $D(z; \PP)$ dominates $D^\tang(z; \PP)$ in the following lemma.

\begin{lemma}\label{lemma:dominate}
    For any $z \in X$, $D(z; \PP) \ge D^\tang(z; \PP)$.
\end{lemma}

\begin{proof}
        Note that, 
    \begin{align*}
        D(z; \PP) \ge D^\tang(z; \PP) &\iff \inf_{\xi \in \partial X}\PP(H^+_{\xi, z}) \ge \inf_{v \in S^{d-1}}\PP(H^\tang_{v, z}) \\
        &\Longleftarrow \PP(H^+_{\xi, z}) \ge \PP(H^\tang_{v, z}), \quad \forall \xi \in \partial X, v \in S^{d-1} \\
        &\iff H^+_{\xi, z} \supset H^\tang_{v, z}, \quad \forall \xi \in \partial X, v \in S^{d-1}. 
    \end{align*}
    Therefore, it's sufficient to show 
    \begin{equation}\label{eq:supset_condtion}
    H^+_{\xi, z} \supset H^\tang_{v, z}, \quad \forall \xi \in \partial X, v \in S^{d-1}.
    \end{equation}
    Fix $\xi \in \partial X$. Let $v\in T_z X$ and consider the geodesic $\gamma(t):= \exp_z(tv)$ with $\gamma(0) =z$ and $\gamma(1) = x$. As shown in 
    Theorem \ref{thm:busemann_f_convex}, 
    $B_\xi \circ \gamma$ is convex. Thus, we have,
    \begin{align*}
        B_\xi(x) = B_\xi(\gamma(1)) &\ge B_\xi(\gamma(0)) + \frac{\dd}{\dd t} \Bigr|_{t=0} B_\xi(\gamma(t)) \\
        &= B_\xi(\gamma(0)) + \langle \grad_{\gamma(0)} B_\xi, \dot \gamma(0) \rangle_{\gamma(0)} \\
        &= B_\xi(z) + \langle \grad_z B_\xi\, v \rangle_z.
    \end{align*}
    That is,  $B_\xi(x) - B_\xi(z) \ge \langle u_\xi, v \rangle_z$. It follows that, given $\xi \in \partial X$ and $z \in X$, we have
    \begin{align*}
        H^+_{\xi, z} = \{x \in X: B_\xi(x) \ge B_\xi(z)\} & \supset \{x \in X: \langle \grad_z B_\xi, v\rangle_z \ge 0\} \\
        &= \{\exp_z(v): v \in T_zX, \langle \grad_z B_\xi, v\rangle_z \ge 0\}\\
        &= H^\tang_{u_\xi, z}, \quad\text{where}\; u_\xi := \grad_z B_\xi
    \end{align*}
    We note that $u_\xi \in S^{d-1} \subset T_zX$ since $\|\grad B_\xi\| \equiv 1$ 
    by Proposition \ref{prop:cont}.
    Furthermore, the map $\xi \mapsto \grad_z B_\xi$ is surjective on $S^{d-1}$ by Lemma \ref{lemma:grad_bijection}. Therefore, \eqref{eq:supset_condtion} follows.
\end{proof}

\begin{proof}[Proof of Theorem \ref{thm:depth_lbound}]
    \medskip
    \noindent
    \textbf{Step 1: convolute $\PP$ using heat kernel.}
    Let $\PP$ as any Borel probability measure on Hadamard manifold $X$. Following \cite[Lemma 6.3]{donoho1992breakdown}, we consider the depth $D(z; \PP_h)$ of the probability measure $\PP_h$, obtained by convoluting $\PP$ with heat kernel $p_h(x, \cdot)$ of time horizon $h$, defined as follows,
    \begin{equation}
        \PP_h(A) := \int_X \int_A p_h(x, y) \, \dd\vol(y) \,\dd\PP(x).
    \end{equation}
    Thus, for any bounded continuous function $f \in C_b(X)$, 
    \begin{align*}
        \int f(x) \,  \dd \PP_h(x) = \int_X \int_X f(y) p_h(x,y) \, \dd \vol(y) \, \dd \PP(x) 
        = \int (e^{h\Delta}f)(x) \dd \PP(x).
    \end{align*}
    By \cite[Theorem 7.16]{grigoryan2009heat}, we have $(e^{h\Delta}f)(x) \to f(x)$ locally uniform in $x$. Since the heat semigroup is a contraction on $L^\infty$, we have $\|e^{h\Delta}f\| \le \|f\|_\infty$. Dominated convergence therefore gives, 
    \[
    \int f \, \dd \PP_h \to \mathbb \int f \, \dd \PP. 
    \]
    Thus, $\PP_h \Rightarrow\PP$ as $h \downarrow 0$. 

       \medskip
    \noindent
    \textbf{Step 2: Lower bounded by tangent space depth.} 
    Since $\PP_h\ll\vol$, \cite[Theorem 2]{rusciano2018riemannian} yields a point $z_h\in X$ such that $D^{\tang}(z_h;\PP_h)\ge 1/(d+1)$. Combine with Lemma \ref{lemma:dominate}, we have for every $h >0$, $D(z_h;\PP_h)\ge 1/(d+1)$.
    
    Now let $h_n\downarrow 0$ and set $z_n:=z_{h_n}$. Since $\PP_{h_n}\Rightarrow \PP$, the sequence $(\PP_{h_n})$ is uniformly tight. Choose $R>0$ such that
    \[
    \sup_n \PP_{h_n}(X\setminus B(o,R))<\frac{1}{d+1}.
    \]
    We show that $z_n\in B(o,R)$ for every $n$. Suppose instead that $r_n:=d(o,z_n)\ge R$. Let $\widetilde\gamma_n:\R\to X$ be the complete geodesic with $\widetilde\gamma_n(0)=o$ and $\widetilde\gamma_n(r_n)=z_n$, and set $\eta_n:=\widetilde\gamma_n(-\infty)$. Then
    \[
    B_{\eta_n}(z_n)=\lim_{t\to\infty}\bigl(d(z_n,\widetilde\gamma_n(-t))-t\bigr)
    =\lim_{t\to\infty}\bigl(d(\widetilde\gamma_n(r_n),\widetilde\gamma_n(-t))-t\bigr)=r_n.
    \]
    Since $B_{\eta_n}(x)\le d(o,x)$ for all $x\in X$, we have
    \[
    H^+_{\eta_n,z_n}
    =\{x\in X: B_{\eta_n}(x)\ge B_{\eta_n}(z_n)\}
    \subset \{x\in X: d(o,x)\ge r_n\}
    \subset X\setminus B(o,R).
    \]
    Therefore
    \[
    D(z_n;\PP_{h_n})\le \PP_{h_n}(H^+_{\eta_n,z_n})
    \le \PP_{h_n}(X\setminus B(o,R))
    < \frac{1}{d+1},
    \]
    contradicting $D(z_n;\PP_{h_n})\ge 1/(d+1)$. Thus $z_n\in B(o,R)$ for all $n$.
    
    Because $\overline{B}(o,R)$ is compact, after passing to a subsequence we may assume $z_n\to z_*\in \overline{B}(o,R)$. By the joint upper semicontinuity of $D(z;\PP)$ in $(z,\PP)$ 
    (Proposition \ref{prop:usc}),
    \[
    D(z_*;\PP)\ge \limsup_{n\to\infty} D(z_n;\PP_{h_n})\ge \frac{1}{d+1}.
    \]
    Hence $D_*(\PP)\ge D(z_*;\PP)\ge 1/(d+1)$.
    
    
\end{proof}

\subsection{Proof of Theorem \ref{thm:region_region_prop}}

\begin{proof}
    \textbf{Part 1:} It directly follows from 
    Theorem \ref{lemma:superlevel}.
    
    \medskip
    \noindent
    \textbf{Part 2:} For $\alpha > D_*$, $\DD^\alpha(\PP) =\emptyset$, thus compact. By 
    Proposition \ref{prop:no_infty},
    for any $0 <\alpha \le D_*$ ($D_* \ge 1/(d+1)$ by Centerpoint 
    Theorem \ref{thm:depth_lbound}), 
    there exists some $R> 0$ such that $d(o, z) \ge R$ implies $D(z;\PP) \le \alpha /2 < \alpha$. Thus, we must have $\DD^\alpha(\PP) \subset \bar{B}(o, R)$, which is compact on any complete manifolds. Since $\DD^\alpha(\PP)$ is the intersection of a family of horoball, $\DD^\alpha(\PP)$ is closed as well. The compactness of $\DD^\alpha(\PP)$ then follows from the fact any closed subset of compact set is compact.

    \medskip
    \noindent
    \textbf{Part 3: } It follows directly from the Centerpoint 
    Theorem \ref{thm:depth_lbound}. 
    
    \medskip
    \noindent
    \textbf{Part 4:} Since $\alpha < D_*(\PP)$, there exist $y \in X$ such that $D(y; \PP) > \alpha$. Set $\varepsilon := (D(y;\PP) - \alpha)/2$. By continuity of $D(\cdot; \PP)$, there exist $r_0 > 0$ such that $d(y, z) < r_0$ implies $|D(z; \PP) - D(y;\PP)| < \varepsilon$. Hence, for every $z \in B(y, r_0)$, 
    \[
    D(z; \PP) > D(y; \PP) - \varepsilon = \frac{D(y; \PP) + \alpha}{2} > \alpha.
    \]
    Therefore, $B(y, r_0) \subset D^\alpha(\PP)$. In particular, $D^\alpha(\PP)$ has nonempty interior. 
    Since $X$ is a Hadamard manifold, $\exp_y$ is a global diffeomorphism. Denote $S_yX = \{v \in T_yX: \|v\| = 1\}$, and for each $u \in S_yX$, define 
    \[
    \rho(u) := \sup\{t\ge0: \exp_y(su) \in D^\alpha(\PP) \; \forall 0 \le s \le t\}.
    \]
    Since $B(y, r_0) \subset D^\alpha(\PP)$, we have $\rho(u) \ge r_0 > 0$. Since $D^\alpha(\PP)$ is compact by Part 2, $\rho(u) < \infty$ for all $u$. Because $D^\alpha(\PP)$ is geodesically convex and contains $y$, the intersection of $D^\alpha(\PP)$ with the geodesic ray $\gamma_u(t) := \exp_y(tu)$ is exactly the interval
    \[
    D^\alpha(\PP) \cap \gamma_u([0,\infty))
    =
    \{ \gamma_u(t) : 0 \le t \le \rho(u) \}.
    \]
    Consequently,
    \begin{equation}
    D^\alpha(\PP)\label{eqn:tangent_ball}
    =
    \{ \exp_y(tu) : u \in S_yX,\ 0 \le t \le \rho(u) \}. 
    \end{equation}
    We next show that $\rho$ is continuous on $S_yX$.
    
    \medskip
    \noindent
    \textbf{Upper semicontinuity of $\rho$.} Let $u_n \to u$ in $S_yX$. Passing to a subsequence if necessary, assume
    \[
    \rho(u_n) \to L = \limsup_{n\to\infty} \rho(u_n).
    \]
    Since each $\exp_y(\rho(u_n)u_n)$ belongs to $D^\alpha(\PP)$ and $D^\alpha(\PP)$ is compact, we obtain $\exp_y(Lu) \in D^\alpha(\PP)$.
    Hence $L \le \rho(u)$, so $\limsup_{n\to\infty} \rho(u_n) \le \rho(u)$.
    
    \medskip
    \noindent
    \textbf{Lower semicontinuity of $\rho$.}
    Fix $u \in S_yX$, and let $0 < s < \rho(u)$. Set $p := \exp_y(\rho(u)u) \in D^\alpha(\PP)$ and $q := \exp_y(su)$. Then $q$ lies on the geodesic segment joining $p$ and $p$, but $q \ne p$. We claim that
    \begin{equation}\label{eqn:int_contain}
    [y,p) \subset \operatorname{int}(D^\alpha(\PP)). 
    \end{equation}
    Indeed, since $y \in \operatorname{int}(D_\alpha(\PP))$, there exists $r>0$ such that
    \[
    B(y,r) \subset D^\alpha(\PP).
    \]
    For $0 \le \lambda < 1$, define $H_\lambda(x) := \exp_p\bigl((1-\lambda)\exp_p^{-1}(x)\bigr)$ for $x \in X$. Since $\exp_p$ is a diffeomorphism, $H_\lambda$ is a homeomorphism of $X$ onto itself.
    Moreover, for each $x \in B(y,r)$, the point $H_\lambda(x)$ lies on the geodesic segment joining $x$ and $p$, hence belongs to $D^\alpha(\PP)$ by geodesic convexity. Therefore, $H_\lambda(B(y,r)) \subset D^\alpha(\PP)$.

    Since $H_\lambda(B(y,r))$ is open and contains $H_\lambda(y)$, we conclude that $H_\lambda(y) \in \operatorname{int}(D^\alpha(\PP))$ for every $0 \le \lambda < 1$.
    As $\lambda$ varies in $[0,1)$, the points $H_\lambda(y)$ trace exactly the segment $[y,p)$, proving \eqref{eqn:int_contain}.
    
    Applying \eqref{eqn:int_contain} to $q$, we get $q \in \operatorname{int}(D^\alpha(\PP))$.
    Since $\operatorname{int}(D^\alpha(\PP))$ is open and $\exp_y$ is continuous, there exists a neighborhood $U$ of $u$ in $S_yX$ such that
    \[
    \exp_y(su') \in D^\alpha(\PP)
    \qquad\text{for all } u' \in U.
    \]
    Hence $\rho(u') \ge s$ for all $u' \in U$, and therefore $\liminf_{u' \to u} \rho(u') \ge s$. Since this holds for every $s < \rho(u)$, it follows that $\liminf_{u' \to u} \rho(u') \ge \rho(u)$. Combining the two parts, $\rho$ is continuous on $S_yX$.
    
    Now identify $\overline{B}^d$ with the closed unit ball in $T_yX$, and define $\Phi : \overline{B}^d \to D^\alpha(\PP)$ by $\Phi(0) := y$, $\Phi(ru) := \exp_y(r\,\rho(u)\,u)$, $u \in S_yX$, $0 < r \le 1$. 

    \medskip
    \noindent
    \textbf{$\Phi$ is homeomorphic.} Since $\rho$ is continuous on the compact set $S_yX$, it is bounded; thus $\Phi$ is continuous at $0$, and clearly continuous elsewhere.
    
    By \eqref{eqn:tangent_ball}, the map $\Phi$ is surjective.
    It is also injective: if $\Phi(r_1u_1)=\Phi(r_2u_2)$, then $\exp_y(r_1\rho(u_1)u_1)=
    \exp_y(r_2\rho(u_2)u_2)$, and since $\exp_y$ is injective, we get $r_1\rho(u_1)u_1 = r_2\rho(u_2)u_2$. Because $u_1,u_2$ are unit vectors and $\rho>0$, this implies $u_1=u_2$ and $r_1=r_2$.
    
    Finally, the inverse map is given by $\Phi^{-1}(y)=0$, and, for $x \neq y$, writing uniquely, $x=\exp_y(tu)$, $u \in S_yX$, $0 < t \le \rho(u)$, one has $\Phi^{-1}(x)=tu/\rho(u)$. Since both $t$ and $u$ depend continuously on $x$ through the global diffeomorphism $\exp_y$, and $\rho$ is continuous and strictly positive, $\Phi^{-1}$ is continuous. Thus $\Phi$ is homeomorphism from $\overline{B}^d$ onto $D^\alpha(\PP)$. Therefore, $D^\alpha(\PP) \cong \overline{B}^d$.
   
\end{proof}

\section{Proofs in Section \ref{sec:uniqueness}}

\subsection{Proof of Theorem \ref{thm:busemann_f_convex}}
\label{proof:busemann_f_convex}

\begin{proof} Let $\widetilde\gamma:\R \to X$ be a complete geodesic with $\widetilde\gamma(+\infty) = \xi^+$ and $\widetilde\gamma(-\infty) = \xi^-$.

    \medskip
    \noindent
    \textbf{Part 2: } Denote $\gamma^+, \gamma^-$ as the geodesic ray with starting point $o$ and end with $\xi^+$ and $\xi^-$, respectively. We have,
    \begin{align*}
        B_{\xi^+}(\widetilde\gamma(t)) \stackrel{(1)}{=} B_{\xi^+, \widetilde\gamma(0)}(\widetilde\gamma(t)) + B_{\xi^+}(\widetilde\gamma(0)) =B_{\widetilde\gamma}(\widetilde\gamma(t)) + B_{\xi^+}(\widetilde\gamma(0)) = -t + B_{\xi^+}(\widetilde\gamma(0))
    \end{align*}
    where (1) follows from 
    \eqref{eqn:busemann_relation}. 
    Similarly, define the reverse direction geodesic of $\widetilde\gamma$ as $\gamma^*(t):= \widetilde\gamma(-t)$, so $\gamma^*(+\infty) = \xi^-$. We have,
    \begin{align*}
        B_{\xi^-}(\widetilde\gamma(t)) &= B_{\xi^-}(\gamma^*(-t))  = B_{\xi^-, \gamma^*(0)}(\gamma^*(-t)) + B_{\xi^-}(\gamma^*(0)) \\
        &=B_{\gamma^*}(\gamma^*(-t)) + B_{\xi^-}(\gamma^*(0)) = t + B_{\xi^-}(\widetilde\gamma(0)). 
    \end{align*}

    \medskip
    \noindent    
    \textbf{Part 1: } For notation convenience, we will use $\gamma$ to denote $\widetilde\gamma$ for this part of the proof. Denote $f(t):= B_\xi (\gamma(t))$. It's sufficient to show that the second derive of $f$, $f''$, is strictly positive under $\sec_X < 0$ and nonnegative under $\sec_X \le 0$. 
    
    \medskip
    \noindent
    \textbf{Step 1: Give alternative form of $f''(t)$.}  By chain rule, we have,
    \begin{align*}
        f'(t) &= \langle \grad_{\gamma(t)} B_\xi, \dot{\gamma}(t)\rangle, \\
        f''(t)&= \langle \nabla_{\dot \gamma(t)} \grad_{\gamma(t)} B_\xi, \dot\gamma(t)\rangle + \langle \grad_{\gamma(t)} B_\xi, \nabla_{\dot\gamma}\dot{\gamma}(t)\rangle.
    \end{align*}
    Since $\gamma$ is a geodesic, $\nabla_{\dot\gamma}\dot\gamma \equiv 0$, thus,
    \begin{equation}\label{eq_hessian}
        f''(t)= \langle \nabla_{\dot \gamma(t)} \grad_{\gamma(t)} B_\xi , \dot\gamma(t)\rangle = \langle \hess_{B_\xi}(\dot \gamma(t)), \dot\gamma(t) \rangle,
    \end{equation}
    where $\hess_{B_\xi}$ denotes the $(1,1)$-tensor field $\nabla (\grad B_\xi)$ (following the convention in \cite[Page 320]{lee2018introduction}).  
    We now simplify $\langle \hess_{B_\xi}(\dot \gamma(t)), \dot\gamma(t) \rangle$ using standard structural properties of Busemann functions. By 
    Proposition \ref{prop:cont},
    one has $\|\grad B_\xi\|\equiv 1$. Differentiating $\frac12\|\grad B_\xi\|^2\equiv \frac12$ yields
    \begin{equation*}
    0= \langle \hess_{B_\xi}\left(v\right),\grad B_\xi\rangle
    \qquad\text{for all }v\in TX.
    \end{equation*}
    Since the Hessian of a $C^2$ function ($B_\xi$ is $C^2$ on Hadamard manifold) is self-adjoint, we have,
    \begin{equation}\label{eq:busemann_zero}
    0= \langle \hess_{B_\xi}\left(\grad B_\xi\right),v\rangle
    \qquad\text{for all }v\in TX.
    \end{equation}
    Decompose the tangent vector $\dot\gamma(t)$ into its component parallel to $\grad_{\gamma(t)} B_\xi$ and its orthogonal component:
    \[
    \dot\gamma(t)=a(t)\,\grad_{\gamma(t)} B_\xi + w(t),
    \qquad
    w(t)\perp \grad_{\gamma(t)} B_\xi.
    \]
    Importantly, we note that under $\xi \not\in \{\xi^+, \xi^-\}$, $w(t) \ne 0$. Indeed, by 
    Proposition \ref{prop:cont},
    we have $\grad B_\xi(\gamma(t)) = - \dot \gamma_{\gamma(t), \xi}$ where $\gamma_{\gamma(t), \xi^+}$ denote the geodesic connecting $\gamma(t)$ and $\xi$. Thus, $\gamma$ and $\gamma_{\gamma(t), \xi}$ share the same starting point but are in different asymptotic class since $\xi \notin \{\xi^+, \xi^-\}$. It follows that $\grad_{\gamma(t)} B_\xi$ and $\dot\gamma(t)$ are linearly independent. 
    
    Using bilinearity and \eqref{eq:busemann_zero} we obtain
    \begin{equation}\label{eq:fpp_reduce_to_orth}
        f''(t)=\langle \hess_{B_\xi}(w(t)),w(t)\rangle.
    \end{equation}
    Next, From \cite[Proposition 3.1]{heintze1977geometry}, for $v \in T_pX$, $\hess_{B_\xi}(v) = -J_{v}'(0)$, where $J_v$ is  the stable Jacobi field along $\gamma_{p, \xi}$ that connects $p$ and $\xi$, with $J_v(0) = v$. Denote $p:= \gamma(t)$, plug into \eqref{eq:fpp_reduce_to_orth}, we have,
    \[
        f''(t) = -\langle J_{w(t)}'(0), J_{w(t)}(0) \rangle.
    \]
    where $J_{w(t)}$ is the stable Jacobi field along $\gamma_{p, \xi}$ that connects $p =\gamma(t)$ and $\xi$, with $J_{w(t)}(0) = w(t)$. 
    Let's denote $a_t(s):= \langle J_{w(t)}'(s), J_{w(t)}(s) \rangle$.
    It follows that \begin{equation}\label{eq:hessian_a}
        -a_t(0) = f''(t).
    \end{equation}
    Therefore, it's sufficient to show $a_t(0) < 0$ under $\sec_X <0$ and $a_t(0) \le 0$ under $\sec_X \le 0$. 

    \medskip
    \noindent
    \textbf{Step 2: Show $a_t$ is non-decreasing.} First, note that, by chain rule, 
    \begin{equation}\label{eq:jacobi_a}
    \frac{\dd}{\dd s} \|J_{w(t)}(s)\|^2 = 2 \langle J'_{w(t)}(s), J_{w(t)}(s)\rangle = 2a_t(s)
    \end{equation}
    Differentiate $a_t(s)$, we have,
    \begin{align*}
        a_t'(s) &= \langle J''_{w(t)}(s), J_{w(t)}(s)\rangle + \langle J'_{w(t)}(s), J'_{w(t)}(s)\rangle \\
        &=\langle J''_{w(t)}(s), J_{w(t)}(s)\rangle + \|J'_{w(t)}(s)\|^2 \\
        &\stackrel{(1)}{=} -\langle R(J_{w(t)}(s), \dot \gamma_{p, \xi}(s))\dot \gamma_{p, \xi}(s), J_{w(t)}(s)\rangle + \|J'_{w(t)}(s)\|^2 \\
        &\stackrel{(2)}{=} -\sec(J_{w(t)}(s), \dot \gamma_{p, \xi}(s))\underbrace{\left(\|J_{w(t)}(s)\|^2\|\dot\gamma_{p, \xi}(s)\|^2 -\langle J_{w(t)}(s), \dot\gamma_{p, \xi}(s) \rangle^2  \right)}_{(3)} + \|J'_{w(t)}(s)\|^2,
    \end{align*}
    where $(1)$ comes from the Jacobi equation,  
    \[
    J''_{w(t)}(s) + R(J_{w(t)}(s), \dot \gamma_{p, \xi}(s))\dot \gamma_{p, \xi}(s) = 0,
    \]
    $(2)$ comes from the definition of sectional curvature. 
    
    Now we observe that $(3) \ge 0$ by Cauchy-Schwarz inequality with equality if and only if $J_{w(t)}(s)$ and $\dot\gamma_{p, \xi}(s)$ are linearly dependent. One can see $J_{w(t)}(s) \perp \dot\gamma_{p, \xi}(s)$ as follows. Consider the scalar function
    \[
    b_t(s):=\big\langle J_{w(t)}(s),\,\dot\gamma_{p, \xi}(s)\big\rangle .
    \]
    Since $\gamma_{p, \xi}$ is a geodesic, $\nabla_{\dot\gamma_{p, \xi}}\dot\gamma_{p, \xi}\equiv 0$, hence
    \[
    b_t'(s)=\big\langle J_{w(t)}'(s),\,\dot\gamma_{p, \xi}(s)\big\rangle,
    \qquad
    b_t''(s)=\big\langle J_{w(t)}''(s),\,\dot\gamma_{p, \xi}(s)\big\rangle.
    \]
    Using the Jacobi equation $J_{w(t)}''(s)+R\!\big(J_{w(t)}(s),\dot\gamma_{p, \xi}(s)\big)\dot\gamma_{p, \xi}(s)=0$,
    we obtain
    \[
    b_t''(s)
    =-\big\langle R\!\big(J_{w(t)}(s),\dot\gamma_{p, \xi}(s)\big)\dot\gamma_{p, \xi}(s),\,\dot\gamma_{p, \xi}(s)\big\rangle
    =0,
    \]
    since $R(u,v)w\perp w$ for all $u,v \in TX$. Hence $b_t(s)=\alpha s+\beta$ is affine. Now $J_{w(t)}$ is a stable Jacobi field, so $\|J_{w(t)}(s)\|$ is bounded as $s \to \infty$. Therefore $|b_t(s)|\le \|J_{w(t)}(s)\|\|\dot\gamma_{p, \xi}(s)\| = \|J_{w(t)}(s)\|$ is bounded, which forces $\alpha=0$, so $b_t \equiv \beta$.
    Moreover,
    \[
    b_t(0)=\big\langle J_{w(t)}(0),\,\dot\gamma_{p, \xi}(0)\big\rangle
    =\big\langle w(t),\,\dot\gamma_{p, \xi}(0)\big\rangle=0,
    \]
    since $w(t)\perp \grad_{\gamma(t)} B_{\xi}$ and $\grad_{\gamma(t)} B_{\xi}=-\dot\gamma_{p, \xi}(0)$ by 
    Proposition \ref{prop:cont}.
    Consequently $b_t \equiv 0$, i.e.
    \[
    J_{w(t)}(s)\perp \dot\gamma_{p, \xi}(s)\qquad\text{for all }s\ge 0.
    \]
    Thus, whenever $J_{w(t)}(s)\neq 0$, we have $(3) > 0$. 
    
    Therefore, we have, for $\sec_X < 0$ and $J_w(t) \ne 0$, we have $\sec(J_{w(t)}(s), \dot \gamma_{p, \xi}(s)) < 0$, and thus $a'(s) > \|J_{w(t)}(s)\|^2 \ge 0$. Thus, for $\sec_X < 0$, $a_t$ is strictly increasing on any interval where $J_{w(t)} \ne 0$ and non-decreasing everywhere. For $\sec_X \le 0$, $a_t$ is non-decreasing everywhere. 

    \medskip
    \noindent
    \textbf{Step 3: Show $a_t(s) \le 0$ for all $s$.} Suppose $a_t(s_0)>0$ for some $s_0$, then it follows $a_t(s) \ge a_t(s_0)$ for all $s > s_0$. Combine with \eqref{eq:jacobi_a}, we must have $\|J_{w(t)}(s)\|^2$ grow at least linearly, and which implies $\|J_{w(t)}(s)\|$ is unbounded as $s \to \infty$. However, this contradicts with the fact that $J_{w(t)}$ is a stable Jacobi field. Therefore, $a_t(s) \le 0$ for all $s$. In particular, we have $a_t(0) \le 0$. 

    \medskip
    \noindent
    \textbf{Step 4: Show $a_t(0) \ne 0$ under $\sec_X < 0$.} Suppose for contradiction, we have $a_t(0) = 0$. Since $a_t(s)$ is non-decreasing and $a_t(s) \le 0$ for all $s$, then we must have $a_t(s) \equiv 0$. Note previously in Step 2, under $\sec_X < 0$, we have $a_t'(s) > 0$ where $J_{w(t)}(s) \ne 0$, which forces $J_{w(t)} \equiv 0$. However, we established in step 1 that $J_{w(t)}(0) = w(t) \ne 0$. Contradiction, and thus $a_t(0) < 0$. Combine with \eqref{eq:hessian_a}, we have $f''(t) > 0$ under $\sec_X < 0$ as desired. 
\end{proof}

\subsection{Proof for Corollary \ref{prop:horoball_convex}}

\begin{proof}
    Let $x, y \in H_\xi(t)$ and geodesic segment $\gamma$ with $\gamma(0) = x$ and $\gamma(1) = y$, by 
    Theorem \ref{thm:busemann_f_convex},
    we have $B_\xi(\gamma(s))\le (1-s)B_\xi(x) + sB_\xi(y)$. Suppose at least one of $x, y$ are in $H^\circ_\xi(t)$, then we have $B_\xi(\gamma(s)) < t$ for any $s \in (0,1)$, thus $\gamma(s) \in H^\circ_\xi(t)$. Now we consider the other case where neither of $x, y$ are in $H^\circ_\xi(t)$, then $x, y \in \partial H_\xi(t)$. Since $x \ne y$ and $x, y \in \partial H_\xi(t)$, we have $\gamma(\infty), \gamma(-\infty) \ne \xi$. By 
    Theorem \ref{thm:busemann_f_convex},
    part (1), we have $B_\xi(\gamma(s)) < t$ if $\sec_X < 0$ and $B_\xi(\gamma(s)) \le t$ otherwise. 
\end{proof}

\subsection{Proof for Corollary \ref{coro:monotonicity}}

\begin{proof}
    By Theorem \ref{thm:busemann_f_convex},
    $B_\xi \circ \gamma$ is convex for any $\xi \in \partial X$ and geodesic $\gamma$. Let $\gamma$ be the geodesic connecting $z_*$ and $z$ with $\gamma(0) = z^*$ and $\gamma(1) = z$. 
    \begin{align*}
    \PP(H^+_{\xi, \gamma(t)}) 
    &\ge \PP(\{x\in X: B_\xi(x) \ge (1-t) B_\xi(z_*) + tB_\xi(z)\}) \\
    &\ge \PP(\{x\in X: B_\xi(x) \ge \max\{B_\xi(z_*), B_\xi(z)\}\}) = \min \{\PP(H^+_{\xi, z_*}), \PP(H^+_{\xi, z}) \}
    \end{align*}
    It follows that
    \begin{align*}
        D(\gamma(t); \PP) = \inf_\xi \PP(H^+_{\xi, \gamma(t)}) \ge \inf_\xi \min \{\PP(H^+_{\xi, z_*}), \PP(H^+_{\xi, z}) \}
        \stackrel{(1)}{=} \inf_\xi\PP(H^+_{\xi,z}) = D(z; \PP). 
    \end{align*}
    where $(1)$ follows from the fact that $z_*$ is maximizer,
    $
        \inf_\xi \PP(H^+_{\xi, z_*}) \ge \inf_\xi \PP(H^+_{\xi, z}).
    $
\end{proof}

\subsection{Proof of Lemma \ref{lemma:D-upper}}
\label{proof:D-upper}

\begin{proof}
Suppose for contradiction that there exist $z$ such that $D(z; \PP)=1$, i.e. $\PP(\{x: B_\xi(x)\ge B_\xi(z)\})=1$ for every $\xi\in\partial X$.

\medskip
\noindent
\textbf{Step 1: The horoball intersection is a singleton.}  We claim
\[
\bigcap_{\xi\in\partial X}\{x\in X : B_\xi(x)\ge B_\xi(z)\} = \{z\}.
\]
Indeed, if $x \ne z$, we can construct $\xi$ such that $B_\xi(x) < B_\xi(z)$ as follows: let $\gamma$ be the unit-speed geodesic from $\gamma(0) = z$ to $\gamma(t) = x$ with $t := d(x,z) $ and set $\xi:=\gamma(+\infty )$. Then $B_\xi(\gamma(t))=B_\xi(z)-t$ 
(part 2 of Theorem \ref{thm:busemann_f_convex})
, so $B_\xi(x) = B_\xi(z) - d(z,x) < B_\xi(z)$. Hence $x\not\in\{x\in X:B_\xi(x) \ge B_\xi(z)\}$.

\medskip
\noindent
\textbf{Step 2: $\PP$ must be a point mass.}  Choose a countable dense subset $\{\xi_k\}_{k\ge 1}\subset\partial X$.  For each $k$, $\PP(\{x :B_{\xi_k}(x)\ge B_{\xi_k}(z)\})=1$.  By countable intersection, 
\[
\PP\left(\bigcap_{k=1}^\infty \{x : B_{\xi_k}(x)\ge B_{\xi_k}(z)\}\right) = 1.
\]
By continuity of $\xi\mapsto B_\xi(x)$ for each fixed $x$: if $B_{\xi_k}(x)\ge B_{\xi_k}(z)$ for all $k$ in a dense subset and $\xi_n\to\xi$, then $B_\xi(x)=\lim_{n \to \infty} B_{\xi_n}(x)\ge\lim_{n \to \infty} B_{\xi_n}(z)=B_\xi(z)$. Therefore
\[
\bigcap_{k=1}^\infty \{x: B_{\xi_k}(x) \ge B_{\xi_k}(z)\} = \bigcap_{\xi\in\partial X}\{x: B_\xi\ge B_\xi(z)\} = \{z\},
\]
where the last equality is from Step 1.  Hence $\PP(\{z\})=1$, i.e. $\PP=\delta_z$.

\medskip
\noindent
\textbf{Step 3: Contradiction.} If $\PP=\delta_z$, then $1 = \PP(\{z\}) \le \PP(\{x : B_\xi(x)=B_\xi(z)\})$, contradicting 
Assumption \ref{assump:no_atom}.
\end{proof}

\subsection{Proof of Theorem \ref{thm:unique_rank_1}}
\label{proof:unique_rank_1}

\begin{proof}
    We prove by contradiction. Let $x \ne y \in X$ and denote $\alpha := \min\{D(x; \PP), D(y;\PP)\}$. By 
    Lemma \ref{lemma:D-upper},
    we have $\alpha < 1$ under 
    Assumption \ref{assump:no_atom}.
    Let geodesic $\gamma$ with $\gamma(0) = x$ and $\gamma(s) = y$ where $s := d(x, y)$. 

    \medskip
    \noindent
    \textbf{Step 1: $B_\xi(\gamma(s)) < t_\xi(\alpha)$ for every $\xi \in \partial X$ and any $s \in (0,1)$. } For $\xi \not\in\{\xi^+, \xi^-\}$, by 
    Theorem \ref{thm:busemann_f_convex} part (1), 
    we have,
    \[
    B_\xi(\gamma(s)) < (1-s) B_\xi(x) + s B_\xi(y) \le t_\xi(\alpha).
    \]
    where the last inequality follows from $B_\xi(x), B_\xi(y) \in \{t: S_\xi(t) \ge \alpha\}$ since 
    \[
    D(x; \PP) = \inf_\xi S_\xi(B_\xi(x)) \ge \alpha, \quad D(y; \PP) = \inf_\xi S_\xi(B_\xi(y)) \ge \alpha; 
    \]
    thus $B_\xi(x), B_\xi(y) \le t_\xi(\alpha)$, 
    see \eqref{eq:t_xi}
    for the definition of $t_\xi(\alpha)$. 
    
    For $\xi \in \{\xi^+, \xi^-\}$, by 
    Theorem \ref{thm:busemann_f_convex} part (2), 
    we have,
    \[
    B_\xi(\gamma(s)) = (1-s)B_\xi(x) + sB_\xi(y) < t_\xi(\alpha),
    \]
    where the strict inequality follows from $B_\xi(x) \ne B_\xi(y)$.

    \medskip
    \noindent
    \textbf{Step 2: Show $p_\xi(\gamma(s)) > \alpha$ for every $\xi$ and every $s \in (0,1)$. } By Step 1, $B_\xi(\gamma(s)) < t_\xi(\alpha)$. By Lemma \ref{lemma:slab-positive}, we have under 
    Assumption \ref{assump:no_atom} and \ref{assump_connected_supp}, 
    for $\alpha \in (0,1)$, $p_\xi(\gamma(s))) > \alpha$.

    \medskip
    \noindent
    \textbf{Step 3: Show $D(\gamma(s); \PP) > \alpha$.}  By step 2, we have $p_\xi(\gamma(s)) > \alpha$ for every $\xi$. By Lemma \ref{lemma:p_usc}, under under 
    Assumption \ref{assump:no_atom},
    the map $\xi \mapsto p_\xi(\gamma(s)):= \PP(\{x: B_\xi(x) \ge B_\xi(\gamma(s))\})$ is lower semicontinuous on the compact set $\partial X$ and thus attain its minimum at some $\xi_0 \in \partial X$, thus we have,
        \[
            D(\gamma(s); \PP) = p_{\xi_0}(\gamma(s)) > \alpha
        \]
    as required. 

    To prove $\mu_*(\PP)$ is singleton, suppose for contradiction let $x, y \in \mu_*(\PP)$, then $\alpha = D_*(\PP)$. By the Centerpoint 
    theorem \ref{thm:depth_lbound} and Assumption \ref{assump:no_atom}, 
    $0 < 1/ (d+1) < \alpha < 1$, and thus by strict quasi-concavity, $D(\gamma(s); \PP) > D_*(\PP)$ for any $s \in (0,1)$, contradiction.

\end{proof}

\subsection{Proof of Corollary \ref{coro:stability}}

\begin{proof}[Proof]
The argument proceeds in three steps. Crucially, no assumptions are placed on $\mathbb{Q}$: the sharp peak of $D(\cdot;\PP)$ at $\mu_*(\PP)$ is the only structural ingredient, and this depends on $\PP$ alone.

\medskip
\noindent\textbf{Step 1: Sharp peak from strict quasi-concavity.}
Under $\sec_X < 0$ and 
Assumptions \ref{assump:no_atom}-\ref{assump_connected_supp}, Theorem \ref{thm:unique_rank_1}
gives strict quasi-concavity of $D(\cdot;\PP)$ along geodesics. Denote $z_* = \mu_*(\PP)$ and $D_* = D(z_*;\PP)$. Define the \emph{depth deficit function}:
\[
\delta(r) := D_* - \sup_{\{z: d(z,z_*) = r\}} D(z;\PP), \qquad r > 0.
\]
Since $D(z; \PP) < D_*$ for any $z \ne z_*$, we have $\delta(r) > 0$ for all $r > 0$: the depth strictly decreases as we move away from $z_*$. By 
Proposition \ref{prop:no_infty},
$D(z;\PP) \to 0$ as $d(o,z) \to \infty$, so $\delta(r) \to D_*$ as $r \to \infty$.

The function $\delta$ is non-decreasing (by monotonicity of depth along rays from $z_*$, 
Corollary \ref{coro:monotonicity})
Therefore $\delta$ has a well-defined inverse: for any $\eta > 0$, define
\[
\rho(\eta) := \inf\{r > 0 : \delta(r) \geq \eta\}.
\]
This is the radius of the $(D_*-\eta)$-depth region:
\begin{equation}\label{eq:nearly-deepest}
\{z : D(z;\PP) \geq D_* - \eta\} \subset \overline{B}(z_*, \rho(\eta)).
\end{equation}

\medskip
\noindent\textbf{Step 2: From depth perturbation to spatial perturbation.}
Let $\mathbb{Q}$ be \emph{any} Borel probability measure with $|\PP - \mathbb{Q}|_{\tv} \leq \eta$. By 
Theorem \ref{thm:robust}:
\[
|D(z;\PP) - D(z;\mathbb{Q})| \leq \eta \quad \forall z \in X.
\]
Let $\hat z \in \mu_*(\mathbb{Q})$ be any maximizer of $D(\cdot;\mathbb{Q})$ 
(existence in Lemma \ref{lemma:existence}). 
Since $\hat z$ maximizes $D(\cdot;\mathbb{Q})$:
\[
D(\hat z;\mathbb{Q}) \geq D(z_*;\mathbb{Q}) \geq D(z_*;\PP) - \eta = D_* - \eta.
\]
Applying the TV bound again:
\[
D(\hat z;\PP) \geq D(\hat z;\mathbb{Q}) - \eta \geq D_* - 2\eta.
\]
By \eqref{eq:nearly-deepest}: $D(\hat z;\PP) \geq D_* - 2\eta$ implies $d(z_*, \hat z) \leq \rho(2\eta)$. Since $\hat z$ was an arbitrary element of $\mu_*(\mathbb{Q})$:
\[
\sup_{z \in \mu_*(\mathbb{Q})} d(z, z_*) \le \rho(2|\PP - \mathbb{Q}|_{\tv}).
\]
Setting $\omega(\eta) := \rho(2\eta)$ gives the claimed bound.


\end{proof}

\section{Proofs in Section \ref{sec:robust}}

\subsection{Proof of Theorem \ref{thm:robust}}

\begin{proof}
    For any probability measures $\PP, \QQ$ and any measurable set $A$,  we have $|\PP(A) - \QQ(A)| \le \|\PP -\QQ\|_{TV}$. 
    \begin{align*}
        |D(z; \PP) - D(z; \QQ)| &= \left|\inf_{\xi \in \partial{X}} \PP(H^+_{\xi,z}) - \inf_{\xi \in \partial{X}} \QQ(H^+_{\xi,z}) \right| \\
        &\stackrel{(1)}{\le} \sup_{\xi \in \partial{X}} \left| \PP(H^+_{\xi,z}) - \QQ(H^+_{\xi,z}) \right| \le \|\PP - \QQ\|_{TV},
    \end{align*}
    where $(1)$ follows from Lemma \ref{lemma:inf_sup}. 
\end{proof}

\subsection{Proof of Corollary \ref{coro:huber}}

\begin{proof}
Since $\|\PP_\varepsilon-\PP\|_{\mathrm{TV}}\le \varepsilon$, the result follows immediately from 
Theorem \ref{thm:robust}.
\end{proof}

\subsection{Proof of Theorem \ref{thm:inclusion_region}}

\begin{proof}
If $z\in \DD^{\alpha+\varepsilon}(\PP)$, then $D(z;\PP)\ge \alpha+\varepsilon$, hence 
Corollary \ref{coro:huber}
gives $D(z;\PP_\varepsilon)\ge \alpha$. This proves the first inclusion. The second is analogous.
\end{proof}

\subsection{Proof of Corollary \ref{coro:median_huber}}

\begin{proof}
By Corollary \ref{coro:huber},
$\|\PP_\varepsilon-\PP\|_{\tv}\le \varepsilon$. The claim then follows from 
Corollary \ref{coro:stability}.
\end{proof}

\subsection{Proof of Theorem \ref{thm:boundary_robust}}

For notation convience, write
\[
D_t(z):=D(z;\PP_t)
=
\inf_{\eta\in\partial X}
\Bigl[(1-\varepsilon)\PP(H_{\eta,z}^+)
      +\varepsilon\,\mathbf 1\{B_\eta(\gamma(t))\ge B_\eta(z)\}
\Bigr].
\]
Further, set $a(z, \eta):=\PP(H_{\eta,z}^+)$ and $I_t(\eta,z):=\mathbf 1\{B_\eta(\gamma(t))\ge B_\eta(z)\}$. Thus
\[
D_t(z)=\inf_{\eta\in\partial X}\bigl((1-\varepsilon)a_z(\eta)+\varepsilon I_t(\eta,z)\bigr).
\]
We will need the following Lemma. 
\begin{lemma}
    $a(\cdot, \cdot)$ is continuous on $X\times \partial X$.
\end{lemma}

\begin{proof}
    Let $(z_n,\eta_n)\to (z,\eta)$ in $X\times \partial X$, and define
\[
h_n(u):=B_{\eta_n}(u)-B_{\eta_n}(z_n),
\qquad
h(u):=B_\eta(u)-B_\eta(z).
\]
By Lemma \ref{lemma:compact_uniform_convergence_busemann}, $h_n\to h$ uniformly on compact subsets of $X$.
Fix $\delta,\rho>0$, and choose $R>0$ such that $\PP\bigl(X\setminus B(o,R)\bigr)<\delta$. For all sufficiently large $n$, one has $|h_n-h|<\rho$ on $B(o,R)$, hence
\[
\{h>\rho\}\cap B(o,R)\subset \{h_n\ge 0\}\cap B(o,R)\subset \{h\ge -\rho\}\cap B(o,R).
\]
Therefore,
\[
\PP(h>\rho)-\delta
\le
\liminf_{n\to\infty} a(z_n,\eta_n)
\le
\limsup_{n\to\infty} a(z_n,\eta_n)
\le
\PP(h\ge -\rho)+\delta.
\]
Letting first $\delta\downarrow 0$ and then $\rho\downarrow 0$, and using 
Assumption \ref{assump:no_atom} 
to obtain
\[
\PP(h=0)=\PP\bigl(\{u:B_\eta(u)=B_\eta(z)\}\bigr)=0,
\]
we conclude that $a(z_n,\eta_n)\to a(z,\eta)$ as required.
\end{proof}

\begin{proof}[Proof of Theorem \ref{thm:boundary_robust}]

\medskip
\noindent
\textbf{Step 1: asymptotics of the contaminating indicator.}
Since $\gamma(\infty)=\xi$, 
Theorem \ref{thm:busemann_f_convex} part (2)
gives
\[
B_\xi(\gamma(t))=-t+B_\xi(\gamma(0))\to -\infty.
\]
Hence, for every fixed $z\in X$, $I_t(\xi,z)=0$ for all sufficiently large $t$.

Now fix $\eta\neq \xi$. Because $\Sec_X<0$, $X$ is a visibility manifold; a standard consequence is that every ray to $\xi$ leaves every horoball centered at a distinct boundary point $\eta$. Equivalently,
\[
B_\eta(\gamma(t))\to +\infty
\qquad\text{as } t\to\infty
\qquad (\eta\neq \xi).
\]
Hence, for every fixed $z\in X$ and every fixed $\eta\neq \xi$, $I_t(\eta,z)\to 1$. More generally, if $U$ is a neighborhood of $\xi$ in $\partial X$, then the visibility statement yields
\[
\inf_{\eta\in \partial X\setminus U} B_\eta(\gamma(t))\to +\infty.
\]
Since $\eta\mapsto B_\eta(z)$ is continuous on the compact set $\partial X\setminus U$ (Lemma \ref{lemma:compact_uniform_convergence_busemann}), it follows that $I_t(\eta,z)=1$ for all $\eta\in \partial X\setminus U$ once $t$ is large enough.

\medskip
\noindent\textbf{Step 2: pointwise convergence of the contaminated depth.}
\emph{Upper bound.}
Evaluating the infimum at $\eta=\xi$ and using Step 1 yields
\[
\limsup_{t\to\infty} D_t(z)\le (1-\varepsilon)a(z,\xi).
\]
Next, because $a(z,\cdot)$ is continuous and $\partial X\setminus\{\xi\}$ is dense in $\partial X$, we may choose a sequence $\eta_m\in\partial X\setminus\{\xi\}$ such that $a(z,\eta_m)\downarrow D(z;\PP)$. For each fixed $m$, Step 1 gives $I_t(\eta_m,z)=1$ for all sufficiently large $t$, hence
\[
\limsup_{t\to\infty} D_t(z)\le (1-\varepsilon)a(z,\eta_m)+\varepsilon.
\]
Letting $m\to\infty$ gives $\limsup_{t\to\infty} D_t(z)\le (1-\varepsilon)D(z;\PP)+\varepsilon$. Therefore,
\[
\limsup_{t\to\infty} D_t(z)
\le
\min\Bigl\{(1-\varepsilon)a(z,\xi),\ (1-\varepsilon)D(z;\PP)+\varepsilon\Bigr\}.
\]

\emph{Lower bound.}
Fix $\delta>0$. By continuity of $a(z,\cdot)$ at $\xi$, choose a neighborhood $U$ of $\xi$ such that $a(z,\eta)\ge a(z,\xi)-\delta$ for all $\eta\in U$. By Step 1, for all sufficiently large $t$ one has $I_t(\eta,z)=1$ for every $\eta\in \partial X\setminus U$. Hence, for all such $t$,
\begin{align*}
D_t(z)
&\ge
\min\Bigl\{
\inf_{\eta\in U}\bigl((1-\varepsilon)a(z,\eta)+\varepsilon I_t(\eta,z)\bigr),\
\inf_{\eta\in \partial X\setminus U}\bigl((1-\varepsilon)a(z,\eta)+\varepsilon\bigr)
\Bigr\}\\
&\ge
\min\Bigl\{
(1-\varepsilon)(a(z,\xi)-\delta),\
(1-\varepsilon)\inf_{\eta\in \partial X\setminus U}a(z,\eta)+\varepsilon
\Bigr\}\\
&\ge
\min\Bigl\{
(1-\varepsilon)(a(z,\xi)-\delta),\
(1-\varepsilon)D(z;\PP)+\varepsilon
\Bigr\}.
\end{align*}
Taking $\liminf_{t\to\infty}$ and then letting $\delta\downarrow 0$, we obtain
\[
\liminf_{t\to\infty} D_t(z)
\ge
\min\Bigl\{(1-\varepsilon)a(z,\xi),\ (1-\varepsilon)D(z;\PP)+\varepsilon\Bigr\}.
\]

Combining the upper and lower bounds gives
\[
D_t(z)\to D_\infty(z):=
\min\Bigl\{(1-\varepsilon)\PP(H^+_{\xi,z}),\ (1-\varepsilon)D(z;\PP)+\varepsilon\Bigr\}.
\]
This proves part \textup{(1)}.

\medskip
\noindent
\textbf{Step 3: description of the limiting depth regions.}
It follows, for $\alpha \in (0, 1 -\varepsilon]$, the $\alpha$-depth region of $D_\infty(z)$ takes the following form, 
\begin{align*}
    \DD^\alpha_\infty &:= \{z \in X: D_\infty(z; \PP) \ge \alpha\} \\
    &= \{z \in X: (1-\varepsilon)\PP(H_{\xi,z}^+) \ge \alpha\} \cap \{z \in X: (1-\varepsilon)D(z;\PP) + \varepsilon \ge \alpha\} \\
    &= \{z \in X: \PP(H_{\xi,z}^+) \ge \alpha / (1-\varepsilon)\} \cap \DD^{(\alpha - \varepsilon)/(1-\varepsilon)}(\PP) \\
    &= H_\xi(t_\xi(\alpha/(1 -\varepsilon))) \cap \DD^{(\alpha - \varepsilon)/(1-\varepsilon)}(\PP).
\end{align*}
This proves part (2).

\medskip
\noindent
\textbf{Step 4: convergence of contaminated Busemann medians.}
Assume now that $D_\infty$ has a unique maximizer $\mu_\infty$ and that $D_\infty(\mu_\infty)>\varepsilon$. Choose $\rho>0$ such that
\[
D_\infty(\mu_\infty)\ge \varepsilon+3\rho.
\]
Since $D(z;\PP)\to 0$ as $d(o,z)\to\infty$ by 
Proposition \ref{prop:no_infty},
there exists $R>0$ such that
\[
d(o,z)\ge R
\implies
(1-\varepsilon)D(z;\PP)+\varepsilon\le \varepsilon+\rho.
\]
Because $D_t(z)\le (1-\varepsilon)D(z;\PP)+\varepsilon$ for all $z\in X,\ t\ge 0$, we obtain
\[
d(o,z)\ge R \implies D_t(z)\le \varepsilon+\rho
\qquad\text{for all } t\ge 0.
\]
On the other hand, by part \textup{(1)},
\[
D_t(\mu_\infty)\to D_\infty(\mu_\infty)\ge \varepsilon+3\rho,
\]
so for all sufficiently large $t$, $D_t(\mu_\infty)\ge \varepsilon+2\rho$. Hence every maximizer $\mu_t\in \mu_*(\PP_t)$ must belong to the compact ball $\overline B(o,R)$ once $t$ is large enough.

Now let $t_n\to\infty$, and choose $\mu_n\in \mu_*(\PP_{t_n})$. After passing to a subsequence, we may assume $\mu_n\to \mu\in \overline B(o,R)$. We claim that
\begin{equation}\label{eq:limsup-moving}
\limsup_{n\to\infty} D_{t_n}(\mu_n)\le D_\infty(\mu).
\end{equation}
To prove this, first note that $B_\xi(\mu_n)$ is bounded (because $\mu_n\in \overline B(o,R)$), so Step 1 gives $I_{t_n}(\xi,\mu_n)=0$ for all large $n$. Hence, by continuity of $a$,
\[
\limsup_{n\to\infty} D_{t_n}(\mu_n)
\le
\lim_{n\to\infty}(1-\varepsilon)a(\mu_n,\xi)
=
(1-\varepsilon)a(\mu,\xi).
\]
Next choose a sequence $\eta_m\in\partial X\setminus\{\xi\}$ such that $a(\mu,\eta_m)\downarrow D(\mu;\PP)$. For each fixed $m$, Step 1 gives $I_{t_n}(\eta_m,\mu_n)=1$ for all large $n$ (again because $\mu_n$ remains in the fixed compact ball $\overline B(o,R)$). Therefore,
\[
\limsup_{n\to\infty} D_{t_n}(\mu_n)
\le
\lim_{n\to\infty}\bigl((1-\varepsilon)a(\mu_n,\eta_m)+\varepsilon\bigr)
=
(1-\varepsilon)a(\mu,\eta_m)+\varepsilon.
\]
Letting $m\to\infty$ gives
\[
\limsup_{n\to\infty} D_{t_n}(\mu_n)\le (1-\varepsilon)D(\mu;\PP)+\varepsilon.
\]
Combining the two bounds yields \eqref{eq:limsup-moving}.

Since $\mu_n$ maximizes $D_{t_n}$, we have $D_{t_n}(\mu_\infty)\le D_{t_n}(\mu_n)$ for every $n$. Passing to the limit superior and using part \textup{(1)} together with \eqref{eq:limsup-moving}, we obtain
\[
D_\infty(\mu_\infty)
=
\lim_{n\to\infty} D_{t_n}(\mu_\infty)
\le
\limsup_{n\to\infty} D_{t_n}(\mu_n)
\le
D_\infty(\mu).
\]
Thus $\mu$ is a maximizer of $D_\infty$. By uniqueness, $\mu=\mu_\infty$.

We have shown that every subsequential limit of $\mu_t$ equals $\mu_\infty$, thus  $\mu_t\to \mu_\infty$ in $X$. This proves part (3).

\medskip
\noindent
\textbf{Step 5: the Fr\'echet mean escapes to the boundary point $\xi$.}

Assume now that $\PP$ has finite second moment, and let
\[
F_t(z):=(1-\varepsilon)\int_X d^2(z,x)\,\PP(\dd x)+\varepsilon\, d^2(z,\gamma(t)).
\]
Let $\mu_t^F$ be any minimizer of $F_t$. Fix $M>0$, and set $y_M:=\gamma(M+1)$. Since $\PP$ has finite second moment and $y_M$ is fixed,
\[
C_M:=(1-\varepsilon)\int_X d^2(y_M,x)\,\PP(\dd x)<\infty.
\]
For any $z\in X$ with $B_\xi(z)\ge -M$, the function $u\mapsto d(z,\gamma(u))-u$ is decreasing in $u$ and converges to $B_\xi(z)$. Hence, for every $t\ge 0$,
\[
d(z,\gamma(t))-t\ge B_\xi(z)\ge -M,
\]
that is, $d(z,\gamma(t))\ge t-M$. Therefore $F_t(z)\ge \varepsilon (t-M)^2$ whenever $B_\xi(z)\ge -M$. On the other hand,
\[
F_t(y_M)=C_M+\varepsilon\, d^2(y_M,\gamma(t))
=
C_M+\varepsilon (t-M-1)^2.
\]
Thus
\[
F_t(z)-F_t(y_M)\ge \varepsilon\bigl((t-M)^2-(t-M-1)^2\bigr)-C_M
=
\varepsilon\bigl(2(t-M)-1\bigr)-C_M.
\]
The right-hand side tends to $+\infty$ as $t\to\infty$. Hence there exists $T_M$ such that for all $t\ge T_M$,
\[
B_\xi(z)\ge -M
\implies
F_t(z)>F_t(y_M).
\]
Since $\mu_t^F$ minimizes $F_t$, we conclude that for all $t\ge T_M$, $B_\xi(\mu_t^F)<-M$. Because $M>0$ was arbitrary, $B_\xi(\mu_t^F)\to -\infty$. The horoballs $\{z:B_\xi(z)\le -M\}$ form a neighborhood basis of $\xi$ in the visual compactification $X\cup\partial X$. Therefore $\mu_t^F\to \xi$ in $X\cup\partial X$. This proves part (4).

\end{proof}

\subsection{Proof of Theorem \ref{thm:breakdown}}

\begin{proof}
We first prove the chain of inequalities
\[
\varepsilon_{\mathrm{bd}}^\xi(\PP)\ge \varepsilon_{\mathrm{bd}}^\partial(\PP)\ge \varepsilon_{\mathrm{bd}}(\PP).
\]

By definition, $\varepsilon_{\mathrm{bd}}^\partial(\PP)=\inf_{\zeta\in\partial X}\varepsilon_{\mathrm{bd}}^\zeta(\PP)$, hence for every fixed $\xi\in\partial X$, $\varepsilon_{\mathrm{bd}}^\xi(\PP)\ge \varepsilon_{\mathrm{bd}}^\partial(\PP)$. Next we show $\varepsilon_{\mathrm{bd}}^\partial(\PP)\ge \varepsilon_{\mathrm{bd}}(\PP)$. Suppose $\mu_n\to \xi\in\partial X$ in the visual compactification $X\cup\partial X$. Since $\mu_*(\PP)$ is compact 
(Corollary \ref{thm:busemann_exist}),
we claim that $d(\mu_n,\mu_*(\PP))\to\infty$. Indeed, if not, then there exist $R<\infty$ and a subsequence $(\mu_{n_k})$ such that $d(\mu_{n_k},\mu_*(\PP))\le R$ for all $k$. Hence
\[
\mu_{n_k}\in \overline{B}(\mu_*(\PP),R):=\{x\in X:d(x,\mu_*(\PP))\le R\}.
\]
Since $\mu_*(\PP)$ is compact and $X$ is proper, the closed $R$-neighborhood
$\overline{B}(\mu_*(\PP),R)$ is compact. Therefore $(\mu_{n_k})$ has a further subsequence converging to a point of $X$, contradicting $\mu_n\to \xi\in\partial X$. This proves the claim. Thus every boundary breakdown sequence is also a coarse breakdown sequence, and so $\varepsilon_{\mathrm{bd}}^\partial(\PP)\ge \varepsilon_{\mathrm{bd}}(\PP)$.

It remains to prove the lower bound
\[
\varepsilon_{\mathrm{bd}}(\PP)\ge \frac{D_*(\PP)}{1+D_*(\PP)}.
\]
Fix $0\le \varepsilon<D_*(\PP)/(1+D_*(\PP))$, and set $\delta:=(1-\varepsilon)D_*(\PP)-\varepsilon>0$. Choose $\mu\in \mu_*(\PP)$ so that $D(\mu;\PP)=D_*(\PP)$. Let $\QQ$ be any Borel probability measure on $X$, and define
\[
\PP_\varepsilon:=(1-\varepsilon)\PP+\varepsilon \QQ.
\]
For every $\eta\in\partial X$,
\[
\PP_\varepsilon(H_{\eta,\mu}^+)
=
(1-\varepsilon)\PP(H_{\eta,\mu}^+)+\varepsilon \QQ(H_{\eta,\mu}^+)
\ge (1-\varepsilon)\PP(H_{\eta,\mu}^+).
\]
Taking the infimum over $\eta\in\partial X$ yields $D(\mu;\PP_\varepsilon)\ge (1-\varepsilon)D_*(\PP)$. On the other hand, by the contamination bound, $D(z;\PP_\varepsilon)\le D(z;\PP)+\varepsilon$ for all $z\in X$. Since $D(z;\PP)\to 0$ as $z\to\infty$, there exists a compact set $K\subset X$ such that $D(z;\PP)<\delta$  for all $z\notin K$. Therefore, for every $z\notin K$,
\[
D(z;\PP_\varepsilon)
\le D(z;\PP)+\varepsilon
<\delta+\varepsilon
=(1-\varepsilon)D_*(\PP)
\le D(\mu;\PP_\varepsilon).
\]
Hence no maximizer of $D(\cdot;\PP_\varepsilon)$ can lie outside $K$, that is, $\mu_*(\PP_\varepsilon)\subset K$ for every Borel probability measure $\QQ$. Now let $\{\QQ^{(n)}\}_n$ be any sequence of contaminating measures, and write
\[
\PP_\varepsilon^{(n)}:=(1-\varepsilon)\PP+\varepsilon\QQ^{(n)}.
\]
If $\mu_n\in \mu_*(\PP_\varepsilon^{(n)})$, then $\mu_n\in K$ for all $n$. Since both $K$ and
$\mu_*(\PP)$ are compact,
\[
\sup_{x\in K} d(x,\mu_*(\PP))<\infty.
\]
Therefore
\[
d(\mu_n,\mu_*(\PP))
\le \sup_{x\in K} d(x,\mu_*(\PP))
<\infty
\qquad\text{for all }n,
\]
so $d(\mu_n,\mu_*(\PP))\not\to\infty$. Thus $\varepsilon$ is strictly below $\varepsilon_{\mathrm{bd}}(\PP)$. Since every
\[
\varepsilon<\frac{D_*(\PP)}{1+D_*(\PP)}
\]
has this property, we conclude that
\[
\varepsilon_{\mathrm{bd}}(\PP)\ge \frac{D_*(\PP)}{1+D_*(\PP)}.
\]
Combining this with the first part gives
\[
\varepsilon_{\mathrm{bd}}^\xi(\PP)\ge
\varepsilon_{\mathrm{bd}}^\partial(\PP)\ge
\varepsilon_{\mathrm{bd}}(\PP)\ge
\frac{D_*(\PP)}{1+D_*(\PP)}.
\]
Finally, by 
Theorem~\ref{thm:depth_lbound}, 
$D_*(\PP)\ge 1/(d+1)$, and therefore
\[
\varepsilon_{\mathrm{bd}}(\PP)\ge
\frac{1/(d+1)}{1+1/(d+1)}
=
\frac{1}{d+2}.
\]
\end{proof}


\section{Proofs in Section \ref{sec:computation}}

\subsection{Proof of Lemma \ref{lemma:VC}}

\begin{lemma}\label{lem:pullback-vc}
Let $\pi:Y\to X$ be a surjective map, and let $\mathcal F\subset \mathcal P(X)$.\footnote{Here $\mathcal{P}(X)$ denotes the powerset of $X$.}
If
\[
\pi^{-1}\mathcal F:=\{\pi^{-1}(A):A\in\mathcal F\}
\]
is a VC class on $Y$, then $\mathcal F$ is a VC class on $X$.
\end{lemma}

\begin{proof}
Suppose $x_1,\dots,x_m\in X$ are shattered by $\mathcal F$. Choose lifts
$y_i\in \pi^{-1}(x_i)$ for $1\le i\le m$. For every $I\subset \{1,\dots,m\}$,
pick $A_I\in\mathcal F$ such that $x_i\in A_I \iff i\in I$. Then
\[
y_i\in \pi^{-1}(A_I)\iff x_i\in A_I \iff i\in I,
\]
so $y_1,\dots,y_m$ are shattered by $\pi^{-1}\mathcal F$. Hence
\[
\operatorname{VCdim}(\mathcal F)\le \operatorname{VCdim}(\pi^{-1}\mathcal F).
\]
\end{proof}

\begin{lemma}\label{lemma:vc-horoballs-symmetric}
Let $X$ be a symmetric space of noncompact type. Then
\[
\mathcal H
:=
\bigl\{
\{x\in X:\, B_\xi(x)\ge t\}
:\ \xi\in \partial_\infty X,\ t\in\R
\bigr\}
\]
is a VC class.
\end{lemma}

\begin{proof}
Write $X=G /K$, where $G$ is a connected linear real reductive algebraic group and $K$ is a maximal compact subgroup, and let $\pi:G\to X$ be the quotient map. Define the pulled-back family
\[
\widetilde{\mathcal H}
:=
\{\pi^{-1}(H):\, H\in \mathcal H\}.
\]
By Lemma \ref{lem:pullback-vc}, it is enough to prove that
$\widetilde{\mathcal H}$ is a VC class on $G$. By the representation-theoretic description of Busemann functions \cite[Theorem 3.7]{solan2025geometric}, there exist finitely many real fundamental representations $\rho_1,\dots,\rho_r$ of $G$ and corresponding highest-weight vectors $v_1,\dots,v_r$ such that every function of the form $B_\xi-t$ can be written as
\[
B_\xi(\pi(g))-t = \sum_{i=1}^r c_i \log \|\rho_i(gg_0)v_i\| + C
\]
for suitable parameters $g_0\in G$, $c_1,\dots,c_r\ge 0$, and $C\in\R$. Therefore every set in $\widetilde{\mathcal H}$ has the form
\[
\left\{
g\in G:\,
\sum_{i=1}^r c_i \log \|\rho_i(gg_0)v_i\|
\ge s
\right\}
\]
for some $(g_0,c,s)\in G\times \R_{\ge 0}^r\times \R$. Choose matrix coordinates on $G\subset \mathrm{Mat}_N(\R)$ and on each
representation space of $\rho_i$. Consider the set
\[
\Sigma
:=
\left\{
(g,g_0,c,s)\in G\times G\times \R_{\ge 0}^r\times \R
:\,
\sum_{i=1}^r c_i \log \|\rho_i(gg_0)v_i\|\ge s
\right\}.
\]
This set is definable in the o-minimal structure $\R_{\mathrm{an},\exp}$ \cite{van1994real}:
the group $G$ is algebraic, multiplication on $G$ is algebraic, each $\rho_i$ is an
algebraic representation, the norm is semialgebraic, and $\log$ is definable in
$\R_{\mathrm{an},\exp}$.

Hence $\widetilde{\mathcal H}$ is a uniformly definable family in an o-minimal
structure. It follows that $\widetilde{\mathcal H}$ has finite VC density, and in
particular finite VC dimension \cite{johnson2010compression}. Thus $\widetilde{\mathcal H}$ is a VC class. Applying Lemma \ref{lem:pullback-vc}, we conclude that $\mathcal H$ is a VC class.
\end{proof}

\begin{corollary}\label{cor:vc-depth-family-symmetric}
The class
\[
\mathcal A
:=
\bigl\{
\{x\in X:\, B_\xi(x)\ge B_\xi(z)\}
:\ \xi\in \partial_\infty X,\ z\in X
\bigr\}
\]
is also a VC class.
\end{corollary}

\begin{proof}
For every $(\xi,z)$, setting $t:=B_\xi(z)$ gives
\[
\{x\in X:\, B_\xi(x)\ge B_\xi(z)\}
=
\{x\in X:\, B_\xi(x)\ge t\}\in \mathcal H.
\]
Hence $\mathcal A\subset \mathcal H$, so $\mathcal A$ is VC.
\end{proof}

\begin{proposition}\label{prop:VC-hyperbolic-horoballs}
Let $X=\mathbb H^d$ be real hyperbolic $n$-space, and let
\[
\mathcal H:=\bigl\{\{x\in X:\ B_\xi(x)\ge t\}:\ \xi\in \partial_\infty X,\ t\in\R\bigr\},
\]
where $B_\xi$ is the Busemann function associated with $\xi$.
Then $\mathcal H$ is a VC class. In fact,
\[
\operatorname{VCdim}(\mathcal H)\le d+2.
\]
\end{proposition}

\begin{proof}
We work in the Poincar\'e ball model
\[
\mathbb H^d \cong \mathbb B^d:=\{x\in \R^d:\ \|x\|<1\},
\]
with ideal boundary $\partial_\infty \mathbb H^d=S^{d-1}$ Fix $\xi\in S^{d-1}$, the Busemann function is given by the standard formula
\[
B_\xi(x)=\log\frac{\|x-\xi\|^2}{1-\|x\|^2},
\qquad x\in \mathbb B^d.
\]
Hence, for any $t\in\R$,
\[
B_\xi(x)\ge t
\iff
\|x-\xi\|^2\ge e^t(1-\|x\|^2).
\]
Expanding the right-hand side gives $(1+e^t)\|x\|^2-2x\cdot \xi +(1-e^t)\ge 0$. Completing the square, we obtain
\[
\left\|x-\frac{1}{1+e^t}\xi\right\|^2 \ge \left(\frac{e^t}{1+e^t}\right)^2.
\]
Therefore
\[
\{x\in \mathbb H^d:\ B_\xi(x)\ge t\}
=
\mathbb B^d\cap
\left\{
x\in\R^d:\ 
\left\|x-\frac{1}{1+e^t}\xi\right\|
\ge
\frac{e^t}{1+e^t}
\right\}.
\]
Thus every set in $\mathcal H$ is the complement (inside $\mathbb B^d$) of an open Euclidean ball
whose center is $c_{\xi,t}:=1/(1+e^t)\xi$ and whose radius is $r_t:=e^t / (1+e^t)$. Since $|c_{\xi,t}|+r_t=1$, this Euclidean ball is tangent to the boundary sphere $S^{d-1}$ at the point $\xi$. It follows that $\mathcal H$ is a subfamily of the class of complements of Euclidean balls in $\R^d$. Since taking complements does not change VC dimension, it is enough to show that the
class of Euclidean balls in $\R^d$ is VC.

Let
\[
\mathcal B:=\bigl\{\overline{B}(c,r):\ c\in\R^d,\ r\ge 0\bigr\}.
\]
Define the lifting map
\[
\Phi:\R^d\to\R^{d+1},
\qquad
\Phi(x):=(x,\|x\|^2).
\]
Then
\[
x\in \overline{B}(c,r)
\iff
\|x-c\|^2\le r^2
\iff
\|x\|^2-2c\cdot x+\|c\|^2-r^2\le 0.
\]
Equivalently,
\[
x\in \overline{B}(c,r)
\iff
\langle (-2c,1),\Phi(x)\rangle \le r^2-\|c\|^2.
\]
Hence $\mathcal B$ is the pullback under $\Phi$ of the class of affine halfspaces in $\R^{d+1}$. Since affine halfspaces in $\R^{d+1}$ have VC dimension $d+2$,
the class $\mathcal B$ is VC with
\[
\operatorname{VCdim}(\mathcal B)\le d+2.
\]
Therefore the class of complements of Euclidean balls is also VC with the same bound, and so is its subfamily $\mathcal H$. 
\end{proof}

\begin{corollary}\label{cor:VC-depth-family-hyperbolic}
The class
\[
\mathcal A
:=
\bigl\{
\{x\in \mathbb H^d:\ B_\xi(x)\ge B_\xi(z)\}
:\ \xi\in \partial_\infty \mathbb H^d,\ z\in \mathbb H^d
\bigr\}
\]
is also a VC class.
\end{corollary}

\begin{proof}
For each $\xi\in \partial_\infty \mathbb H^d$ and $z\in \mathbb H^d$, put
\[
t=B_\xi(z).
\]
Then
\[
\{x\in \mathbb H^d:\ B_\xi(x)\ge B_\xi(z)\}
=
\{x\in \mathbb H^d:\ B_\xi(x)\ge t\}\in \mathcal H.
\]
Hence $\mathcal A\subset \mathcal H$, and therefore $\mathcal A$ is VC.
\end{proof}

\subsection{Proof of Theorem \ref{thm:sample_converg}}
\label{proof:sample_converg}

\begin{proof}
    By Lemma \ref{lemma:inf_sup}, we have, given any $z \in X$,
    \[
    \left|\inf_{\xi}\PP_n(H^+_{\xi, z}) - \inf_{\xi}\PP(H^+_{\xi, z}) \right| \le \sup_{\xi} \left|\PP_n(H^+_{\xi, z}) - \PP(H^+_{\xi, z})\right|.
    \]
    It follows that, we have,
    \begin{align*}
        \sup_z \left| D(z; \PP_n) - D(z; \PP) \right| 
        \le \sup_{\xi,z} \left|\PP_n(H^+_{\xi, z}) - \PP(H^+_{\xi, z})\right|.
    \end{align*}
    Thus, it is sufficient to show 
    \[
    \sup_{\xi, z} \left|\PP_n(H^+_{\xi, z}) - \PP(H^+_{\xi, z})\right| \to 0, \quad a.s., 
    \]
    where the supremum is over $\{(\xi, z): \xi \in \partial X, z \in X\}$. 
    Denote $\mathcal{A}= \{\{x \in X: B_\xi(x) \ge B_\xi(z)\}:z \in X, \xi \in \partial X\}$, then we can write
    \[
    \sup_{\xi, z} \left|\PP_n(H^+_{\xi, z}) - \PP(H^+_{\xi, z})\right| \le \sup_{A \in \mathcal A} \left|\PP_n(A) - \PP(A)\right|.
    \]
    It's then sufficient to show $\mathcal A$ is Glivenko-Cantelli (GC) class. Under various regularity conditions, a class of sets is GC class if and only if it's Vapnik–Chervonenkis (VC) class, that is, of finite VC dimension. The result then follows from 
    Lemma \ref{lemma:VC}.

\end{proof}


    
    
\subsection{Proof of Theorem \ref{thm:sample_converg_2}}
\label{proof:sample_converg_2}
\begin{proof}
    First we note, since $X$ is a Hadamard manifold (separable metric space), Varadarajan's theorem \cite{varadarajan1958convergence} yields $\PP_n\Rightarrow \PP$ a.s.
    
    \medskip 
    \noindent
    \textbf{Step 1: point-wise convergence.}
    We claim that under 
    Assumption \ref{assump:no_atom},
    if $\PP_n\Rightarrow \PP$ and $(\xi_n,z_n)\to (\xi,z)$ in $\partial X\times X$, then
    \[
    \PP_n\bigl(H^+_{\xi_n,z_n}\bigr)\to\PP\bigl(H^+_{\xi,z}\bigr).
    \]
    Indeed, define $h_n(x):=B_{\xi_n}(x)-B_{\xi_n}(z_n)$ and $h(x):=B_\xi(x)-B_\xi(z)$. Fix $\varepsilon,\eta>0$ and choose $R>0$ so that
    \[
    \PP\bigl(X\setminus \overline{B}(o,R)\bigr)<\eta
    \qquad\text{and}\qquad
    \PP\bigl(\partial \overline{B}(o,R)\bigr)=0.
    \]
    Since $z_n\to z$ and $B_{\xi_n}\to B_\xi$ uniformly on compact sets (by Lemma \ref{lemma:compact_uniform_convergence_busemann}), we have
    \[
    \sup_{x\in \overline{B}(o,R)}|h_n(x)-h(x)|<\varepsilon
    \]
    for all sufficiently large $n$. Hence, for all large $n$,
    \[
    \{h>\varepsilon\}\cap \overline{B}(o,R)\subset H^+_{\xi_n,z_n}\cap \overline{B}(o,R)
    \subset \{h\ge -\varepsilon\}\cap \overline{B}(o,R).
    \]
    Therefore
    \[
    \PP_n(\{h>\varepsilon\})-\PP_n\bigl(X\setminus \overline{B}(o,R)\bigr)
    \le
    \PP_n(H^+_{\xi_n,z_n})
    \le
    \PP_n(\{h\ge -\varepsilon\})+\PP_n\bigl(X\setminus \overline{B}(o,R)\bigr).
    \]
    Since $\PP_n\Rightarrow \PP$ and $\overline{B}(o,R)$ is a $\PP$-continuity set,
    \[
    \PP_n\bigl(X\setminus \overline{B}(o,R)\bigr)\to \PP\bigl(X\setminus \overline{B}(o,R)\bigr)<\eta.
    \]
    By the Portmanteau theorem,
    \begin{align*}
        \PP(\{h>\varepsilon\}) - \eta \le \liminf_{n\to\infty}\PP_n(H^+_{\xi_n,z_n}) \le \limsup_{n\to\infty}\PP_n(H^+_{\xi_n,z_n}) \le 
        \PP(\{h\ge -\varepsilon\}) + \eta.
    \end{align*}
    Now let $\eta\downarrow 0$, then $\varepsilon\downarrow 0$. Since under 
    Assumption \ref{assump:no_atom}, 
    \[
    \PP(\{h=0\})=\PP\bigl(\{x:\, B_\xi(x)=B_\xi(z)\}\bigr)=0,
    \]
    we obtain $\PP_n(H^+_{\xi_n,z_n})\to \PP(H^+_{\xi,z})$, which proves the claim.
    
    \medskip
    \noindent
    \textbf{Step 2: uniform convergence on compact sets.}
    Fix a compact set $K\subset X$. We claim that
    \[
    \sup_{(\xi,z)\in \partial X\times K}
    \bigl|\PP_n(H^+_{\xi,z})-\PP(H^+_{\xi,z})\bigr| \to 0.
    \]
    Suppose not. Then there exist $\varepsilon>0$, integers $n_k\to\infty$, and points $\{(\xi_k,z_k)\}\in \partial X\times K$ such that
    \[
    \bigl|\PP_{n_k}(H^+_{\xi_k,z_k})-\PP(H^+_{\xi_k,z_k})\bigr|\ge \varepsilon
    \qquad\text{for all }k.
    \]
    Since $\partial X\times K$ is compact, after passing to a subsequence we may assume
    \[
    (\xi_k,z_k)\to (\xi,z)\in \partial X\times K.
    \]
    By Step 1, we have $\PP_{n_k}(H^+_{\xi_k,z_k})\to \PP(H^+_{\xi,z})$ and $\PP(H^+_{\xi_k,z_k})\to \PP(H^+_{\xi,z})$. Hence
    \[
    \bigl|\PP_{n_k}(H^+_{\xi_k,z_k})-\PP(H^+_{\xi_k,z_k})\bigr|\to 0,
    \]
    and contradiction. Thus
    \[
    \sup_{(\xi,z)\in \partial X\times K}
    \bigl|\PP_n(H^+_{\xi,z})-\PP(H^+_{\xi,z})\bigr|\to 0.
    \]
    It follows immediately that (by Lemma \ref{lemma:inf_sup})
    \[
    \sup_{z\in K}|D(z;\PP_n)-D(z;\PP)|
    \le
    \sup_{(\xi,z)\in \partial X\times K}
    \bigl|\PP_n(H^+_{\xi,z})-\PP(H^+_{\xi,z})\bigr|
    \to 0.
    \]
    So the depth converges uniformly on compact subsets of $X$.
    
    \medskip
    \noindent
    \textbf{Step 3: tail estimate.} Let $z\in X$ and $r:=d(o,z)$, consider a geodesic $\gamma:\R\to X$ with $\gamma(0)=o$ and $\gamma(r)=z$. Let $\eta=\gamma(-\infty)\in \partial X$ be the endpoint opposite to $z$. It follows that $B_\eta(z)=r$. Since $B_\eta(o)=0$ and $B_\eta$ is $1$-Lipschitz 
    (Proposition \ref{prop:cont}), 
    $B_\eta(x)\le d(o,x)$ for any $x \in X$. Hence
    \[
    H^+_{\eta,z}=\{x:\, B_\eta(x)\ge B_\eta(z)\} \subset \{x:\, d(o,x)\ge r\} = X\setminus B(o,r).
    \]
    Therefore, for every Borel probability measure $\PP$ on $X$ and every $R>0$,
    \begin{equation}\label{eqn:tail}
        \sup_{z\notin \overline{B}(o,R)} D(z;\PP) \le \PP\bigl(X\setminus \overline{B}(o,R)\bigr).
    \end{equation}
    Now fix $\delta>0$ and choose $R>0$ such that
    \[
    \PP\bigl(X\setminus \overline{B}(o,R)\bigr)<\delta
    \qquad\text{and}\qquad
    \PP\bigl(\partial \overline{B}(o,R)\bigr)=0.
    \]
    Since $\PP_n\Rightarrow \PP$, $\PP_n\bigl(X\setminus \overline{B}(o,R)\bigr) \to \PP\bigl(X\setminus \overline{B}(o,R)\bigr)$. Now we have
    \begin{align*}
    &\sup_{z\in X}|D(z;\PP_n)-D(z;\PP)| \\
    \le &\sup_{z\in \overline{B}(o,R)}|D(z;\PP_n)-D(z;\PP)| + \sup_{z\notin \overline{B}(o,R)}D(z;\PP_n) + \sup_{z\notin \overline{B}(o,R)}D(z;\PP).
    \end{align*}
    Using the compact-uniform convergence on $\overline{B}(o,R)$, the tail bound \eqref{eqn:tail}, and taking $\limsup$ as $n\to\infty$, we obtain
    \[
    \limsup_{n\to\infty}\sup_{z\in X}|D(z;\PP_n)-D(z;\PP)|
    \le 2\delta.
    \]
    Since $\delta>0$ was arbitrary, we have $\sup_{z\in X}|D(z;\PP_n)-D(z;\PP)|\to 0$ as required.
\end{proof}

\subsection{Proof of Theorem \ref{thm:sample_converg_depth_region}}

\begin{proof}
Parts (1) and (2) are the manifold version of the contour-convergence theorem of \cite{zuo2000structural}. The Euclidean proof uses only upper semicontinuity of the depth, vanishing at infinity, and uniform convergence of the sample depth; these ingredients are available here by 
Proposition \ref{prop:usc}, Proposition \ref{prop:no_infty}, and \eqref{eqn:u_conver}.

For part (3), let $\theta$ be the unique population median. Fix $r>0$. By uniqueness of the maximizer, upper semicontinuity of $D(\cdot;\PP)$, and vanishing at infinity, there exists $\eta_r>0$ such that
\[
\sup_{d(z,\theta)\ge r} D(z;\PP)\le D(\theta;\PP)-2\eta_r.
\]
On the almost sure event where 
\eqref{eqn:u_conver}
holds, for all sufficiently large $n$ we have
\[
\sup_{z\in X}\bigl|D(z;\PP_n)-D(z;\PP)\bigr|\le \eta_r.
\]
Therefore $D(\theta;\PP_n)\ge D(\theta;\PP)-\eta_r$, while every $z$ with $d(z,\theta)\ge r$ satisfies $D(z;\PP_n)\le D(\theta;\PP)-\eta_r$. Hence no maximizer of $D(\cdot;\PP_n)$ can lie outside $B(\theta,r)$, so $\theta_n\in B(\theta,r)$ for all sufficiently large $n$. Since $r>0$ was arbitrary, $\theta_n\to \theta$ almost surely.
\end{proof}

\end{document}